\documentclass[12pt,a4paper,reqno]{amsart}
\usepackage{amsmath}
\usepackage{amssymb}
\usepackage{enumitem}
\usepackage{geometry}
\usepackage{mathtools,nccmath}
\usepackage{etoolbox} 
\usepackage{xparse} 
\usepackage{amsthm}
\usepackage{tikz}
\usetikzlibrary{shapes}
\usepackage{mathrsfs}
\usepackage{hyperref}
\usetikzlibrary{arrows}
\usepackage[utf8]{inputenc}
\numberwithin{equation}{section}
\usepackage{tikz-cd}
\usepackage[T1]{fontenc}
\usepackage{txfonts}

\usetikzlibrary{calc}
\usepackage{comment}
\usepackage{xcolor}

\usepackage[%
shortalphabetic, 
abbrev, 
backrefs, 
]{amsrefs}
\hypersetup{%
	linktoc=all, 
	colorlinks=true,
	citecolor=magenta,
}

\NewDocumentCommand\Crefnameitem { m m m O{\textup} O{(\roman*)}} {%
	\Crefname{#1enumi}{#2}{#3} 
	\AtBeginEnvironment{#1}{%
		\crefalias{enumi}{#1enumi}%
		\setlist[enumerate,1]{
			label={#4{#5}.},
			ref={#5}
		}%
	}
}

\usetikzlibrary{%
	cd, 
	shapes, 
	arrows, 
}
\NewDocumentCommand\placeholder{}{\:\cdot\:} 
\NewDocumentCommand\NewPairedDelimiterS{mmm}{%
	\DeclarePairedDelimiterX{#1}[1]{#2}{#3}%
	{\ifblank{##1}{\placeholder}{##1}}%
}
\NewDocumentCommand\NewPairedDelimiterSS{mmmO{,}}{%
	\DeclarePairedDelimiterX{#1}[2]{#2}{#3}%
	{\ifblank{##1}{\placeholder}{##1}%
		#4%
		\ifblank{##2}{\placeholder}{##2}}%
}
\NewPairedDelimiterS\abs\lvert\rvert 
\NewPairedDelimiterS\norm\lVert\rVert 
\NewPairedDelimiterS\lrangle\langle\rangle 
\NewPairedDelimiterS\paren\lparen\rparen 
\NewPairedDelimiterS\lrbrack\lbrack\rbrack 
\NewPairedDelimiterS\lrbrace\lbrace\rbrace 
\NewPairedDelimiterS\ceil\lceil\rceil 
\NewPairedDelimiterS\floor\lfloor\rfloor 
\NewPairedDelimiterS\bra\langle\vert 
\NewPairedDelimiterS\ket\vert\rangle 
\NewPairedDelimiterS\normord{\mathopen{:}}{\mathclose{:}} 
\NewPairedDelimiterSS\braket\langle\rangle[%
\,\delimsize\vert\,\mathopen{}%
] 
\NewPairedDelimiterSS\pairing\langle\rangle
\NewPairedDelimiterSS\inner\lparen\rparen
\NewPairedDelimiterSS\Liebracket\lbrack\rbrack

\newtheorem{thm}{Theorem}[section]
\newtheorem*{theorem*}{Theorem}

\newtheorem{prop}[thm]{Proposition}
\newtheorem{lm}[thm]{Lemma}

\newtheorem{coro}[thm]{Corollary}

\newcommand\numberthis{\addtocounter{equation}{1}\tag{\theequation}}

\theoremstyle{definition}
\newtheorem{example}[thm]{Example}
\newtheorem{df}[thm]{Definition}
\newtheorem{remark}[thm]{Remark}

\newtheorem{innercustomgeneric}{\customgenericname}
\providecommand{\customgenericname}{}
\newcommand{\newcustomtheorem}[2]{%
	\newenvironment{#1}[1]
	{%
		\renewcommand\customgenericname{#2}%
		\renewcommand\theinnercustomgeneric{##1}%
		\innercustomgeneric
	}
	{\endinnercustomgeneric}
}
\newcustomtheorem{customthm}{Theorem}

\DeclareMathOperator{\Adm}{Adm}

\newcommand{\bt}{\boxtimes}
\newcommand{\fusion}[3]{{\binom{#3}{#1\;#2}}}

\newcommand{\al}{\alpha}

\newcommand{\1}{{{\bf 1}}}

\newcommand{\Hom}{{\rm Hom}\,}

\newcommand{\End}{{\rm End}\,}
\newcommand{\Res}{{\rm Res}\,}

\newcommand{\Z}{\mathbb{Z}}
\newcommand{\Q}{\mathbb{Q}}
\newcommand{\N}{\mathbb{N}}
\newcommand{\A}{\mathbb{A}}
	\def\Res{{\rm Res}}
	\def\wt{{\rm wt}}

	\newcommand{\la}{\lambda}
	
	\def\C{{\mathbb C}}
	\def\P{{\mathbb P}}
	\def\R{{\mathcal R}}

	\def\Z{{\mathbb Z}}
	
	\def\N{{\mathbb N}}
	
	\def\Y{{\mathcal Y}}
	
	\def\1{{\bf 1}}
	
	\def \End{{\rm End}}
	\def \Hom{{\rm Hom}}

	\def \b{\beta}

	\def \ra {\rightarrow}
	\def \g{\mathfrak{g}}

	\def\Log{\mathrm{Log}}

	\def \ssq{\subseteq}
	\def \C {\mathbb{C}}
	
	\def\Id{\mathrm{Id}}
	\def\op{\oplus}
	
	\def\vac{\mathbf{1}}
	\def\spn{\mathrm{span}}
	\def \vac{\mathbf{1}}
	
	\def\bs{\backslash}
	\def\om{\omega}
	\def\fd{\mathrm{fd}}
	\def\sl{\mathfrak{sl}}

	\def\o{\otimes}
	
	\def\<{\langle}
	\def\>{\rangle}
	
	\def\Om{\Omega}

	\def\ds{\dots}
	\def\res{\mathrm{res}}
	
	\def\L{\mathsf{L}}
	\newcommand\CC{\mathcal C}

	\def\gl{\mathfrak{gl}}
	
	\def\Mod{\mathsf{Mod}}
	\def\ord{\mathrm{ord}}
	
	\def\Arg{\mathrm{Arg}}
	
	\def\ga{\gamma}
	
	\def\sL{\mathsf{L}}
	\def\sN{\mathsf{N}}
	\def\CC{\mathcal{C}}
	
	\def\Adm{\mathsf{Adm}}
	\def\scrU{\mathscr{U}}
	\def\Ord{\mathsf{Ord}}
	\newcommand \WV {\begin{smallmatrix}W\\WV\end{smallmatrix}}

	\def\<{\langle}
	\def\>{\rangle}
	
	\newcommand{\jq}[1]{}
	
	\allowdisplaybreaks
	
	\makeatletter
	\providecommand*{\braiding@longrightarrowfill}{%
		\arrowfill@\relbar\relbar\longrightarrow}
	\providecommand*{\xlongrightarrow}[2][]{%
		\ext@arrow 0359\braiding@longrightarrowfill{#1}{#2}}
	\makeatother
	
\begin{document}
		\title{Restricted conformal blocks and the braiding on Zhu algebra modules}
		
		\author{Zhiguang Cui}
		\address{School of Mathematics, Sichuan university, Chengdu, China,}
		\email{zgcui@scu.edu.cn}
		
		\author{Jianqi Liu}
		\address{Department of Mathematics, New Uzbekistan University, Movarounnahr street 1, Tashkent 100001, Uzbekistan}
		\email{j.liu@newuu.uz}

		\begin{abstract}
		Let $V$ be a vertex operator algebra (VOA), not necessarily rational or $C_2$-cofinite, such that the category $\CC_1(V)$ of $C_1$-cofinite grading-restricted generalized $V$-modules is semisimple and contains the generalized Verma modules associated to finite-dimensional modules over the Zhu algebra $A(V)$. Such VOAs include affine VOAs at admissible and generic levels. We prove that the bottom degree of Huang-Lepowsky's $P(w)$-tensor product module $M^1\bt_{P(w)}M^2$ is isomorphic to the left $A(V)$-module $M^1\odot_{w,0}M^2$ arising from the restricted conformal blocks on $(\P^1,\infty,w,0)$. By shifting the marked points and using the parallel transport of the restricted conformal blocks, we construct an explicit braiding isomorphism on the category of finite-dimensional $A(V)$-modules, which is the restriction of the braiding isomorphism of Huang-Lepowsky-Zhang to the bottom degrees. For the affine VOA $L_k(\mathfrak g)$ at a positive integral level, it is given by the Casimir tensor of $\mathfrak g$ on certain quotients of tensor products of finite-dimensional $\mathfrak g$-modules, and at generic levels it recovers Drinfeld's commutativity isomorphism.
		\end{abstract}
		
		\maketitle
		
		\section{Introduction}

The tensor product of modules over a vertex operator algebra (VOA) $V$ and its braiding isomorphism play an essential role in the representation theory of VOAs and in conformal field theory (CFT). Using the restriction technique of conformal blocks \cite{GLZ25,GLZ26,Liu26}, we obtain an explicit description of the braiding isomorphism on the bottom degrees of $V$-modules, which are modules over the Zhu algebra $A(V)$ \cite{Z96}. As a consequence, for a class of VOAs that are not necessarily rational or $C_2$-cofinite, the category $\Mod_\fd(A(V))$ of finite-dimensional left $A(V)$-modules is a braided tensor category with an explicit braiding isomorphism, and it is braided equivalent to the corresponding category of $V$-modules.

In the VOA tensor product theory, the module $M^1\bt_{P(z)}M^2$ is defined by a universal property, and the braiding isomorphism is given by the commutativity isomorphism and the parallel transport of the $P(z)$-tensor products \cite{HLZ}, which depend on the analytic properties of intertwining operators among infinite-dimensional modules. On the other hand, the fusion rules can be computed through the bottom degrees $M^2(0)$ and $M^3(0)^\ast$, and the bimodule $A(M^1)$ over the Zhu algebra $A(V)$. This is the well-known fusion rules theorem claimed in \cite{FZ92}, which was revised by Li \cite{Li99b}, and a variant of it was proved by the second-named author \cite{Liu23}; see also \cite{H22,H24}. It is natural to expect that the tensor product and the braiding isomorphism can also be described using the bottom degrees, that is, on the category of finite-dimensional $A(V)$-modules.

Recently, Gao, the second-named author, and Zhu introduced (twisted) restricted conformal blocks in \cite{GLZ25,GLZ26}. By restricting the module $(M^3)'$ attached to $\infty$ in the three-point conformal blocks on $(\P^1,\infty,1,0)$ to its bottom degree $M^3(0)^\ast$, the second-named author introduced a left $A(V)$-module
		\[
		M^1\odot M^2=\frac{M^1\o_\C M^2}{\mathcal{L}_{\P^1\bs\{1,0\}}(V)_{<0}.(M^1\o_\C M^2)},	\]
		where $\mathcal{L}_{\P^1\bs\{1,0\}}(V)_{<0}$ is an ideal of the $\infty$-restricted chiral Lie algebra \cite{Liu26,GLZ26}. This quotient space also occurred in a construction of fusion tensor product of modules over $C_2$-cofinite VOAs by Tsuchiya and Wood (see \cite[(2.84)]{TW13}). If $V$ is strongly rational, then $M^1\odot M^2$ is isomorphic to the bottom degree of Huang-Lepowsky's $P(1)$-tensor product module $M^1\bt_{P(1)}M^2$, and $\Mod_\fd(A(V))$ is a tensor category with respect to the fusion tensor product $S\bt T=\Phi^\L(S)\odot \Phi^\L(T)$ \cite[Theorem~6.8]{Liu26}, where $\Phi^\L$ is the generalized Verma module functor \cite{DLM98,DGK25}. In this paper, we give an explicit formula for the braiding on $\Mod_\fd(A(V))$ and extend this tensor product construction to VOAs that need not be rational or $C_2$-cofinite.

\subsection{Main theorems}

Let $\CC_1(V)$ denote the category of $C_1$-cofinite grading-restricted generalized $V$-modules, and write $\Om$ for the bottom-degree functor. Throughout most of the paper, we assume that
\begin{equation}\label{eq:intro-assumption}
    \CC_1(V)\ \text{is semisimple, and}\ \ \Phi^\L(\Mod_\fd(A(V)))\ssq \CC_1(V).
\end{equation}
Then $(\Phi^\L\dashv \Om):\Mod_\fd(A(V))\rightleftarrows \CC_1(V)$ is an equivalence of categories, see Proposition~\ref{prop:condforC1V} and \cite[Theorem~3.2]{AL26}.

For $w\in\C^\times$, the same restriction at $\infty$ gives a left $A(V)$-module $M^1\odot_{w,0}M^2$, with $M^1\odot_{1,0}M^2=M^1\odot M^2$, see Definition~\ref{df:odotw}. Our first main theorem identifies this module with the bottom degree of the $P(w)$-tensor product (see Theorem~\ref{thm:eta-explicit}).

\begin{customthm}{A}
			Assume $\CC_1(V)$ is semisimple and $\Phi^\L(\Mod_\fd(A(V)))\ssq \CC_1(V)$. Let $M^1,M^2\in \CC_1(V)$. Then the canonical map
			\[
			\eta_w: M^1\odot_{w,0}M^2\ra \Om\left(M^1\bt_{P(w)}M^2\right),\quad v^1\odot_w v^2\mapsto \pi_0\left(v^1\bt_{P(w)}v^2\right),
			\]
			is a well-defined isomorphism of left $A(V)$-modules, which is natural in $M^1$ and $M^2$, where $\pi_0$ is the projection onto the bottom degree.
		\end{customthm}

By shifting the marked points and using parallel transport of the restricted conformal blocks, we obtain the following isomorphisms (see Theorem~\ref{thm:Rmatrix}).

\begin{customthm}{B}
			Assume $\CC_1(V)$ is semisimple and $\Phi^\L(\Mod_\fd(A(V)))\ssq \CC_1(V)$. Let $M^1,M^2\in \CC_1(V)$. Then there exist isomorphisms of left $A(V)$-modules which are natural in $M^1,M^2$:
			\[
			\begin{aligned}
			R^{M^2,M^1}_\pm:&\ M^2\odot M^1\ra M^1\odot M^2,\\
			R^{M^2,M^1}_\pm(v^2\odot v^1)
			&=e^{\mp\pi iL(0)}.\left(e^{\pm\pi iL(0)}v^1\odot e^{\pm\pi iL(0)}v^2\right),
			\end{aligned}
			\]
			where the outer factor $e^{\mp\pi iL(0)}$ acts on $M^1\odot M^2$ through $e^{\mp\pi i[\om]}$. Moreover, \[\left(R_-^{M^1,M^2}\right)^{-1}=R_+^{M^2,M^1}\quad  \mathrm{and}\quad \left(R_+^{M^1,M^2}\right)^{-1}=R_-^{M^2,M^1}.\]
		\end{customthm}

Let $\R_{M^1,M^2}=R_-^{M^1,M^2}$. Using Theorems A and B, we prove that $\R$ is the restriction of the braiding isomorphism $c_{M^1,M^2}$ of Huang-Lepowsky-Zhang \cite{HLZ,H25} to the bottom degrees, namely, $\eta_1\circ\R_{M^1,M^2}=\Om(c_{M^1,M^2})\circ\eta_1$, see Theorem~\ref{thm:braided-tensor-structure}. Hence $\Mod_\fd(A(V))$ is a braided tensor category with respect to the fusion tensor product $S\bt T=\Phi^\L(S)\odot \Phi^\L(T)$ and the braiding isomorphism
		\begin{equation}\label{eq:intro-braiding}
			\R_{M^1,M^2}(v^1\odot v^2)=e^{\pi i[\om]}.\left(e^{-\pi iL(0)}v^2\odot e^{-\pi iL(0)}v^1\right),
		\end{equation}
		and $\Phi^\L:\Mod_\fd(A(V))\to \CC_1(V)$ is a braided tensor equivalence. In particular, if $V$ is strongly rational, then $\Mod_\fd(A(V))$ is a modular tensor category with the twist $e^{2\pi i[\om]}$. Note that the associativity constraints on $\Mod_\fd(A(V))$ are transported from $\CC_1(V)$ through $\Phi^\L$, while the braiding isomorphism \eqref{eq:intro-braiding} is given by an explicit formula on finite-dimensional $A(V)$-modules, which only involves the action of $L(0)$ and the central element $[\om]\in A(V)$. As far as we know, such a formula for the braiding isomorphism on the bottom degrees did not appear in the literature before, even in the strongly rational case.

Besides the strongly rational VOAs, condition \eqref{eq:intro-assumption} is satisfied by the affine VOAs at admissible levels, the affine VOAs $V^k(\g)$ at generic levels $k+h^\vee\notin \Q$, the affine VOAs $V^{-2\pm\frac{1}{p}}(\sl_2)$, and some affine VOAs at negative levels in \cite{CY21}, see Example~\ref{ex:levels}. Note that the second condition in \eqref{eq:intro-assumption} does not follow from the first one. For instance, it fails for the Virasoro VOA with generic central charge.

\subsection{Relation to earlier work and proofs of the main theorems}

The braided tensor structure on module categories has been studied extensively in the theory of both VOAs and CFTs. For instance, on the CFT side, Tsuchiya and Kanie proved that the conformal blocks of the WZNW-model on $\P^1$ give rise to monodromy representations of the braid groups \cite{TK87}, and Moore and Seiberg described the braiding and fusion matrices of rational CFTs by a system of polynomial equations \cite{MS89}. For affine Lie algebras, these conformal blocks satisfy the Knizhnik-Zamolodchikov (KZ) equations \cite{KZ84}, whose monodromy was studied by Kohno \cite{Koh87} and Drinfeld \cite{Dri90}. Using the KZ equations, Drinfeld introduced a braided tensor category of modules over the quasi-triangular quasi-Hopf algebra $U(\g)[\![\varpi]\!]$, whose commutativity isomorphism is given by $e^{\pi i\varpi t}$, where $t$ is a multiple of the Casimir tensor of $\g$ \cite{Dri90}; and Kazhdan and Lusztig constructed a braided tensor structure on the category $KL^k(\g)$ of modules over the affine Lie algebra $\hat{\g}$ of level $k$, with $k+h^\vee\notin \Q_{\geq 0}$, and proved that it is equivalent to the module category of a quantum group \cite{KL1,KL2,KL3,KL4}. On the VOA side, Huang and Lepowsky developed the tensor product theory for module categories of VOAs \cite{HL1,HL2,HL3}, and Huang proved the convergence of products of intertwining operators \cite{H05} and the modularity of the module category of a strongly rational VOA \cite{H08}. This theory was generalized to the logarithmic case by Huang, Lepowsky, and Zhang \cite{HLZI,HLZII,HLZ}, and it was applied to many irrational VOAs, such as the affine VOAs at admissible levels \cite{CHY18,Cre19,MY25}, the affine VOAs beyond admissible levels \cite{CY21}, and the Virasoro VOAs \cite{CJORY21}. Recently, Huang proved that the category $\CC_1(V)$ admits a natural vertex tensor category structure for any VOA $V$ of CFT-type \cite{H25}.

We prove Theorem A by relating the universal property of Huang-Lepowsky's $P(w)$-tensor product \cite{HL1,HL2} with the isomorphisms among the spaces of $P(w)$-intertwining maps, three-point conformal blocks, and $\infty$-restricted conformal blocks on $(\P^1,\infty,w,0)$, see the diagram \eqref{diagram}. In comparison with \cite{Liu26}, one advantage of Theorem A is that we do not need $V$ to be rational or $C_2$-cofinite. Moreover, $\eta_w$ is given by an explicit formula, which will be used to compare our braiding isomorphism with the one of Huang-Lepowsky-Zhang.

To prove Theorem B, we first study the effect of shifting the marked points on the restricted conformal blocks. Recall that there is a linear isomorphism $\Om_0: I\fusion{M^1}{M^2}{M^3}\cong I\fusion{M^2}{M^1}{M^3}$ given by Huang and Lepowsky \cite{HL2}. It turns out that the geometric effect of $\Om_0$ on the space of three-point conformal blocks is given by shifting the marked points $(w,0)$ uniformly to $(0,-w)$, with the attached modules unchanged, see Proposition~\ref{prop:Om0blocks} and Remark~\ref{remark:Om0}. Since the factor $e^{-wL(1)}$ in $\Om_0$ acts trivially on $M^3(0)^\ast$, $\Om_0$ restricts to the $\infty$-restricted conformal blocks, and it gives rise to a canonical isomorphism of left $A(V)$-modules
		\[
		R_w: M^2\odot_{-w,0}M^1\ra M^1\odot_{w,0}M^2,\quad R_w(v^2\odot_{-w}v^1)=v^1\odot_w v^2,
		\]
		see Theorem~\ref{thm:Rw}. Then we move the point $-1$ back to $1$ by the parallel transport of the three-point conformal blocks along the upper and lower half unit circles, see Subsection~\ref{subsec:parallel}. Combining these maps gives the isomorphisms in Theorem B.


\subsection{Affine examples}

Our main example is the affine VOA $L_k(\g)$ at a positive integral level $k$, where everything in \eqref{eq:intro-braiding} can be written in terms of finite-dimensional $\g$-modules. Let $\g$ be a finite-dimensional simple Lie algebra, $\kappa=k+h^\vee$, and $\Om=\sum_a x_a\o x^a$ the Casimir tensor of $\g$. Recall that $A(L_k(\g))\cong U(\g)/I_k$, with $I_k=\<e_\theta^{k+1}\>$ \cite{FZ92}. For $\la,\mu\in P_+^k$, we first show that
		\[
		L(\la)\bt L(\mu)\cong T_{\la,\mu}=\frac{L(\la)\o_\C L(\mu)}{U(\g).J_{\la,\mu}},\quad
		J_{\la,\mu}=v_\la\o e_\theta^{k-(\la|\theta)+1}L(\mu)+e_\theta^{k-(\mu|\theta)+1}L(\la)\o v_\mu,
		\]
		where $U(\g)$ acts on $L(\la)\o_\C L(\mu)$ diagonally, see Theorem~\ref{thm:affine-odot-quotient}, which is a variation of \cite[Proposition~6.5]{AL26}. Using the Segal-Sugawara construction and the relations in the $\odot$-tensor product, we show that $[\om]$ acts on $T_{\la,\mu}$ by $h_\la+h_\mu+\frac{1}{\kappa}\Om_{\la,\mu}$, see Lemma~\ref{prop:affine-omega-action}. Then Theorem B gives the following explicit braiding isomorphism on the modular tensor category $\Mod_\fd(U(\g)/I_k)$ (see Theorem~\ref{thm:VOA-Drinfeld-braiding}):
		\begin{equation}\label{eq:intro-affine}
			\R_{\la,\mu}: T_{\la,\mu}\xrightarrow{\simeq} T_{\mu,\la},\quad \R_{\la,\mu}\left(\pi_{\la,\mu}(u\o v)\right)
			=\pi_{\mu,\la}\left(\left(\tau\circ \exp\left(\frac{\pi i }{\kappa} \Om_{\la,\mu}\right)\right)(u\o v) \right),
		\end{equation}
		where $\pi_{\la,\mu}: L(\la)\o_\C L(\mu)\to T_{\la,\mu}$ is the quotient map, and $\tau(u\o v)=v\o u$. Note that the ideal $I_k$ is not a Hopf ideal of $U(\g)$, and $L(\la)\o_\C L(\mu)$ is in general not a $U(\g)/I_k$-module. It turns out that the operator $\tau\circ\exp(\frac{\pi i}{\kappa}\Om)$ on $L(\la)\o_\C L(\mu)$ descends to the quotients $T_{\la,\mu}$, and gives the braiding isomorphism of the modular tensor category of $L_k(\g)$-modules. Thus the fusion tensor product and the braiding for $L_k(\g)$-modules are described by finite-dimensional $\g$-modules, their quotients $T_{\la,\mu}$, and the Casimir tensor $\Om$.

At generic levels, there is no quotient. If $k+h^\vee\notin\Q$, or $\g=\sl_2$ and $k=-2\pm\frac{1}{p}$, then $A(V^k(\g))\cong U(\g)$, and we show that $S\bt T=S\o_\C T$ with the diagonal action of $U(\g)$, and the braiding isomorphism is $\tau\circ\exp(\frac{\pi i}{\kappa}\Om_{S,T})$, see Proposition~\ref{prop:generic-odot} and Corollary~\ref{coro:generic-braiding}. This is Drinfeld's commutativity isomorphism $e^{\pi i\varpi t}$, with $t=d\,\Om$, at $\varpi=1/(d\kappa)$ \cite{Dri90,KL3}. In other words, the braiding isomorphism of the Drinfeld category is recovered from the restricted conformal blocks. However, the associativity constraints are not trivial, see Remark~\ref{rem:generic-associativity}, and their description using the KZ connection will be addressed in a subsequent work.

This paper is organized as follows: we first recall the basics of module categories of VOAs, some examples of $\CC_1(V)$, and the three-point conformal blocks on $\P^1$ in Section~2. In Section~3, we introduce the category $\CC_{A(V)}$ and the left $A(V)$-module $M^1\odot_{w,0}M^2$ using the restricted conformal blocks, and prove Theorem A. In Section~4, we construct the isomorphisms $R_w$ and $T_\ga$ and prove Theorem B. Then we compare our braiding isomorphism with the one of Huang-Lepowsky-Zhang, and show that $\Mod_\fd(A(V))$ is a braided tensor category. In Section~5, we determine the fusion tensor product and the braiding isomorphism for the affine VOAs at positive integral levels. Finally, in Section~6, we discuss the affine VOAs at generic levels.

		\section[Three-point conformal blocks on P1 and P(w)-intertwining operators]{Preliminaries}
		
		\subsection{Module categories of VOAs}
		
		For the definitions and basic properties of vertex operator algebras
		(VOAs), their modules, Borcherds Lie algebras, and Zhu algebras, we refer
		to \cite{FLM88,FHL93,Z96,LL04}. We recall the module categories used in
		this paper and fix our conventions. All vector spaces are over $\C$,
		and $\N=\{0,1,2,\ldots\}$.
		
		Throughout this paper, $V$ is a VOA of \emph{CFT-type}, that is,
		\[ V=\bigoplus_{n\in\N}V_n=\C\vac\oplus V_+,
		\quad V_+=\bigoplus_{n\geq 1}V_n.
		\]
		We denote its vacuum element and conformal element by $\vac\in V_0$ and
		$\omega\in V_2$, respectively. For a homogeneous vector $a\in V_n$,
		we write $\wt a=n$. Recall the following definition for various module categories over a VOA $V$.
		
		\begin{df}\cite{FHL93,HL1,Z96,DLM98,LL04,H08,H25}\label{df:2.1}
			A \emph{weak $V$-module} is a vector space $M$ equipped with a linear map $Y_M: V\to \End (M)[\![z,z^{-1}]\!]$, dented by
			\[
			Y_M(a,z)=\sum_{n\in\Z}a_{(n)}z^{-n-1},
			\qquad a\in V,
			\]
			satisfying the lower truncation, vacuum property, and Jacobi identity. We write
			\[
			Y_M(\omega,z)=\sum_{n\in\Z}L(n)z^{-n-2}.
			\]
			Morphisms in all module categories considered below are
			$V$-module homomorphisms.
			
			\begin{enumerate}
				\item An \emph{admissible} or $\N$-gradable \cite{Z96,DLM98} $V$-module is a weak $V$-module $M$ admitting a decomposition
				$M=\bigoplus_{n=0}^\infty M(n)$ such that
				\begin{equation}\label{eq:grading}
					a_{(m)}M(n)\subseteq M(n+\wt a-m-1)
				\end{equation}
				for every homogeneous $a\in V$, $m\in\Z$, and $n\in\N$,
				where $M(r)=0$ for $r<0$. The category of admissible $V$-modules is denote by $\Adm(V)$. A VOA is called {\bf rational} if the category $\Adm(V)$ is semisimple.
				\item
				A \emph{generalized $V$-module} is a weak $V$-module admitting
				a decomposition
				\[
				M=\bigoplus_{\alpha\in\C}M_{[\al]},
				\quad M_{[\al]}
				=\{v\in M:(L(0)-\alpha)^Nv=0
				\text{ for some }N\geq 1\}.
				\]
				A generalized module is called
				\emph{grading-restricted} \cite{H08,H25} if every $M_{[\al]}$ is
				finite-dimensional and $M_{[\al]}=0$ when
				$\operatorname{Re}\alpha\ll 0$.
				
				Note that for homogeneous $a\in V$ and $\al\in \C$, one has
				$ a_{(m)}M_{[\al]} \subseteq M_{[\alpha+\wt a-m-1]}.$ Hence
				every grading-restricted generalized module is admissible: $M=\bigoplus_{n=0}^\infty M(n)$, where $M(n)=\bigoplus_{\b\in \C/\Z} M_{[\b +n]}$ satisfying \eqref{eq:grading}.
				
				\item
				We denote by $\Ord(V)$ the category of \emph{ordinary}
				$V$-modules \cite{FHL93}. With our lower-bound convention,
				these are grading-restricted generalized modules on which
				$L(0)$ acts semisimply:
				\[ M=\bigoplus_{\alpha\in\C}M_\alpha,
				\quad M_\alpha=\{v\in M:L(0)v=\alpha v\}. \]
				An irreducible ordinary module has the form
				$
				M=\bigoplus_{n\in \N }M_{h+n},
				$ where $M_h\neq 0.$
				The number $h\in \C $ is called the \emph{conformal weight} of $M$.
				
				\item
				For a weak $V$-module $M$, set
				\[
				C_1(M)
				=\operatorname{span}_{\C}
				\{a_{(-1)}v:a\in V_+,\ v\in M\}.
				\]
				The module $M$ is called \emph{$C_1$-cofinite} if
				$\dim_\C M/C_1(M)< \infty$ \cite{H05}. Introduce the following category \cite{H25}:
				\begin{equation}\label{eq:maincat}
					\mathcal C_1(V)=\{ \text{\(C_1\)-cofinite grading-restricted generalized $V$-modules}\}.
				\end{equation}
				In particular, $C_1(V)$ is a full-subcategory of $\Adm(V)$.
			\end{enumerate}
		\end{df}
		For a grading-restricted generalized module $M$, its
		\emph{contragredient module} is the restricted dual
		$ M'=\bigoplus_{\alpha\in\C}M_{\al}^*, $
		equipped with $Y_{M'}: V\to \End(M')[\![z,z^{-1}]\!]$
		determined by
		\begin{equation}\label{eq:contragredient}
			\langle Y_{M'}(a,z)f,v\rangle
			=\left\langle
			f,\,
			Y_M\bigl(e^{zL(1)}(-z^{-2})^{L(0)}a,z^{-1}\bigr)v
			\right\rangle
		\end{equation}
		for $a\in V$, $f\in M'$, and $v\in M$
		\cite{FHL93,HLZ}. In particular, as an adjoint module, $V$ itself has a contragredient module $V'$. Set
		\[
		C_2(V)=\operatorname{span}_{\C}
		\{a_{(-2)}b:a,b\in V\}.
		\]
		The VOA $V$ is \emph{$C_2$-cofinite} \cite{Z96} if
		$\dim_\C V/C_2(V)<\infty$.
		
		\begin{df} \cite{DLM98,DLM00}
			A VOA $V$ is called {\bf strongly rational} if it is of CFT-type, simple,
			rational, $C_2$-cofinite, and self-dual $V\cong V'$. In this case, Any irreducible admissible $V$-module is ordinary; there are finitely many irreducible ordinary modules up to isomorphism, denoted by
			\[
			\mathscr W=\{V=M^0,M^1,\dots, M^N\};
			\]
			and the conformal weight $h_{M^i}$ of an irreducible ordinary module $M^i$ is a rational number.
		\end{df}
		
		The following theorem was proved by Huang:
		
		\begin{thm} \cite[Theorem~1.2]{H25}\label{thm:huang'sthm}
			Let $V$ be a CFT-type VOA. Then the category $\mathcal C_1(V)$ of $C_1$-cofinite grading-restricted
			generalized $V$-modules admits a natural vertex tensor category structure. In particular, $\mathcal C_1(V)$ is a braided tensor category with a twist.
		\end{thm}
		
		\subsection{\texorpdfstring{Examples of $\mathcal C_1(V)$}{Examples of C\_1(V)}}\label{sec2.2}
		The category $\mathcal C_1(V)$ will be the principal VOA module
		category considered in this paper. We list some examples of VOA and this category.
		
		\subsubsection{\bf Strongly rational VOAs.}
		If $V$ is strongly rational, then any grading-restricted generalized \(V\)-module is ordinary and $C_1$-cofinite \cite{ABD04}. In particular,
		\begin{equation}
			\mathcal C_1(V)=\Ord(V)\subsetneq \Adm(V),
		\end{equation}
		which is semisimple. $\mathcal C_1(V)$ consist precisely of finite direct sums of irreducible ordinary modules. Moreover, $\Adm(V)$ is generated by $\mathcal C_1(V)$ by arbitrary direct sums. Furthermore, $\Ord(V)$ is a modular tensor category, with fusion tensor product $\bt=\bt_{P(1)}$, tensor unit $V$, and twist $\theta_M=e^{2\pi iL(0)}$ \cite{H08,HLZ}.
		
		Examples include positive-definite
		even lattice VOAs $V_L$ \cite{D93}, Virasoro minimal model VOAs $L(c_{p,q},0)$ \cite{W93,DMZ94}, affine VOAs at positive integral levels $L_k(\g)$ \cite{DLM97}, and the principal series $\mathscr W$-algebras $W_k(\g)$ \cite{A15,ACL19}, etc.
		
		The examples listed below are all irrational:
		
		\subsubsection{\bf Admissible level affine VOAs } \label{sec:2.2.2}
		Let $\mathfrak g$ be a finite-dimensional simple Lie algebra,
		and let $p,q\in \Z_{\geq 1}$ be coprime integers. Define $k\in \Q$ by \cite{KW89}: $$k+h^\vee=\frac{p}{q},\quad p\geq \begin{cases}
			h^\vee &\mathrm{if}\ (r^\vee,q)=1,\\
			h&\mathrm{if}\ (r^\vee,q)=r^\vee.
		\end{cases}$$
		The simple affine VOA $L_k(\g)$ at level $k$ is called an {\em admissible level affine VOA}. The category $\mathcal O_{k,\mathrm{Ord}}$ of ordinary $L_k(\g)$-modules in the category $\mathcal O_k$ of $\hat{\g}$-modules is semisimple \cite{A16}. Moreover, it was proved in \cite[Section~3.4]{CY21} that
		\begin{equation}
			\mathcal C_1(L_k(\g))=\mathcal O_{k,\mathrm{Ord}},
		\end{equation}
		which is also semisimple.
		Prior to Theorem~\ref{thm:huang'sthm}, it was proved that $\mathcal O_{k,\mathrm{Ord}}$ is a braided tensor category with a twist \cite{CHY18}, and it is rigid when $\g$ is simply-laced \cite{Cre19}.
		
		\subsubsection{\bf Generic level affine VOAs}\label{sec:2.2.3} Let $\mathfrak g$ be a finite-dimensional simple Lie algebra, and let $k\in \C$ be such that $k+h^\vee\notin\mathbb Q_{\geq 0}$. Such a level $k$ is called {\em generic}. In this case, the vacuum module VOA $V^k(\g)$ is simple \cite{KL1}, and so it is the affine VOA $V^k(\g)=L_k(\g)$.
		It was proved in \cite[Corollary 3.4.5]{CY21} that
		\begin{equation}
			\mathcal C_1(V^k(\g))=KL^k(\g)
		\end{equation}
		is the Kazhdan--Lusztig's category for the affine Lie algebra $\g$, which contains finite-length grading-restricted generalized $V^k(\g)$-modules. Note that $KL^k(\g)$ was proved to admit a monoidal structure in \cite{KL1,KL2}.
		Note that this category is semisimple if $k+h^\vee\notin \Q$.
		
		However, if $k+h^\vee\in \Q_{<0}$, then $KL^k(\g)$ may or may not be semisimple. For instance, let
		\[
		V=V^{-2-\frac{1}{p}}(\mathfrak{sl}_2),
		\quad p\in\Z_{>0}.
		\]
		Then the category $\mathcal C_1(V)=KL_{-2-\frac{1}{p}}(\mathfrak {sl}_2)$ is semisimple.
		Its objects are finite direct sums of the simple $\hat{\mathfrak{sl}}_2$
		modules whose top spaces are finite-dimensional
		$\mathfrak{sl}_2$-modules \cite{CY21,AM26}.
		
		But if we take
		\[
		V=V^{-2-\frac{a}{b}}(\sl_2),\quad a\geq 2,\ b\geq 1,\ \gcd(a,b)=1.
		\]
		Then it follows from Kazhdan-Lusztig equivalence \cite{KL3,KL4} that
		\begin{equation}
			\mathcal C_1(V^{-2-\frac{a}{b}}(\sl_2))=KL_{-2-\frac{a}{b}}(\mathfrak {sl}_2)\simeq \mathcal C(q,\sl_2),\quad q=e^{\frac{\pi i}{k+2}}=e^{-\frac{\pi ib}{a}},
		\end{equation}
		where $\mathcal C(q,\sl_2)$ is the module category of the quantum group $U_q(\sl_2)$, which is not semisimple.
		
		\subsubsection{\bf The Virasoro VOAs at generic central charge}\label{sec:2.2.4}
		Let $V=L(c,0)$ be the Virasoro VOA \cite{FZ92}, with central charge
		\[
		c=13-6(t+t^{-1}),\qquad t\notin\mathbb Q.
		\]
		It was proved in \cite[Theorems~3.1.4 and~5.1.2]{CJORY21} that $\mathcal C_1(L(c,0))$ is semisimple, with infinitely many simple objects
		\[
		L(c,h_{r,s}(t)),\quad r,s\in\mathbb Z_{>0},\quad \mathrm{where}\quad h_{r,s}(t)
		=\frac{(rt-s)^2-(t-1)^2}{4t}.
		\]
		Every object in $\mathcal C_1(L(c,0))$ is a finite direct sum of these modules.
		
		\subsection[Formal intertwining operators and P(w)-intertwining operators]{\texorpdfstring{Formal intertwining operators and $P(w)$-intertwining operators}{Formal intertwining operators and P(w)-intertwining operators}}
		In this paper, we primarily consider $\CC_1(V)$, or a full
		tensor subcategory, whose objects are {\bf ordinary}
		$V$-modules.
		For any three such objects, every logarithmic intertwining
		operator among them contains no logarithmic terms
		\cite[Remark~3.23]{HLZII}, and hence is an intertwining operator
		in the usual sense \cite{FHL93}.
		
		\begin{lm}\label{lm:ssordmod}
			If $\CC_1(V)$ is semisimple, then every object of $\CC_1(V)$
			is an ordinary $V$-module.
		\end{lm}
		\begin{proof}
			Let $W$ be an irreducible grading-restricted generalized module
			in $\CC_1(V)$. Choose $\al\in\C$ such that $W_{[\al]}\neq 0$,
			and let $0\neq v\in W_{[\al]}$. There exists a positive integer
			$N$ such that
			$
			u=(L(0)-\al)^{N-1}v\neq 0,
			$ $
			(L(0)-\al)^Nv=0.
			$
			Then $0\neq u\in W_\al$, where $W_\al$ denotes the eigenspace
			of $L(0)$ with eigenvalue $\al$. Set
			$
			W'=\bigoplus_{\al\in\C}W_\al.
			$
			Then $W'$ is a nonzero submodule of $W$, so irreducibility
			implies that $W=W'$. Hence $W$ is ordinary.
			
			For any object $M\in\CC_1(V)$, semisimplicity gives a decomposition
			$M=\bigoplus_i W^i$, where each $W^i$ is irreducible.
			Since each $W^i$ is ordinary and $M$ is grading-restricted,
			$M$ is also ordinary.
		\end{proof}
		
		Examples of semisimple $\CC_1(V)$ are given in the previous subsection.
		
		\begin{df}\cite{FHL93}
			Let $M^1,M^2,M^3$ be ordinary $V$-modules. A {\em (formal) intertwining operator of type $\fusion{M^1}{M^2}{M^3}$} is a linear map
			$$I(\cdot,x):M^1\ra \Hom(M^2,M^3)\{x,x^{-1}\},\quad I(v^1,x)v^2=\sum_{n\in \C}v^1_{(n)}v^2\cdot x^{-n-1},$$
			where $x$ is a formal variable,
			satisfying the following axioms:
			\begin{enumerate}
				\item (truncation property) For any $v^1\in M^1$ and $v^2\in M^2$, $v^1_{(n)}v^2=0$ when the real part of $n$ is sufficiently large.
				\item ($L(-1)$-derivative property) For any $v^1\in M^1$, $I(L(-1)v^1,x)=\frac{d}{dx} I(v^1,x)$.
				\item (Formal Jacobi identity) For any $v^1\in M^1$, $v^2\in M^2$, and $a\in V$, one has
				\begin{equation}\label{formalIOjacobi}
					\begin{aligned}
						x_0^{-1}\delta\left(\frac{x_1-x_2}{x_0}\right) &Y_{M^3}(a,x_1)I(v^1,x_2)v^2-x_0^{-1}\delta\left(\frac{-x_2+x_1}{x_0}\right)I(v^1,x_2)Y_{M^2}(a,x_1)v^2\\
						&=x_2^{-1}\delta\left(\frac{x_1-x_0}{x_2}\right) I(Y_{M^1}(a,x_0)v^1,x_2)v^2.
					\end{aligned}
				\end{equation}
			\end{enumerate}
			In particular, $v^1_{(n)}.M^2_\al\subset M^3_{\wt (v^1)-n-1+\al}$ for all $n,\al\in \C $.
			
			The vector space of intertwining operators of type $\fusion{M^1}{M^2}{M^3}$ is denoted by $I\fusion{M^1}{M^2}{M^3}$. Its dimension, denoted by $N\fusion{M^1}{M^2}{M^3}$, is called the {\em fusion rule among $M^1,M^2$, and $M^3$}.
		\end{df}
		
		\begin{remark}\label{rem:modes}
			Assume $M^1,M^2,M^3$ are irreducible ordinary $V$-modules of conformal weights $h_1,h_2$, $h_3$, respectively. Let $M^i=\bigoplus_{n=0}^\infty M^i(n)$, where $M^i(n)=M^i_{h_i+n}$ is the eigenspace of $L(0)$ of eigenvalue $h_i+n$. Denote $\deg (v^i)=n$ for $v^i\in M^i(n)$, $i=1,2,3$, and write $h=h_1+h_2-h_3$. Then the intertwining operator can be written as
			\[
			I(v^1,x)v^2=\sum_{n\in \Z} v^1(n)v^2 \cdot x^{-n-1-h}.
			\]
			The components satisfy $v^1(n)M^2(m)\subset M^3(\deg v^1-n-1+m),\quad m\in \N,\ n\in \Z$. However, if the modules$M^i$ do not have a uniform conformal weight $h_i$, one cannot make such a simplification.
		\end{remark}
		
		
		
		Let $w\in \P^1\bs\{\infty,0\}$. Fix an integer $p\in \Z$ for the branch of logarithm of $w$:
		\begin{equation}\label{logw}
			l_p(w)=\Log(w)+2p \pi i, \quad \mathrm{where}\quad \Log(w)=\log(|w|)+i\Arg(w)
		\end{equation}
		is the principal branch of $\log(w)$, with the branch cut along the negative real axis. i.e., $-\pi< \Arg(w)\leq \pi$. Then $w^n=e^{nl_p(w)}$ for any $n\in \C$.
		
		The notion of $P(w)$-intertwining operators was introduced by Huang-Lepowsky. Given a graded vector space $W=\bigoplus_{n\in \C} W(n)$, with $\dim W(n)<\infty$ for all $n$, let its graded dual space be $W'=\bigoplus_{n\in \C} W(n)^\ast$. Write
		\[
		\overline{W}=\Hom_\C(W',\C)=\prod_{n\in \C} W(n).
		\]
		If $(W,Y_W)$ is an ordinary $V$-module, then the action of $a_{(n)}=\Res_z z^n Y_W(a,z)$ on $W$ extends to an action on $\overline{W}$, for all $a\in V$, $n\in \Z$.
		
		\begin{df}\cite{HL1,HL3}\label{df:P(w)IO}
			Let $M^1,M^2,M^3$ be three ordinary $V$-modules. A $P(w)$-intertwining operator is a linear map
			\[ F:M^1\o_\C M^2\ra \overline{M^3}=\prod_{\al\in\C} M^3_\al,\]
			satisfying the formal Jacobi identity
			\begin{equation}\label{Fjacobi}
				\begin{aligned}
					x_0^{-1}\delta\left(\frac{x_1-w}{x_0}\right)Y_{M^3}(a,x_1)F(v_1\o v_2)
					&= w^{-1}\delta\left(\frac{x_1-x_0}{w}\right) F(Y_{M^1}(a,x_0)v_1\o v_2)\\
					&\ \ \ + x_0^{-1}\delta\left(\frac{w-x_1}{-x_0}\right) F(v_1\o Y_{M^2}(a,x_1)v_2).
				\end{aligned}
			\end{equation}
			The vector space of $P(w)$-intertwining operators is denoted by $\mathcal{M}[P(w)]_{M^1M^2}^{M^3}$.
		\end{df}
		
		Note that the formal Jacobi identity above also has an equivalent residue form:
		\begin{equation}\label{FResjacobi}
			\begin{aligned}
				&\Res_{x_1} Y_{M^3}(a,x_1) F(v^1\o v^2) \cdot \iota_{|x_1|>|w|}( x_1^n(x_1-w)^l)\\
				&\ \ -\Res_{x_1} F(v^1\o Y_{M^2}(a,x_1)v^2 )\cdot \iota_{|w|>|x_1|}(x_1^n(x_1-w)^l)\\
				&=\Res_{x_1-w} F(Y_{M^1}(a,x_1-w)v^1\o v^2) \cdot \iota_{|w|>|x_1-w|} (x_1^n (x_1-w)^l),
			\end{aligned}
		\end{equation}
		for all $n,l\in \Z$.
		Indeed, by applying $\Res_{x_0}\Res_{x_1}x_0^lx_1^n$ to \eqref{Fjacobi}, we obtain \eqref{FResjacobi}, with the indicated expansions of the rational functions. Conversely, \eqref{FResjacobi} gives the coefficient of $x_0^{-l-1}x_1^{-n-1}$ in \eqref{Fjacobi} for every $l,n\in\Z$. Hence it implies \eqref{Fjacobi}.
		
		We briefly recall the argument in \cite[Section 12]{HL3} showing that $$I\fusion{M^1}{M^2}{M^3}\cong \mathcal{M}[P(w)]_{M^1M^2}^{M^3}.$$
		
		Let $I\in I\fusion{M^1}{M^2}{M^3}$. Fix a branch $l_p(w)$ for $\log(w)$ \eqref{logw}, define $F^{P(w)}_{I,p}\in \mathcal{M}[P(w)]_{M^1M^2}^{M^3}$ by
		\begin{equation}\label{eq:forward}
			F^{P(w)}_{I,p}(v^1\o v^2):=I(v^1,e^{l_p(w)})v^2=\sum_{n\in \C} v^1_{(n)}v^2 \cdot e^{(-n-1)l_p(w)} \in \overline{M^3}.
		\end{equation}
		Conversely, given $F\in \mathcal{M}[P(w)]_{M^1M^2}^{M^3}$, define $I_{F,p}\in I\fusion{M^1}{M^2}{M^3}$ by
		\begin{equation}\label{eq:backward}
			I_{F,p}(v^1,x)v^2:=\sum_{n\in \C} v^1_{(n)}v^2 \cdot x^{-n-1} \in M^3\{x \},
		\end{equation}
		where $v^1_{(n)}v^2\in M^3_{\wt (v^1)-n-1+\wt (v^2)}$ is the projection of $F(v^1\o v^2)\in \overline{M^3}$ onto this homogeneous graded subspace multiplied by $e^{(n+1)l_p(w)}\in \C$. 
		
		\begin{lm}\cite[Proposition 12.2]{HL3}\label{lm2.6}
			Fix a branch $l_p(w)$ for $\log(w)$ \eqref{logw}, so that $w^{m}=e^{ml_p(w)}$ for any $m\in \C$. The linear maps $I\mapsto F_{I,p}$ \eqref{eq:forward} and $F\mapsto I_{F,p}$ \eqref{eq:backward} induce mutually inverse isomorphism between the vector spaces $I\fusion{M^1}{M^2}{M^3}$ and $\mathcal{M}[P(w)]_{M^1M^2}^{M^3}$.
			
			Furthermore, let $z=l_p(w)\in\C$, so that $e^z=w$. Then there is an isomorphism of vector spaces
			\[
			\mathcal{M}[P(w)]_{M^1M^2}^{M^3} \xrightarrow{\simeq} \mathcal{M}[P(1)]_{M^1M^2}^{M^3},\quad F^{P(w)}\mapsto F^{P(1)},
			\]
			where $F^{P(1)}(v^1\o v^2)=e^{-z L_3(0)} F^{P(w)}(e^{z L_1(0)}v^1\o e^{z L_2(0)}v^2)$, where $L_i(0)$ acts on $M^i$, for $i=1,2,3$.
		\end{lm}
		
		In the residue formulas below, $\Res_z^{\mathrm{out}}$ denotes the coefficient of $z^{-1}$ in the Laurent expansion for $|z|>|w|$. Thus it is the negative of the geometric residue at $z=\infty$:
		\[
		\Res_z^{\mathrm{out}}F(z)\, dz
		=-\Res_{z=\infty}F(z)\, dz.
		\]
		
		We use the usual local residues at the finite marked points.
		Under the correspondence \eqref{eq:forward}, we can rewrite the residue form Jacobi identity for the $P(w)$-intertwining operators \eqref{FResjacobi} as a residue form Jacobi identity for $I(v^1,e^{l_p(w)})$, which is useful in our later discussion.
		
		\begin{lm}
			Let $I\in I\fusion{M^1}{M^2}{M^3}$. Fix a branch $l_p(w)$ for $\log(w)$ \eqref{logw}, so that $w^{m}=e^{ml_p(w)}$ for any $m\in \C$.
			The formal Jacobi identity \eqref{formalIOjacobi} has the following equivalent residue form in terms of the corresponding $P(w)$-intertwining operator $I(\cdot, e^{l_p(w)})$:
			\begin{equation}\label{eq:residueJacobi}
				\begin{aligned}
					&\Res_z^{\mathrm{out}} \braket*{(v^3)'}{Y_{M^3}(a,z)I(v^1,e^{l_p(w)})v^2} \cdot \iota_{|z|>|w|}\left(\frac{z^n}{(z-w)^m}\right) dz\\
					&\ \ \ -\Res_{z=0} \braket*{(v^3)'}{I(v^1,e^{l_p(w)}) Y_{M^2}(a,z)v^2} \cdot \iota_{|w|>|z|}\left(\frac{z^n}{(z-w)^m}\right)dz\\
					&=\Res_{z-w=0} \braket*{(v^3)'}{I(Y_{M^1}(a,z-w)v^1,e^{l_p(w)})v^2} \cdot \iota_{|w|>|z-w|} \left(\frac{z^n}{(z-w)^m}\right)dz,
				\end{aligned}
			\end{equation}
			where $z^n/(z-w)^m$, with $m,n\in \Z$, is a standard rational function on $\P^1$ in the affine coordinate $z$, with possible poles at $z=0,\infty, w$, and $a\in V$, $(v^3)'\in (M^3)'$, $v^1\in M^1$, $v^2\in M^2$ are homogeneous elements.
		\end{lm}
		
		\subsection[VOA-conformal blocks associated to P1]{\texorpdfstring{Space of VOA-conformal blocks associated to $(\P^1, \infty,w,0)$}{Space of VOA-conformal blocks associated to (P\textasciicircum{}1, infinity, w, 0)}}
		Following the notation in \cite{NT05,Liu26},
		spanning elements of the chiral Lie algebra $\mathcal{L}_{\P^1\bs\{\infty,w, 0\}}(V)$ can be represented by
		\begin{equation}\label{section4point}
			X_w(z)=a\o f(z)(dz)^{1-\wt a}=a\o \frac{z^n}{(z-w)^m}(dz)^{1-\wt a},\quad a\in V,\ m,n\in \Z,
		\end{equation}
		where $z$ is the affine coordinate on $\P^1\bs\{\infty,w,0\}$.
		The representatives \eqref{section4point} satisfy the following relation:
		$$
		L(-1)a\o f(z)(dz)^{-\wt a}=-a\o \frac{df(z)}{dz} (dz)^{1-\wt a},\quad a\in V.
		$$
		The Lie bracket on $\mathcal{L}_{\P^1\bs\{\infty,w,0\}}(V)$ is given by
		\begin{equation}
			\left[ a\o f(z) (dz)^{1-\wt a}, b\o g(z)(dz)^{1-\wt b} \right]:=\sum_{j\geq 0} \frac{1}{j!} a_{(j)}b\o \frac{d^jf(z)}{dz^j}\cdot g(z) (dz)^{2-\wt a-\wt b+j}.
		\end{equation}
		
		We omit the symbol of the differential forms for simplicity of our notation.
		
		
		\begin{df}\label{def:cfb} \cite{NT05,DGT24,Liu26}
			Let $V$ be a VOA, $M^1,M^2,M^3\in \Ord(V)$, and let $(M^3)'$ be the contragredient module of $M^3$. Then
			\begin{enumerate}
				\item $(M^3)'$ is a module over the Lie algebra $\mathcal{L}_{\P^1\bs\{\infty,w,0\}}(V)$ via $\rho_\infty:\mathcal{L}_{\P^1\bs\{\infty,w,0\}}(V)\ra \gl((M^3)')$,
				\begin{equation}\label{eq:inftyaction}
					\begin{aligned}
						\rho_\infty\left(a\o\frac{z^n}{(z-w)^m}\right)((v^3)')&=\Res_z^{\mathrm{out}} Y_{(M^3)'}(\vartheta(a),z^{-1})(v^3)'\cdot \iota_{|z|>|w|}\left(\frac{z^n}{(z-w)^m}\right)\\
						&=-\sum_{j\geq 0}\binom{-m}{j} (-w)^j a'_{(n-m-j)}(v^3)',\quad (v^3)'\in (M^3)',
					\end{aligned}
				\end{equation}
				where $\vartheta(a)=-e^{zL(1)}(-z^{-2})^{L(0)}(a)$ and $a'_{(k)}=\sum_{i\geq 0} \frac{(-1)^{\wt a}}{i!} (L(1)^ia)_{(2\wt a-k-i-2)}$.
				\item $M^1$ is a module over the Lie algebra $\mathcal{L}_{\P^1\bs\{\infty,w,0\}}(V)$ via $\rho_w:\mathcal{L}_{\P^1\bs\{\infty,w,0\}}(V)\ra \gl(M^1)$,
				\begin{equation}\label{eq:1action}
					\begin{aligned}
						\rho_w\left(a\o \frac{z^n}{(z-w)^m}\right)(v^1)&=\Res_{z=w} Y_{M^1}(a,z-w)v^1\cdot \iota_{|w|>|z-w|}\left(\frac{z^n}{(z-w)^m}\right)\\
						&=\sum_{j\geq 0} \binom{n}{j} w^{n-j} a_{(j-m)}v^1,\quad v^1\in M^1.
					\end{aligned}
				\end{equation}
				\item $M^2$ is a module over the Lie algebra $\mathcal{L}_{\P^1\bs\{\infty,w,0\}}(V)$ via $\rho_0:\mathcal{L}_{\P^1\bs\{\infty,w,0\}}(V)\ra \gl(M^2)$,
				\begin{equation}\label{eq:0action}
					\begin{aligned}
						\rho_{0}\left(a\o\frac{z^n}{(z-w)^m}\right)(v^2)&=\Res_{z=0} Y_{M^2}(a,z)v^2\cdot \iota_{|w|>|z|}\left(\frac{z^n}{(z-w)^m}\right)\\
						&=\sum_{j\geq 0} \binom{-m}{j}(-w)^{-m-j} a_{(n+j)}v^2,\quad v^2\in M^2.
					\end{aligned}
				\end{equation}
			\end{enumerate}
			
			In particular, $(M^3)' \o_\C M^1\o_\C M^2$ is a tensor product module over the chiral Lie algebra $\mathcal{L}_{\P^1\bs\{\infty,w,0\}}(V)$. The action of the Lie algebra can be expressed explicitly:
			\begin{equation}\label{eq:actionofchiralliealgebraelements}
				\begin{aligned}
					&\rho_{\infty,w,0}\left(a\o\frac{z^n}{(z-w)^m}\right). ((v^3)'\o v^1\o v^2)\\
					&=-\sum_{j\geq 0}\binom{-m}{j} (-w)^j a'_{(n-m-j)}((v^3)')\o v^1\o v^2+ \sum_{j\geq 0} \binom{n}{j} w^{n-j}(v^3)'\o a_{(j-m)}(v^1)\o v^2\\
					&\ \ \ +\sum_{j\geq 0} \binom{-m}{j}(-w)^{-m-j}(v^3)'\o v^1\o a_{(n+j)}(v^2).
				\end{aligned}
			\end{equation}
			where $(v^3)'\o v^1\o v^2\in (M^3)' \o M^1\o M^2$.
			
			The {\em space of coinvariants} associated to $(\P^1,(\infty, w,0), ((M^3)',M^1, M^2))$:
			\[ [(M^3)' \o M^1\o M^2]_{(\P^1,\infty,w,0)}=\frac{(M^3)'\otimes M^1\otimes M^2}{\mathcal{L}_{\P^1\bs\{\infty,w,0\}}(V).\left((M^3)'\otimes M^1\otimes M^2\right)}.\]
			The {\em space of conformal blocks} associated to $(\P^1,(\infty, w,0), ((M^3)',M^1, M^2))$:
			\[ \mathscr{C}((M^3)' \o M^1\o M^2)_{(\P^1,\infty,w,0)}= \left(\frac{(M^3)'\otimes M^1\otimes M^2}{\mathcal{L}_{\P^1\bs\{\infty,w,0\}}(V).\left((M^3)'\otimes M^1\otimes M^2\right)}\right)^\ast. \]
		\end{df}
		
		\begingroup\emergencystretch=1em
		When $M^1,M^2,M^3$ are irreducible ordinary modules, the space of three-point conformal blocks on $\P^1$ can be identified with the intertwining operators $I\fusion{M^1}{M^2}{M^3}$ \cite{Z96,GLZ26}. We generalize this fact to arbitrary ordinary $V$-modules in $\Ord(V)$ and to the $P(w)$-intertwining operators.
		\par\endgroup
		
		\begin{prop}\label{prop2.7}
			Let $M^1,M^2,M^3$ be three ordinary $V$-modules. Fix a branch $l_p(w)$ for $\log(w)$ \eqref{logw}, so that $w^{m}=e^{ml_p(w)}$ for any $m\in \C$. Then we have the following isomorphism of vector spaces:
			\begin{equation}\label{eq:P(w)isocfb}
				\begin{aligned}
					I\fusion{M^1}{M^2}{M^3}&\cong \mathcal{M}[P(w)]_{M^1M^2}^{M^3}\cong \mathscr{C}((M^3)' \o M^1\o M^2)_{(\P^1,\infty,w,0)},\\
					I(\cdot, x)&\mapsto F=I(\cdot , e^{l_p(w)})\mapsto \varphi_{I,p},
				\end{aligned}
			\end{equation}
			where
			\begin{equation}\label{eq:defofvarphiI}
				\braket*{\varphi_{I,p}}{(v^3)'\o v^1\o v^2}:= \braket*{(v^3)'}{F(v^1\o v^2)}=\braket*{(v^3)'}{I(v^1,e^{l_p(w)})v^2},
			\end{equation}
			where $(v^3)', v^1,v^2$ are homogeneous elements in $(M^3)',M^1,M^2$, respectively.
		\end{prop}
		\begin{proof}
			By Lemma~\ref{lm2.6}, $ I(\cdot, x)\mapsto F=I(\cdot , e^{l_p(w)})$ is a linear isomorphism. It remains to show that $F\mapsto \varphi_{I,p}$ is an isomorphism.
			Given an element $a\o z^n(z-w)^l\in \mathcal{L}_{\P^1\bs\{\infty,w,0\}}(V)$, with $n,l\in \Z$, it follows from \eqref{eq:actionofchiralliealgebraelements} that
			\begin{align*}
				&\braket*{\varphi_{I,p}}{(a\o z^n(z-w)^l).((v^3)'\o v^1\o v^2)}\\
				&=-\sum_{j\geq 0} \binom{l}{j} (-w)^j \braket*{\varphi_{I,p}}{a'_{(n+l-j)}(v^3)'\o v^1\o v^2}+ \sum_{j\geq 0}\binom{n}{j} w^{n-j} \braket*{\varphi_{I,p}}{(v^3)'\o a_{(j+l)}v^1\o v^2}\\
				&\ \ \ +\sum_{j\geq 0}\binom{l}{j} (-w)^{l-j} \braket*{\varphi_{I,p}}{(v^3)'\o v^1\o a_{(n+j)}v^2}\\
				&=-\sum_{j\geq 0}\binom{l}{j} (-w)^j \braket*{(v^3)'}{a_{(n+l-j)}F(v^1\o v^2)}+\sum_{j\geq 0}\binom{n}{j}w^{n-j}\braket*{(v^3)'}{F(a_{(j+l)}v^1\o v^2)}\numberthis\label{eq:process}\\
				&\ \ \ +\sum_{j\geq 0}\binom{l}{j} (-w)^{l-j}\braket*{(v^3)'}{F(v^1\o a_{(n+j)}v^2)}\\
				&=-\Res_{x_1} \braket*{(v^3)'}{Y_{M^3}(a,x_1) F(v^1\o v^2)} \iota_{|x_1|>|w|}( x_1^n(x_1-w)^l)\\
				&\ \ \ +\Res_{x_1-w} \braket*{(v^3)'}{F(Y_{M^1}(a,x_1-w)v^1\o v^2)} \iota_{|w|>|x_1-w|} (x_1^n (x_1-w)^l)\\
				&\ \ \ +\Res_{x_1}\braket*{(v^3)'}{F(v^1\o Y_{M^2}(a,x_1)v^2)} \iota_{|w|>|x_1|} (x_1^n (x_1-w)^l)\\
				&=0,
			\end{align*}
			where the last equality follows from the fact that $F$ satisfies Jacobi identity \eqref{FResjacobi}. Hence $\varphi_{I,p}$ is a well-defined element in $\mathscr{C}((M^3)' \o M^1\o M^2)_{(\P^1,\infty,w,0)}$.
			
			Conversely, we define
			\begin{align*}
				\mathscr{C}((M^3)' \o M^1\o M^2)_{(\P^1,\infty,w,0)}&\ra \mathcal{M}[P(w)]_{M^1M^2}^{M^3},\quad \varphi\mapsto F_\varphi,\\
				\braket*{(v^3)'}{F_\varphi(v^1\o v^2)}&:=\braket*{\varphi}{(v^3)'\o v^1\o v^2},
			\end{align*}
			The reverse calculation of \eqref{eq:process} shows that $F_\varphi$ is a well-defined $P(w)$-intertwining operator. Clearly, $\varphi\mapsto F_\varphi$ is the inverse of the second map of \eqref{eq:P(w)isocfb}.
		\end{proof}
		
		The following observation which immediately follows from Proposition~\ref{prop2.7} and Definition~\ref{def:cfb} is useful in our later discussion. It is convenient for us to use the notation $I(\cdot, e^{l_p(w)})$ instead of $F$ for the $P(w)$-intertwining operators.
		
		\begin{coro}\label{coro2.8}
			Let $X_w(z)=a\o \frac{z^n}{(z-w)^m}\in \mathcal{L}_{\P^1\bs\{\infty,w,0\}}(V)$. Under the identification \eqref{eq:P(w)isocfb}, we have
			\begin{align*}
				&\braket*{\varphi_{I,p}}{\rho_\infty(X_w(z))(v^3)'\o v^1\o v^2}\\
				&=-\Res_z^{\mathrm{out}} \braket*{(v^3)'}{Y_{M^3}(a,z)I(v^1,e^{l_p(w)})v^2} \cdot \iota_{|z|>|w|}\left(\frac{z^n}{(z-w)^m}\right) dz,\\
				&\braket*{\varphi_{I,p}}{(v^3)'\o \rho_w(X_w(z))(v^1)\o v^2}\\
				&=\Res_{z-w=0} \braket*{(v^3)'}{I(Y_{M^1}(a,z-w)v^1,e^{l_p(w)})v^2}\cdot \iota_{|w|>|z-w|}\left(\frac{z^n}{(z-w)^m}\right)dz,\\
				&\braket*{\varphi_{I,p}}{(v^3)'\o v^1\o \rho_0(X_w(z))(v^2)}\\
				&=\Res_{z=0} \braket*{(v^3)'}{I(v^1,e^{l_p(w)}) Y_{M^2}(a,z)v^2} \cdot\iota_{|w|>|z|}\left(\frac{z^n}{(z-w)^m}\right)dz
			\end{align*}
			In particular, the fact that $\varphi_{I,p}$ is invariant under the action of $X_w(z)$ \eqref{eq:actionofchiralliealgebraelements} is equivalent to the residue form Jacobi identity \eqref{eq:residueJacobi} of the $P(w)$-intertwining operator $I(\cdot, e^{l_p(w)})$.
		\end{coro}
		
		
		
		\section{\texorpdfstring{The fusion tensor structure on $\mathcal C_{A(V)}$}{The fusion tensor structure on C\_\{A(V)\}}}
In this section, we generalize the contracted tensor product $M^1\odot M^2$ in \cite{Liu26} to an arbitrary point $w\in\C^\times$, and show that if $\CC_1(V)$ is semisimple and $\Phi^\L(\Mod_\fd(A(V)))\ssq \CC_1(V)$, then $M^1\odot_{w,0}M^2$ is isomorphic to the bottom degree of Huang-Lepowsky's $P(w)$-tensor product module $M^1\bt_{P(w)}M^2$, see Theorem~\ref{thm:eta-explicit}. This isomorphism will be used to define the fusion tensor structure on $\CC_{A(V)}=\Mod_\fd(A(V))$ in Subsection~\ref{sec:braiding}.
		\subsection{\texorpdfstring{Zhu algebra and the category $\CC_{A(V)}$}{Zhu algebra and the category C\_\{A(V)\}}}
We first recall properties of the Zhu algebra $A(V)$ \cite{Z96} and functors between the categories of $A(V)$-modules and $V$-modules.
		\subsubsection{\texorpdfstring{The adjoint functors $(\Phi^\L \dashv \Omega): \Mod(A(V))\rightleftarrows \Adm(V)$}{The adjoint functors (Phi\textasciicircum{}L left adjoint to Omega): Mod(A(V)) <-> Adm(V)}}
		
		Let $A(V)$ be the {\em Zhu algebra} associated to the VOA $V$ \cite{Z96}, and let $\Mod(A(V))$ denote the category of left $A(V)$-modules.
		
		The bottom-degree $M(0)$ of an admissible $V$-module $M$, or more generally, the space of highest-weight vectors $\Om(M)=\{v\in M: a_{(m)}v=0,\ \wt a-m-1<0\}$ of a weak $V$-module $M$, is a left $A(V)$-module via
		$$A(V)\ra \End(\Om(M)),\quad [a]\mapsto o(a)=a_{(\wt a-1)},\quad a\in V.$$
		We can view $\Omega$ as a functor $\Om: \Adm(V)\to \Mod(A(V))$.
		
		For a left $A(V)$-module $S$, the linear dual space $S^\ast=\Hom_\C(S,\C)$ is a natural right $A(V)$-module. We give it a left $A(V)$-module via the contragredient action $A(V)\times S^\ast \to S^\ast,\ ([a],f)\mapsto [a].f$, where
		\begin{equation}\label{eq:dualAVmod}
			\<[a].f,s\>=\<f, \theta_{A(V)}([a])s\>,\qquad
			\theta_{A(V)}([a])=[e^{L(1)}(-1)^{L(0)}a],
		\end{equation}
		for $f\in S^\ast$, $s\in S$.
		Now let $M\in \Ord(V)$, and $M'$ be its contragredient module via \eqref{eq:contragredient}. Then $\Om(M')$ is also a natural right $A(V)$-module:
		\begin{equation}\label{eq:OmM'naturalrightmod}
			\<f.[a], v\>=\<f,o(a)v\>,\quad f\in \Om(M'),\ [a]\in A(V),\ v\in M,
		\end{equation}
		which can be viewed as a left $A(V)$-module via \eqref{eq:dualAVmod}. Note that this left $A(V)$-module structure agrees with the left $A(V)$-module structure inherited from the contragredient module $M'$.
		
		Let $\scrU=\scrU(V)$ be the {\em universal enveloping algebra} associated to the VOA $V$ \cite{FZ92,FBZ04}, which is a complete $\Z$-graded topological associative algebra with neighborhood bases
		$$\sN^n_\sL\scrU=\scrU\cdot \scrU_{\leq -n},\quad n\geq 0.$$
		The {\em generalized Verma module} \cite{DLM98,DGK25,GL26} associated to a left $A(V)$-module $S$ is an admissible $V$-module defined by
		\begin{equation}\label{Verma}
			\Phi^\L(S)=(\scrU/\sN^1_\sL\scrU)\o_{\scrU_0}S=\bigoplus_{n=0}^\infty \Phi^\L(S)(n).
		\end{equation}
		Here $S$ is a $\scrU_0$-module via $\scrU_0\twoheadrightarrow A(V)\cong \scrU_0/\sN^1_\L \scrU_0$, and
		\[
		\Phi^\L(S)(n)=\spn\left\{ a^1_{[n_1]}\ds a^r_{[n_r]} v: a^i\in V,\ v\in S,\  \sum_{i=1}^r (\wt a^i-n_i-1)=n\right\}.
		\]
		$\Phi^\L$ is also a functor $\Phi^\L: \Mod(A(V))\to \Adm(V)$.
		Moreover, $(\Phi^\L\dashv \Om): \Mod(A(V))\rightleftarrows \Adm(V)$ is an adjoint pair of functors. Recall the following fact in \cite{DLM98,DGK25,GL26}:
		
		\begin{lm}\label{lm:sic}
			If the VOA $V$ satisfies the {\em strong identity condition}, then the pair of functors
			$$(\Phi^\L\dashv \Om): \Mod(A(V))\rightleftarrows \Adm(V)$$
			is an adjoint equivalence. In particular, it is an adjoint equivalence if $V$ is rational or if $V$ is the Heisenberg VOA $M(1)$.
		\end{lm}
		
		\subsubsection{\texorpdfstring{The subcategory $\CC_{A(V)}$ of $A(V)$-modules }{The subcategory C\_\{A(V)\} of A(V)-modules}}
		\begin{df}
			Note that $\CC_1(V)\subset \Adm(V)$ is a full subcategory, see Definition~\ref{df:2.1}. Then we have a restricted functor
			\begin{equation}\label{eq:resOm}
				\Omega: \CC_1(V)\to \Mod(A(V)).
			\end{equation}
			We denote its essential image by $\CC_{A(V)}$, which is the full subcategory of $\Mod(A(V))$ consisting of modules isomorphic to $\Omega(M)$ for some $M\in\CC_1(V)$.
		\end{df}
		Whenever the functor \eqref{eq:resOm} is fully faithful,
		it induces an equivalence of categories
$
\Omega:\CC_1(V)\xrightarrow{\sim}\CC_{A(V)}.$
		We give some sufficient conditions for the fully faithfulness of $\Om$ \eqref{eq:resOm}.
		\begin{prop}\label{prop:sufffullyfaithful}
			The functor $\Om$ \eqref{eq:resOm} is fully faithful if either one of the following condition holds:
			\begin{enumerate}
				\item $V$ satisfies the strong identity condition. In this case, the quasi-inverse of $\Om$ is the generalized Verma module functor $\Phi^\L$ \eqref{Verma}, and so $$(\Phi^\L\dashv \Om): \CC_{A(V)}\rightleftarrows \CC_1(V)$$ are equivalence of categories.
				\item $\CC_1(V)$ is semisimple. In this case, the quasi-inverse of $\Om$ is Dong-Li-Mason's quotient module functor $L$ \cite[Section 7]{DLM98}, and
				$$(L\dashv \Om): \CC_{A(V)}\rightleftarrows \CC_1(V)$$
				are equivalence of categories. Moreover, in this case, if $\Phi^\L(\Om(M))\in C_1(V)$ for any object $M$ in $C_1(V)$, then there is a natural isomorphism
				$\Phi^\L|_{\CC_{A(V)}}\cong L|_{\CC_{A(V)}}.$
				In particular, $(\Phi^\L\dashv\Om)$ also restricts
				to an adjoint equivalence.
			\end{enumerate}
			In particular, when $V$ is strongly rational, then both conditions hold, and $\Phi^\L|_{\CC_{A(V)}}\cong L|_{\CC_{A(V)}}.$
		\end{prop}
		\begin{proof}
			If $V$ satisfies the strong identity condition, then $(\Phi^\L\dashv \Om): \Mod(A(V))\rightleftarrows \Adm(V)$ is an adjoint equivalence by Lemma~\ref{lm:sic}. Hence the restriction $\Om|_{\CC_1(V)}$ is fully faithful, and the restriction of functors give rise to an adjoint equivalence $(\Phi^\L\dashv \Om): \CC_{A(V)}\rightleftarrows \CC_1(V)$.
			
			Now assume $\CC_1(V)$ is semisimple. By Lemma~\ref{lm:ssordmod} any object $M$ in $\CC_1(V)$ is an ordinary $V$-module. Then by \cite[Theorem 7.2]{DLM98}, for any irreducible $V$-module $W$ in $\CC_1(V)$, one has $L(\Om(W))\cong W$. It follows from the semisimplicity of $\CC_1(V)$ that $L(\Om(M))\cong M$ for any object $M$ in $\CC_1(V)$. i.e., $(L\dashv \Om): \CC_{A(V)}\rightleftarrows \CC_1(V)$ are adjoint equivalence.
			
			Finally, if $\Phi^\L(\Om(M))\in C_1(V)$ for all the objects $M$ in $\CC_1(V)$, then both $\Phi^\L|_{\CC_{A(V)}}$
			and $L|_{\CC_{A(V)}}$ are therefore left adjoints to
			$\Om|_{\CC_1(V)}$. By uniqueness of left adjoints,
			there is a natural isomorphism
			$
			\Phi^\L|_{\CC_{A(V)}}\cong L|_{\CC_{A(V)}}.
			$
		\end{proof}
		
		\subsection{\texorpdfstring{The left $A(V)$-module $M^1\odot_{w,0}M^2$ from one-point restricted conformal blocks}{The left A(V)-module M\textasciicircum{}1 odot\_\{w,0\} M\textasciicircum{}2 from one-point restricted conformal blocks}} We recall the definition of (one-point) restricted conformal blocks on smooth genus-zero curves \cite{GLZ25,GLZ26}, then mildly generalize the construction left $A(V)$-module $M^1\odot M^2$ by the second-named author in \cite{Liu26}.
		In this subsection, we fix a point $w\in \P^1\bs \{ \infty, 0\}=\C^\times$.
		
		
		\begin{df}\cite{GLZ26,Liu26}\label{def:inftyrestrictedchiralLie}
			Consider the following subspace of the chiral Lie algebra $\mathcal{L}_{\P^1\bs\{\infty,w,0\}}(V)$:
			\begin{equation}\label{eq:defofinftyrestrictchiral}
				\begin{aligned}
					\mathcal{L}_{\P^1\bs\{w,0\}}(V)_{\prec 0}
					&=\spn\{a\o f(z)\in \mathcal{L}_{\P^1\bs\{\infty,w,0\}}(V)
					: a\in V,\ \ord_\infty (f(z))\succ 1-\wt a \}\\
					&=\spn\left\{a\o\frac{z^{n}}{(z-w)^m}\in \mathcal{L}_{\P^1\bs\{\infty,w,0\}}(V)
					: a\in V,\ n-m\prec \wt a-1 \right\},
				\end{aligned}
			\end{equation}
			where $\prec$ represents $\leq $ or $<$.
			
			$ \mathcal{L}_{\P^1\bs\{w,0\}}(V)_{\leq 0}$ is a Lie subalgebra of $ \mathcal{L}_{\P^1\bs\{\infty,w,0\}}(V)$, called the the {\em $\infty$-restricted chiral Lie algebra}. $\mathcal{L}_{\P^1\bs\{w,0\}}(V)_{<0}$ is an ideal of the Lie algebra $ \mathcal{L}_{\P^1\bs\{w,0\}}(V)_{\leq 0}$.
			
			The quotient Lie algebra $ \mathcal{L}_{\P^1\bs\{w,0\}}(V)_{\leq 0}/ \mathcal{L}_{\P^1\bs\{w,0\}}(V)_{<0}$ is isomorphic to the degree-zero part of Borcherds' Lie algebra $ L(V)_0=\spn\{a_{[\wt a-1]}:a\in V \}$.
		\end{df}
		
		\begin{df}\cite{GLZ26,Liu26}
			Let $M^1,M^2,M^3\in \Ord(V)$. Then $\Om((M^3)')$ is a left $A(V)$-module via $\theta_{A(V)}$ \eqref{eq:dualAVmod}, which is also module over the Lie algebra $ \mathcal{L}_{\P^1\bs\{w,0\}}(V)_{\leq 0}$ via $\rho_\infty$ \eqref{eq:inftyaction}.
			The space of $\infty$-restricted conformal blocks associated to the data $(\P^1, (\infty, w,0),(\Om((M^3)'), M^1, M^2))$ is defined by
			\begin{equation}\label{eq:defrescfb}
				\begin{aligned}
					&\mathscr{C}_{\mathrm{res}}\left(\Om((M^3)')\o M^1\o M^2\right)_{(\P^1,\infty,w,0)}=\left(\frac{\Om((M^3)')\o M^1\o M^2}{ \mathcal{L}_{\P^1\bs\{w,0\}}(V)_{\leq 0}.(\Om((M^3)')\o M^1\o M^2) }\right)^*\\
					&\cong \Hom_{L(V)_0}\left(\frac{M^1\o M^2}{ \mathcal{L}_{\P^1\bs\{w,0\}}(V)_{< 0}.(M^1\o M^2)}, \Om((M^3)')^\ast \right).
				\end{aligned}
			\end{equation}
			The action of $ \mathcal{L}_{\P^1\bs\{w,0\}}(V)_{<0}$ on $M^1\o M^2$ is given by the tensor action of $\rho_w$ and $\rho_0$ as below:
			\begin{equation}\label{odotaction}
				\begin{aligned}
					\rho_{w,0}: \mathcal{L}_{\P^1\bs\{w,0\}}(V)_{< 0}&\ra \gl(M^1\o M^2),\\
					\rho_{w,0} \left(a\o\frac{z^n}{(z-w)^m}\right)(v^1\o v^2)&=\rho_w \left(a\o\frac{z^n}{(z-w)^m}\right)v^1\o v^2+v^1\o \rho_0 \left(a\o\frac{z^n}{(z-w)^m}\right)v^2,
				\end{aligned}
			\end{equation}
			where $\rho_w$ and $\rho_0$ are given by \eqref{eq:1action} and \eqref{eq:0action}, respectively.
		\end{df}
		
		Restricted conformal blocks extend uniquely to (unrestricted) conformal blocks after replacing the restricted $A(V)$-module by its associated generalized Verma module; see \cite[Theorem~5.5]{Liu26} and \cite[Theorems~5.18 and~5.19]{GLZ26}.
		
		\begin{lm}\label{lm:rescfbandcfb}
			Let $M^1,M^2,M^3$ be ordinary $V$-modules, and let $S=\Om((M^3)')$, equipped with the natural left $A(V)$-module structure
			inherited from the contragredient module $(M^3)'$ \eqref{eq:dualAVmod}. Then the restriction of conformal blocks induces an isomorphism
			\[
			\mathscr{C}(\Phi^\L(S)\o M^1\o M^2)_{(\P^1,(\infty, w,0))}\cong \mathscr{C}_{\mathrm{res}}(S\o M^1\o M^2)_{(\P^1,(\infty, w,0))},\quad \varphi\mapsto \varphi|_{S\o M^1\o M^2}.
			\]
			In particular, if the counit
			$\varepsilon_{(M^3)'}:
			\Phi^\L(\Om((M^3)'))\to (M^3)'$
			is an isomorphism, then the restriction also induces
			an isomorphism
			\begin{equation}\label{eq:cfb-restriction}
				\mathscr{C}((M^3)'\o M^1\o M^2)_{(\P^1,(\infty, w,0))}\cong \mathscr{C}_{\mathrm{res}}(\Om((M^3)')\o M^1\o M^2)_{(\P^1,(\infty, w,0))},\quad \varphi\mapsto \varphi|_{\Om((M^3)')\o M^1\o M^2}.
			\end{equation}
		\end{lm}
		
		\begin{df}\label{df:odotw}
			With the notations as above, define
			\[
			M^1\odot_{w,0} M^2:=\frac{M^1\o M^2}{ \mathcal{L}_{\P^1\bs\{w,0\}}(V)_{< 0}.(M^1\o M^2)}=\spn\{v^1\odot_w v^2: v^1\in M^1,\ v^2\in M^2\},
			\]
			where $v^1\odot_w v^2$ denotes $\overline{v^1\o v^2}$.
			
			$ M^1\odot_{w,0} M^2$ is spanned by symbols $v^1\odot_{w} v^2$, subject to the following relations:
			\begin{equation}\label{eq:actionfogeneralspann}
				\begin{aligned}
					\rho_w\left(a\o\frac{z^{\wt a-s}}{(z-w)^t}\right). v^1 \odot_w v^2
					&=-v^1\odot_w \rho_0\left(a\o \frac{z^{\wt a-s}}{(z-w)^t}\right). v^2 ,\quad \mathrm{equivalently,}\\
					\sum_{j\geq 0} \binom{\wt a-s}{j} w^{\wt a-s-j} a_{(j-t)}v^1\odot_w v^2
					&=-v^1\odot_w \sum_{j\geq 0}\binom{-t}{j} (-w)^{-t-j} a_{(\wt a-s+j)}v^2,
				\end{aligned}
			\end{equation}
			where $s,t\in \Z$ such that $s+t>1$.
			Moreover, $M^1\odot_{w,0} M^2$ is a $L(V)_0$-module with respect to the action
			\begin{equation}\label{LV0action}
				a_{[\wt a-1]}. (v^1\odot_w v^2)=\sum_{j\geq 0} \binom{\wt a-1}{j}w^{\wt a-1-j} a_{(j)}v^1\odot_w v^2+ v^1\odot_w o(a)v^2,
			\end{equation}
			where $o(a)=a_{(\wt a-1)}$. When $w=1$, $M^1\odot _{1,0} M^2$ is the contracted tensor product $M^1\odot M^2$ in \cite[Definition 4.3]{Liu26}.
		\end{df}
		
		The following proposition follows from \cite[Theorem 4.8]{Liu26} by the change of coordinate $z=wx$. The proof is similar, we omit the details. 
		\begin{prop}\label{prop:AVaction}
			Assume either $M^1$ or $M^2$ is generated by its bottom degrees $\Om(M^1)$ or $\Om(M^2)$, respectively. Then the $L(V)_0$-action \eqref{LV0action} factors through $L(V)_0\twoheadrightarrow A(V),\, a_{[\wt a-1]}\mapsto [a]$, which makes $M^1\odot_{w,0} M^2$ a left $A(V)$-module with respect to the following action:
			\begin{equation}\label{A(V)action}
				\begin{aligned}
					A(V)&\times (M^1\odot_{w,0} M^2)\ra M^1\odot_{w,0} M^2,\\
					[a].(v^1\odot_w v^2)
					&=\sum_{j\geq 0}\binom{\wt a-1}{j} w^{\wt a-1-j} a_{(j)}v^1\odot_w v^2+ v^1\odot_w o(a)v^2\\
					&=\rho_{w,0}(a\o z^{\wt a-1}).(v^1\odot_{w}v^2).
				\end{aligned}
			\end{equation}
		\end{prop}
		Furthermore, it follows from \cite[Remark 4.7]{Liu26} that
		\begin{equation}\label{AVactionq}
			[a].(v^1\odot_w v^2)=\rho_{w,0}\left(a\o \frac{z^{\wt a-1+q}}{(z-w)^q}\right). (v^1\odot_w v^2),\quad q\in \Z.
		\end{equation}
		With the definition of $M^1\odot_{w,0} M^2$ given by Definition~\ref{df:odotw}, it follows from \eqref{eq:defrescfb} that
		\[
		\mathscr{C}_{\mathrm{res}}\left(\Om((M^3)')\o M^1\o M^2\right)_{(\P^1,\infty,w,0)}\cong \Hom_{A(V)}\left(M^1\odot_{w,0} M^2,\Om((M^3)')^\ast\right).
		\]
		
		Recall that $\Om((M^3)')$ is a natural right $A(V)$-module via \eqref{eq:OmM'naturalrightmod}. Hence $\Om((M^3)')^\ast$ is a natural left $A(V)$-module via $A(V)\times \Om((M^3)')^\ast\to \Om((M^3)')^\ast,\ ([a],u)\mapsto [a].u$, where
		\[
		\<[a].u, f\>:=\<u,f.[a]\>,\quad u\in \Om((M^3)')^\ast,\ f\in \Om((M^3)'),\ [a]\in A(V).
		\]
		Since we do not assume $V$ is strongly rational, it is a subtle question whether $\Om((M^3)')^\ast$ is isomorphic to $\Om(M^3)$ as a left $A(V)$-module. Fortunately, the sufficient conditions in Proposition~\ref{prop:sufffullyfaithful} are also sufficient here.
		\begin{lm}\label{lm:Omdualiso}
			Let $M\in \CC_1(V)$ be an ordinary module. Then the natural left $A(V)$-module $\Om(M')^\ast$ is isomorphic to the bottom-degree left $A(V)$-module $\Om(M)$ if one of the following conditions hold:
			\begin{enumerate}
				\item $\CC_1(V)$ is semisimple.
				\item $V$ satisfies the strong identity condition.
			\end{enumerate}
			In particular, when $V$ is strongly rational, then both condition are satisfied.
		\end{lm}
		\begin{proof}
			Assume $\CC_1(V)$ is semisimple, then $M=\bigoplus_i W^i$ with $W^i$ being irreducible $C_1$-cofinite ordinary modules. Since $\Om(M)=\bigoplus_{i} \Om(W^i)$, it suffices to show $\Om(W)\cong \Om(W')^\ast$ for irreducible ordinary module $W\in \CC_1(V)$. By Definition~\ref{df:2.1}, $W$ is admissible and $\Om(W)=W(0)$, which is finite dimensional. Then $W'=\bigoplus_{n=0}^\infty W(n)^\ast$ is also irreducible \cite{FHL93}, with $\Om(W')=W(0)^\ast$. Hence there is a natural isormophism $\Om(W')^\ast=(W(0)^\ast)^\ast \cong W(0)=\Om(W)$.
			
			Now assume $V$ satisfies the strong identity condition. Then by Lemma~\ref{lm:sic} and Proposition~\ref{prop:sufffullyfaithful}, there is an adjoint equivalence $(\Phi^\L\dashv \Om): \Mod(A(V))\rightleftarrows \Adm(V)$. For $M\in \CC_1(V)\subset \Adm(V)$, since $M\cong \Phi^\L(\Om(M))$, $M'\cong \Phi^\L (\Om(M'))$, and $(M')'\cong M$ as $M$ is ordinary, the natural pairing $\<\cdot,\cdot\>: M\times M'\to \C$ restricts to a non-degenerate paring $\<\cdot,\cdot\>: \Om(M)\times \Om(M')\to \C$. On the other hand, since
			\begin{equation}\label{eq:dualpairing}
				\<f,a_{(-1)}v\>=\<\theta(a_{(-1)})f,v\>=0,\quad a\in V_+,\ v\in M,\  f\in \Om(M'),
			\end{equation}
			we have $\Om(M')\subseteq (M/C_1(M))^\ast$, which is finite dimensional. Hence $\Om(M)\cong \Om(M')^\ast$ via the perfect pairing.
		\end{proof}
		\begin{coro}\label{coro:3.10}
			Assume $\CC_1(V)$ is semisimple or $V$ satisfies the strong identity condition, then for any ordinary modules $M^1,M^2,M^3\in \CC_1(V)$, we have
			\begin{equation}\label{eq:rescfbandhom}
				\mathscr{C}_{\mathrm{res}}\left(\Om((M^3)')\o M^1\o M^2\right)_{(\P^1,\infty,w,0)}\cong \Hom_{A(V)}\left(M^1\odot_{w,0} M^2,\Om(M^3)\right).
			\end{equation}
			The corresponding elements $\varphi_{\mathrm{res}}$ and $f$ satisfy
			\begin{equation}\label{f}
				\braket*{(u^3)'}{f(v^1\odot_wv^2)}
				=\braket*{\varphi_{\mathrm{res}}}{(u^3)'\otimes v^1\otimes v^2},
				\quad (u^3)'\in\Om((M^3)').
			\end{equation}
		\end{coro}
		
		
		
		
		\subsection[ P(w)-tensor product ]{\texorpdfstring{The isomorphism $\eta_w: M^1\odot_{w,0}M^2\to \Om(M^1\boxtimes_{P(w)}M^2)$}{The isomorphism eta\_w: M\textasciicircum{}1 odot\_\{w,0\} M\textasciicircum{}2 -> Omega(M\textasciicircum{}1 boxtimes\_\{P(w)\} M\textasciicircum{}2)}}\label{sec:eta}
		In this subsection, we fix a point $w\in \P^1\bs\{\infty, 0\}=\C^\times$. We show that under certain conditions for the category $\CC_1(V)$, the $\odot$-tensor product in Definition~\ref{df:odotw} captures all the information of the bottom degree Zhu algebra module of Huang-Lepowsky's $P(w)$-tensor product.
		
		\subsubsection{\texorpdfstring{Key assumptions for $\CC_1(V)$}{Key assumptions for C\_1(V)}}
		In order to describe the top degree Zhu algebra module of $M^1\boxtimes_{P(w)}M^2$, we need some further conditions on $\CC_1(V)$.
		
		
		\begin{prop}\label{prop:condforC1V}
			Assume $\CC_1(V)$ is semisimple, and $\Phi^\L (\Mod_{\fd}(A(V)))\ssq \CC_1(V)$, where $\Phi^\L$ is the generalized Verma module functor \eqref{Verma}. Then
			\begin{enumerate}
				\item $\CC_{A(V)}=\Mod_{\fd}(A(V))$, and $$(\Phi^\L \dashv \Om): \Mod_{\fd}(A(V))\rightleftarrows \CC_1(V)$$ are equivalence of categories;
				\item $\CC_1(V)\ssq \Ord(V)$ is closed under contragredient: $M'\in \CC_1(V)$ when $M\in \CC_1(V)$;
				\item For $M^1,M^2\in \CC_1(V)$, we have $M^1\odot_{w,0} M^2\in \CC_{A(V)}$.
			\end{enumerate}
		\end{prop}
		\begin{proof}
$(1)$. Since $\CC_1(V)$ is semisimple, by Proposition~\ref{prop:sufffullyfaithful} and Lemma~\ref{lm:ssordmod}, $\Om:\CC_1(V)\to \Mod(A(V))$ is fully faithful and any object in $\CC_1(V)$ is an ordinary $V$-module. Let $M\in \CC_1(V)$. Then $M=\bigoplus_i W^i$, with $W^i$ being irreducible for all $i$ and
			$M/C_1(M)\cong \bigoplus_i W^i/C_1(W)$ is finite dimensional. Therefore, the sum is finite and $\Om(M)=\bigoplus_{i=0}^n \Om(W^i)\in \Mod_{\fd}(A(V))$. Since $\Phi^\L(S)\in \CC_1(V)$ for any $S\in \Mod_\fd(A(V))$ by the assumption, and we have equivalence of categories
			\[(\Phi^\L\dashv \Om): \CC_{A(V)} \rightleftarrows \CC_1(V) ,\quad \CC_{A(V)}\subseteq \Mod_\fd(A(V)). \]
			It follows that $S\cong \Om(\Phi^\L(S))\in \CC_{A(V)}$, and so $\CC_{A(V)}=\Mod_\fd(A(V))$. 
			
$(2)$. Given $M\in \CC_1(V)$, it follows from \eqref{eq:dualpairing} that $\Om(M')\ssq (M/C_1(M))^\ast$ is finite dimensional. Then $\Om(M')\in \Mod_\fd(A(V))=\CC_{A(V)}$ and $M'\cong \Phi^\L(\Om(M'))\in \CC_1(V)$. 
			
$(3)$.	Let $M^1,M^2\in \CC_1(V)$ be irreducible ordinary $V$-modules. Choose homogeneous subspaces $U^1,U^2$ such that $M^i=U^i\op C_1(M^i)$ and $M^i(0)\ssq U^i$, $i=1,2$.
			In order to show $(3)$, we claim that the canonical map
			\begin{equation}\label{eq:odotsurj}
				U^1\otimes U^2\to M^1\odot_{w,0} M^2,\quad u^1\otimes u^2\mapsto u^1\odot_w u^2
			\end{equation}
			is surjective. In other words, $M^1\odot_{w,0}M^2=U^1\odot_{w,0} U^2=\spn\{u^1\odot_w u^2: u^i\in U^i,\ i=1,2 \}$.
			
			Let $u\in M^1(p)$, $v\in M^2(q)$, and $a\in V_d$, with $p,q\geq 0$ and $d\geq 1$. For homogeneous elements $u^i\in M^i$, we define $\deg(u^1\odot_w u^2):=\deg(u^1)+\deg(u^2)$.
			It follows from \eqref{eq:actionfogeneralspann}, with $(s,t)=(1,1)$ that
			\[
			(a_{(-1)}u)\odot_w v=-\sum_{j=1}^{d-1} \binom{d-1}{j} w^{-j} (a_{(j-1)}u)\odot_w v+\sum_{j\geq 0} w^{-d-j} u\odot_w a_{(d-1+j)}v.
			\]
			Note that $\deg ((a_{(-1)}u)\odot_w v)=d+p+q$, and $\deg ( (a_{(j-1)}u)\odot_w v)=d-j+p+q< d+p+q$ for all $j\geq 1$, and $\deg(u\odot_w a_{(d-1+j)}v)=-j+p+q< d+p+q$ for all $j\geq 0$. Then by induction on the degree $\deg((a_{(-1)}u)\odot_w v)$, we have $C_1(M^1)\odot_{w,0} M^2\ssq U^1\odot_{w,0} U^2$.
			
			On the other hand, apply \eqref{eq:actionfogeneralspann} with $(s,t)=(d+1,0)$, we have
			\[
			u\odot_w (a_{(-1)}v)=
			-\sum_{j\ge0}(-1)^j w^{-1-j}
			(a_{(j)}u)\odot_w v,
			\]
			where $\deg ((a_{(j)}u)\odot_w v)=d-j-1+p+q<d+p+q=\deg ( u\odot_w (a_{(-1)}v))$, for all $j\geq 0$. Then by induction on degree again, we have $M^1\odot_{w,0} C_1(M^2)\ssq U^1\odot_{w,0}U^2$. This proves the surjectivity of \eqref{eq:odotsurj}. Since $M^1,M^2$ are $C_1$-cofinite, we may choose $U^1,U^2$ to be finite dimensional. Then $\dim (M^1\odot_{w,0}M^2)\leq \dim U^1\cdot \dim U^2<\infty$. Hence $M^1\odot_{w,0} M^2\in \Mod_\fd(A(V))=\CC_{A(V)}$. The general case follows since $\odot_{w,0}$ commutes with finite direct sums.
		\end{proof}
		
		\begin{example}\label{ex:levels}
			The following VOAs satisfy the conditions in Proposition~\ref{prop:condforC1V}, namely, $\CC_1(V)$ is semisimple and $\Phi^\L (\Mod_{\fd}(A(V)))\ssq \CC_1(V)$.
			\begin{enumerate}
				\item Any strongly rational VOA $V$, in which case $\CC_1(V)=\Ord(V)$. In particular, $V=L_k(\g)$ an affine VOA at positive integral level $k\in \Z_{>0}$ \cite{FZ92,DLM97}.
				\item $V=L_k(\g)$ an admissible level affine VOA, where
				$$k+h^\vee=\frac{p}{q},\qquad p\geq \begin{cases}
					h^\vee &\mathrm{if}\ (r^\vee,q)=1,\\
					h&\mathrm{if}\ (r^\vee,q)=r^\vee.
				\end{cases}$$
				See subsection~\ref{sec:2.2.2}.
				\item $V=V^k(\g)=L_k(\g)$ a generic level affine VOA, where $k+h^\vee\notin \Q$. See subsection~\ref{sec:2.2.3}.
				\item
				$V=V^k(\mathfrak{sl}_2)=L_k(\mathfrak{sl}_2)$, where
				\[
				k=-2\pm\frac{1}{p},
				\qquad p\in\mathbb Z_{>0};
				\]
				see \cite[Proposition~3.1]{AM26}.
				\item
				$V=L_k(\g)$, where $(\g,k)$ is one of the following
				pairs:
				\[
				\setlength{\arraycolsep}{4pt}
				\begin{array}{c|c|c|c|c|c|c|c|c|c}
					\g
					& A_\ell & A_{2\ell-1} & B_\ell & C_\ell
					& D_\ell & D_{2\ell-1} & E_6 & E_7 & F_4 \\ \hline
					k
					& -1 & -\ell & -2 & -1-\ell/2
					& -2 & -2\ell+3 & -4 & -6 & -3 \\ \hline
					\text{range}
					& \ell\geq2 & \ell\geq2 & \ell\geq2 & \ell\geq2
					& \ell\geq3 & \ell\geq3 & & &
				\end{array}
				\]
				See \cite[Theorems~4.2.1 and~4.2.3]{CY21}.
			\end{enumerate}
		\end{example}
		
		\begin{remark}
			Even if $\CC_1(V)$ is semisimple, we do not have $\Phi^\L(\Mod_\fd(A(V)))\ssq \CC_1(V)$ in general. For example, let $V=L(c,0)$ be the Virasoro VOA with generic central charge in the subsection~\ref{sec:2.2.4}. Then $A(V)\cong \C[x]$ \cite{W93,DMZ94}.
			For $S=\C v_{c,0}\in \Mod_\fd(A(V))$, we have $\Phi^\L(S)\cong M(c,0)$, the Verma module over the Virasoro algebra. It is not $C_1$-cofinite since $L(-1)^mv_{c,0}=\omega_{(0)}^m v_{c,0}\neq 0$ in $M(c,0)$ for all $m\geq 1$.
		\end{remark}
		
		\subsubsection{\texorpdfstring{Universal property of $\boxtimes_{P(w)}$}{Universal property of boxtimes\_\{P(w)\}}}
		We recall the definition of $P(w)$-tensor product of ordinary modules by Huang-Lepowsky \cite{HL1,HL2}.
		Let $M^1,M^2\in \CC_1(V)$. A {\em $P(w)$-tensor product} is a pair $\left(M^1\bt_{P(w)}M^2,\ \bt_{P(w)}\right)$, consisting of a $V$-module $M^1\bt_{P(w)}M^2\in \CC_1(V)$ together with a $P(w)$-intertwining map, see Definition~\ref{df:P(w)IO},
		\[
		\bt_{P(w)}: M^1\o_\C M^2\ra \overline{M^1\bt_{P(w)}M^2},\quad v^1\o v^2\mapsto v^1\bt_{P(w)}v^2,
		\]
		satisfying the following universal property: for any $M^3\in \CC_1(V)$, and any $P(w)$-intertwining map $F\in \mathcal{M}[P(w)]^{M^3}_{M^1M^2}$, there exists a unique homomorphism of $V$-modules $g: M^1\bt_{P(w)}M^2\ra M^3$ such that $\bar{g}\circ \bt_{P(w)}=F$.
		
		It was proved by Huang \cite[Theorem~1.2]{H25} that the category $\CC_1(V)$
		admits a natural vertex tensor category structure. In particular, $M^1\boxtimes_{P(w)} M^2\in \CC_1(V)$ for $M^1,M^2\in \CC_1(V)$. Then the universal property of the $P(w)$-tensor product states that, for every
		$M^3\in\CC_1(V)$, there is a natural isomorphism
		\begin{equation}\label{eq:P(w)IOhom}
			\mathcal{M}[P(w)]^{M^3}_{M^1M^2}
			\xrightarrow{\ \simeq\ }
			\Hom_{\CC_1(V)}
			(M^1\boxtimes_{P(w)}M^2,M^3),
			\qquad F\mapsto g,
		\end{equation}
		where $g$ is the unique $V$-module homomorphism
		satisfying $\bar g\circ\boxtimes_{P(w)}=F$.
		
		Moreover, $\bt_{P(w)}$ is a bifunctor on $\CC_1(V)$. Given homomorphisms $g_i\in \Hom_V(M^i, N^i)$, $i=1,2$, then $g_1\bt_{P(w)}g_2$ is the unique homomorphism in $\Hom_V(M^1\bt_{P(w)} M^2, N^1\bt_{P(w)}N^2 )$ satisfying
		\[
		\overline{g_1\bt_{P(w)}g_2}\left(v^1\bt_{P(w)}v^2\right)=g_1v^1\bt_{P(w)}g_2v^2,\quad v^1\in M^1,\ v^2\in M^2.
		\]
		
		\subsubsection{\texorpdfstring{The top degree of $M^1\boxtimes_{P(w)}M^2$}{The top degree of M\textasciicircum{}1 boxtimes\_\{P(w)\} M\textasciicircum{}2}}
		Now assume the category $\CC_1(V)$ is semisimple and $\Phi^\L(\Mod_\fd(A(V)))\ssq \CC_1(V)$. We fix a set $\mathscr W$ of representatives of the irreducible objects in $\CC_1(V)$.
From the proof of Proposition~\ref{prop:condforC1V}, every $M\in\CC_1(V)$ is a finite direct sum of irreducible ordinary modules. Since an irreducible module $W\in \mathscr W$ is admissible $W=\bigoplus_{n=0}^\infty W(n)$, with $ \Om(W)=W(0)$, we may write
		\begin{equation}\label{eq:canonical-grading}
			M=\bigoplus_{W\in \mathscr W} n_W W,\quad M(n)=\bigoplus_{W\in \mathscr W} n_W W(n),\quad n\geq 0.
		\end{equation}
		Then $M$ is admissible, with $\Om(M)=\bigoplus_{W\in \mathscr W} n_W \Om(W)=M(0)$.
		Write
		\begin{equation}
			\pi_0:\overline M=\prod_{n=0}^\infty M(n)\to\Om(M)=M(0)
		\end{equation}
		for the bottom-degree projection. For a $V$-module homomorphism $f:M\to N$ and $y\in\C$, it satisfies
		\begin{equation}\label{eq:pi0-properties}
			\pi_0\bar f=\Om(f)\pi_0,\qquad
			\pi_0e^{yL(0)}=e^{y[\om]}\pi_0,\qquad
			\pi_0e^{yL(-1)}=\pi_0.
		\end{equation}
		Indeed, module homomorphisms preserve this grading, $L(0)$ acts on $\Om(M)$ by $[\om]$, and $L(-1)$ raises degree.
		
		\begin{thm}\label{thm:eta-explicit}
			Assume $\CC_1(V)$ is semisimple and $\Phi^\L(\Mod_\fd(A(V)))\ssq \CC_1(V)$.
			Let $M^1,M^2\in \CC_1(V)$. Then the canonical map
			\begin{equation}\label{eq:eta-explicit}
				\eta_w=\eta_w^{M^1,M^2}: M^1\odot_{w,0}M^2\ra \Om\left(M^1\bt_{P(w)}M^2\right),\quad v^1\odot_w v^2\mapsto \pi_0\left(v^1\bt_{P(w)}v^2\right),
			\end{equation}
			is a well-defined isomorphism of left $A(V)$-modules, which is natural in $M^1$ and $M^2$.
		\end{thm}
		\begin{proof}
			By Proposition~\ref{prop:condforC1V}, $(\Phi^\L\dashv \Om): \CC_{A(V)}=\Mod_\fd(A(V))\rightleftarrows \CC_1(V)$ are equivalence of categories, and $M^1\odot_{w,0}M^2\in \CC_{A(V)}$ for any $M^1,M^2\in \CC_1(V)$.
			Let $M^3\in \CC_1(V)$, then by our previous results, there is a commutative diagram of isomorphisms
			\begin{equation}\label{diagram}
				\begin{tikzcd}
					\mathcal M[P(w)]_{M^1M^2}^{M^3}\arrow[r,"\text{Prop.}~\ref{prop2.7}","\simeq"'] \arrow[d,"\eqref{eq:P(w)IOhom}","\simeq"']& \mathscr{C}((M^3)' \o M^1\o M^2)_{(\P^1,\infty,w,0)} \arrow[d,"\text{Lemma}~\ref{lm:rescfbandcfb}","\simeq"']\\
					\Hom_{\CC_1(V)}(M^1\bt_{P(w)}M^2, M^3)\arrow[r,dashed,"\beta_{M^3}","\simeq"'] \arrow[d,"\text{Prop.}~\ref{prop:condforC1V}","\simeq"']& \mathscr{C}_{\mathrm{res}}(\Om((M^3)') \o M^1\o M^2)_{(\P^1,\infty,w,0)}\arrow[d,"\text{Corollary}~\ref{coro:3.10}","\simeq"']\\
					\Hom_{A(V)}\left(\Om(M^1\bt_{P(w)}M^2), \Om(M^3)\right)\arrow[r,dashed,"\beta^\res_{M^3}","\simeq"']& \Hom_{A(V)}\left(M^1\odot_{w,0}M^2, \Om(M^3)\right),
				\end{tikzcd}
			\end{equation}
			Note that all the isomorphisms in diagram \eqref{diagram} are natural in
			the input $M^3$. Denote the middle horizontal
			isomorphism by
			\[
			\beta_{M^3}:
			\Hom_{\CC_1(V)}\left(M^1\bt_{P(w)} M^2,M^3\right)
			\xrightarrow{\simeq}
			\mathscr{C}_{\mathrm{res}}(\Om((M^3)') \o M^1\o M^2)_{(\P^1,\infty,w,0)}.
			\]
			Then by \eqref{eq:P(w)IOhom}, Proposition~\ref{prop2.7}, and Lemma~\ref{lm:rescfbandcfb}, given $g\in \Hom_{\CC_1(V)}(M^1\bt_{P(w)}M^2, M^3)$ and $v^1\in M^1$, $v^2\in M^2$, and $(u^3)'\in \Om((M^3)')$, we have
			\begin{equation}\label{eq:betaM3}
				\begin{aligned}
					\braket*{\b_{M^3}(g)}{(u^3)'\o v^1\o v^2}&=\braket*{\varphi_{I,p}}{(u^3)'\o v^1\o v^2}=\braket*{(u^3)'}{I(v^1,e^{l_p(w)})v^2}\\
					&=\braket*{(u^3)'}{F(v^1\o v^2)}=\braket*{(u^3)'}{\bar{g}(v^1\bt_{P(w)}v^2)}\\
					&=\braket*{(u^3)'}{\Om(g)(\pi_0(v^1\bt_{P(w)} v^2))},
				\end{aligned}
			\end{equation}
			where the last equality follows from the fact that $g$ preserves the grading subspaces. Denote the bottom horizontal isormorphism by
			\[
			\b_{M^3}^\res: \Hom_{A(V)}\left(\Om(M^1\bt_{P(w)}M^2), \Om(M^3)\right)\xrightarrow{\simeq} \Hom_{A(V)}\left(M^1\odot_{w,0}M^2, \Om(M^3)\right).
			\]
			Let $M^3=M^1\bt_{P(w)} M^2\in \CC_1(V)$, define
			\[
			\eta_w:=\beta^\res_{M^1\bt_{P(w)}M^2}(\Id_{\Om(M^1\bt_{P(w)} M^2)})\in\Hom_{A(V)}\left(M^1\odot_{w,0}M^2,\Om(M^1\bt_{P(w)} M^2)\right).
			\]
			Since $\Om$ is an equivalence of category, then by \eqref{eq:betaM3} we have
			\begin{align*}
				\braket*{(u^3)'}{\eta_w(v^1\odot_w v^2)}&= \braket*{(u^3)'}{\b^\res_{M^1\bt_{P(w)}M^2}(\Om(\Id_{M^1\bt_{P(w)}M^2}))(v^1\odot_w v^2)}\\
				&= \braket*{\b_{M^1\bt_{P(w)}M^2}(\Id_{M^1\bt_{P(w)}M^2})}{(u^3)'\o v^1\o v^2}\\
				&=\braket*{(u^3)'}{\pi_0(v^1\bt_{P(w)} v^2)},
			\end{align*}
			where $(u^3)'\in \Om((M^3)')$ is arbitrary. This shows
			\[
			\eta_w(v^1\odot_w v^2)
			=
			\pi_0(v^1\bt_{P(w)}v^2),\quad v^1\in M^1,\ v^2\in M^2.
			\]
			Thus, \eqref{eq:eta-explicit} is a
			well-defined homomorphism of left $A(V)$-modules.
			
			Furthermore, for any $f\in\Hom_{A(V)}(\Om(M^1\bt_{P(w)}M^2),\Om(M^3))$, by the full faithfulness
			of $\Om$, there exists a unique $V$-module homomorphism
			$g:M^1\bt_{P(w)}M^2\to M^3$ such that $\Om(g)=f$. By \eqref{eq:betaM3}, we have
			\begin{align*}
				\braket*{(u^3)'}{\b_{M^3}^\res(f)(v^1\odot_w v^2)}&=\braket*{\b_{M^3}(g)}{(u^3)'\o v^1\o v^2}=\braket*{(u^3)'}{\Om(g)(\pi_0(v^1\bt_{P(w)}v^2))}\\
				&=\braket*{(u^3)'}{(f\circ \eta_w)(v^1\odot_w v^2))}.
			\end{align*}
			This shows $\b^\res_{M^3}(f)=f\circ \eta_w$. In other words, $\beta^\res_{M^3}=\eta_w^\ast$ is an isomorphism for every object $S=\Om(M^3)$ in the category $\CC_{A(V)}$. Then by the Yoneda lemma, $\eta_w$ is also an isomorphism.
			
			Finally, let $f_i:M^i\to N^i$ be morphisms in $\CC_1(V)$, $i=1,2$.
			Since both $-\bt_{P(w)}-$ and $-\odot_{w,0}-$ are natural in the two arguments, by adopting a similar argument as above, we have
			\[
			\Om(f_1\bt_{P(w)}f_2)\circ\eta_w^{M^1,M^2}
			=
			\eta_w^{N^1,N^2}\circ(f_1\odot_{w,0}f_2).
			\]
			Hence $\eta_w=\eta_w^{M^1,M^2}$ is natural in $M^1$ and $M^2$.
		\end{proof}
		
		It follows from Lemma~\ref{lm2.6} that $\mathcal M[P(w)]_{M^1M^2}^{M^3}\cong \mathcal M[P(1)]_{M^1M^2}^{M^3}$. Moreover, by \cite[Section~12.2]{HLZ}, one can take
		$\bt=\bt_{P(1)}$ as the tensor product bifunctor of the braided tensor category $\CC_1(V)$.
		
		In particular, Theorem~\ref{thm:eta-explicit} and Proposition~\ref{prop:condforC1V} give the natural isomorphism
		\begin{equation}\label{eq:chain}
			\begin{aligned}
				\Hom_V\left(M^1\bt_{P(w)}M^2,M^3\right)
				&\xrightarrow{\ \simeq\ }\Hom_{A(V)}\left(M^1\odot_{w,0}M^2,\Om(M^3)\right),\\
				g&\longmapsto\Om(g)\circ\eta_w,
			\end{aligned}
		\end{equation}
		for $M^3\in\CC_1(V)$. This isomorphism will be used later.
		
		Moreover, semisimplicity gives the decomposition
		\begin{equation}\label{odotiso}
			M^1\odot_{w,0}M^2\cong
			\bigoplus_{W\in\mathscr W}W(0)\otimes_\C
			I\fusion{M^1}{M^2}{W}^{\ast},
		\end{equation}
where only finitely many summands are nonzero; see also \cite[Section~6]{Liu26}.
		
		For simplicity, we denote $M^1\odot_{1,0} M^2$ by $M^1\odot M^2$, and $v^1\odot_{1} v^2$ by $v^1\odot v^2$ as in \cite[Section 4]{Liu26}. Then $$M^1\odot M^2=\spn\{v^1\odot v^2: v^i\in M^i,\ i=1,2 \},$$
		subjects to the following relations for $s+t>1$ and $a\in V$:
		\begin{equation}\label{eq:finalodotrel}
			\sum_{j\geq 0} \binom{\wt a-s}{j} \left(a_{(j-t)}v^1\right)\odot v^2
			=-\sum_{j\geq 0}\binom{-t}{j} (-1)^{-t-j} v^1\odot \left(a_{(\wt a-s+j)}v^2\right).
		\end{equation}
		
		
		
		
		
		\section[Left A(V)-module tensor product and point-shifting isomorphism]{\texorpdfstring{Construction of $R$-matrices on $\Mod_{\mathrm{fd}}(A(V))$ from restricted conformal blocks}{Construction of R-matrices on Mod\_fd(A(V)) from restricted conformal blocks}}\label{sec:odot}
		Throughout this section, we assume that $\CC_1(V)$ is semisimple and $\Phi^\L(\Mod_{\mathrm{fd}}(A(V)))\ssq \CC_1(V)$.
Using the point-shifting isomorphism $R_w$ and the parallel transport of the restricted conformal blocks, we construct the isomorphisms $R_\pm$ in Theorem~\ref{thm:Rmatrix}, and show that $R_-$ is the restriction of the braiding isomorphism of Huang-Lepowsky-Zhang to the bottom degrees, see Theorem~\ref{thm:braided-tensor-structure}.
		
		\subsection[Isomorphism of intertwining-operator spaces and point shifting]{\texorpdfstring{The isomorphism $I\fusion{M^1}{M^2}{M^3}\cong I\fusion{M^2}{M^1}{M^3}$ and the points shifting $(w,0)\mapsto (0,-w)$}{The isomorphism I(M\textasciicircum{}3; M\textasciicircum{}1, M\textasciicircum{}2) isomorphic to I(M\textasciicircum{}3; M\textasciicircum{}2, M\textasciicircum{}1) and the points shifting (w,0) -> (0,-w)}}
		We first recall some of the constructions in \cite[Section 7]{HL2}.
		Let $M^1,M^2,M^3$ be ordinary $V$-modules, and $r\in \Z$.
		There is a linear isomorphism
		\begin{align*}\label{ss}
			\Om_r: I\fusion{M^1}{M^2}{M^3}&\cong I\fusion{M^2}{M^1}{M^3},\quad I\mapsto \Om_r(I),\\ \Om_r(I)(v^2,e^{\pi i}x)v^1&=e^{-xL(-1)}I(v^1,e^{2r\pi i}x)v^2,
		\end{align*}
		such that $\Om_{1-r}(\Om_r(I))=\Om_r(\Om_{1-r}(I))=I$, see \cite[Proposition 7.1]{HL2}. Our $\Om_r$ is the operator denoted by $\Om_{r-1}$ in \cite{HL2}. We denote
		\begin{equation}\label{Om0}
			\Om_0(I)=\tilde{I},\quad \mathrm{then}\quad \tilde{I}(v^2,e^{\pi i}x)v^1=e^{-xL(-1)} I(v^1,x)v^2.
		\end{equation}
		Fix a determination $l_p(w)$ of $\log(w)$ as in \eqref{logw}. We use $l_p(w)+\pi i$ as the determination of $\log(-w)$.
		
		This isomorphism \eqref{Om0} leads to the following commutative diagram of linear isomorphisms:
		\begin{equation}\label{diagrampart1}
			\begin{tikzcd}
				I\fusion{M^1}{M^2}{M^3}\arrow[r,"\Om_0","\simeq"'] \arrow[d,"\simeq"]& I\fusion{M^2}{M^1}{M^3}\arrow[d,"\simeq"]\\
				\mathcal{M}[P(w)]_{M^1M^2}^{M^3}\arrow[r,"\Om_0","\simeq"'] \arrow[d,"\simeq"]& \mathcal{M}[P(-w)]_{M^2M^1}^{M^3}\arrow[d,"\simeq"]\\
				\mathscr{C}((M^3)' \o M^1\o M^2)_{(\P^1,\infty,w,0)}\arrow[r,dashed,"\Om_0","\simeq"'] & \mathscr{C}((M^3)' \o M^2\o M^1)_{(\P^1,\infty,-w,0)},
			\end{tikzcd}
		\end{equation}
		where the vertical maps are given by the identification \eqref{eq:P(w)isocfb}. The correspondence of elements in these vector spaces are given by the following diagram:
		\[
		\begin{tikzcd}
			I\arrow[r,mapsto] \arrow[d,mapsto]& \tilde{I}\arrow[d,mapsto]\\
			I(\cdot, e^{l_p(w)})\arrow[r,mapsto]\arrow[d,mapsto]& \tilde{I}(\cdot, e^{\pi i+l_p(w)})\arrow[d,mapsto]\\
			\varphi \arrow[r,dashed,mapsto]& \widetilde{\varphi}
		\end{tikzcd}
		\]
		Then by \eqref{eq:defofvarphiI}, \eqref{Om0}, and the fact that $\<v'|L(-1)v\>=\<L(1)v'|v\>$, we have
		\begin{equation}\label{eq:varphitilde}
			\begin{aligned}
				\braket*{\widetilde{\varphi}}{(v^3)'\o v^2\o v^1}&=\braket*{(v^3)'}{ \tilde{I}(v^2, e^{\pi i+l_p(w)})v^1}=\braket*{(v^3)'}{ e^{-wL(-1)}I(v^1, e^{l_p(w)})v^2}\\
				&=\braket*{e^{-wL(1)}(v^3)'}{ I(v^1, e^{l_p(w)})v^2}\\
				&=\braket*{\varphi}{e^{-wL(1)}(v^3)'\o v^1\o v^2}
			\end{aligned}
		\end{equation}
		Note that $e^{-wL(1)}(v^3)'$ is a well-defined element in $(M^3)'$ since $L(1).M(n)\subset M(n-1)$.
		
		\begin{remark}
			Note that the element $\om \o 1=\om \o \frac{z^0}{(z-w)^0}\in \mathcal{L}_{\P^1\bs\{\infty,w,0\}}(V)$. Then by \eqref{eq:actionofchiralliealgebraelements},
			\begin{align*}
				&\rho_{\infty,w,0}(\om\o 1).((v^3)'\o v^1\o v^2)\\
				&=-L(1)(v^3)'\o v^1\o v^2+ (v^3)'\o L(-1)v^1\o v^2+(v^3)'\o v^1\o L(-1)v^2.
			\end{align*}
			In particular, in the space of three-point coinvariants, we have
			\begin{equation}\label{2.21}
				[L(1)(v^3)'\o v^1\o v^2]=[(v^3)'\o L(-1)v^1\o v^2]+[(v^3)'\o v^1\o L(-1)v^2],
			\end{equation}
			where $[\al]$ is the equivalence class of $\al\in (M^3)'\o M^1\o M^2$ in the quotient space, see Definition~\ref{def:cfb}. The operator $e^{-wL(1)}$ acts on the restricted dual, whereas $e^{-wL(-1)}v^i$ in general lies only in the completion of $M^i$. We therefore use the formula on the restricted dual in the following proposition.
		\end{remark}
		
		
		\begin{prop}\label{prop:Om0blocks}
			Let $M^1,M^2,M^3$ be ordinary $V$-modules. Then there exists a linear isomorphism between the space of conformal blocks
			\begin{equation}\label{IsoOm0}
				\begin{aligned}
					\Om_0: \mathscr{C}((M^3)' \o M^1\o M^2)_{(\P^1,\infty,w,0)}&\cong \mathscr{C}((M^3)' \o M^2\o M^1)_{(\P^1,\infty,-w,0)},\quad \varphi\mapsto \widetilde{\varphi},\\
					\braket*{\widetilde{\varphi}}{(v^3)'\o v^2\o v^1}&=\braket*{\varphi}{e^{-wL(1)}(v^3)'\o v^1\o v^2}.
				\end{aligned}
			\end{equation}
			More precisely, the following identities hold
			\begin{align*}
				&\braket*{\varphi}{\rho_{\infty}\left(a\o\frac{z^n}{(z-w)^m}\right). e^{-wL(1)}(v^3)'\o v^1\o v^2}=\braket*{\widetilde{\varphi}}{\rho_{\infty}\left(a\o\frac{u^{-m}}{(u+w)^{-n}}\right). (v^3)'\o v^2\o v^1},\\
				&\braket*{\varphi}{e^{-wL(1)}(v^3)'\o \rho_{w}\left(a\o\frac{z^n}{(z-w)^m}\right).v^1\o v^2}=\braket*{\widetilde{\varphi}}{(v^3)'\o v^2\o \rho_{0}\left(a\o\frac{u^{-m}}{(u+w)^{-n}}\right).v^1},\numberthis\label{relvarphiandtildevarphi}\\
				&\braket*{\varphi}{e^{-wL(1)}(v^3)'\o v^1\o \rho_{0}\left(a\o\frac{z^n}{(z-w)^m}\right).v^2}=\braket*{\widetilde{\varphi}}{(v^3)'\o \rho_{-w}\left(a\o\frac{u^{-m}}{(u+w)^{-n}}\right).v^2\o v^1},
			\end{align*}
			for any $m,n\in \Z$, where $u$ is the affine coordinate on $\P^1\bs\{\infty,-w,0\}$.
		\end{prop}
		\begin{proof}
			The isomorphism \eqref{IsoOm0} follows directly from the commutative diagram \eqref{diagrampart1} and our calculation \eqref{eq:varphitilde}. It also follows from the identities \eqref{relvarphiandtildevarphi}, which gives a more precise relation between $\varphi$ and $\widetilde{\varphi}$.
			
			We use the language of contour integration for residues to show \eqref{relvarphiandtildevarphi}. All contours are positively oriented, and $\int_C$ denotes the normalized integral $(2\pi i)^{-1}\oint_C$. For the fixed complex number $w\in \P^1\bs\{ 0,\infty\}$, consider the following contours in the affine coordinate $z$:
			\begin{itemize}
				\item $C_\infty(z)$ a circle in $z$ centered at $0$ such that $|z|>2|w|$,
				\item $C_0(z)$ a circle in $z$ centered at $0$ such that $|z|<|w|$,
				\item $C_w(z)$ a circle in $z$ centered at $w$ such that $|z-w|<|w|$.
			\end{itemize}
			
			\begin{figure}[h]
				\centering
				\begin{tikzpicture}[scale=1]
					
					\draw[blue, thick, ->]
					(0,0) circle (3.2);
					\node[blue] at (3.4,-1.6) {$C_\infty(z)$};
					
					\draw[red, thick, ->]
					(0,0) circle (0.5);
					\node[red] at (0.8,-0.5) {$C_0(z)$};
					
					\fill (0,0) circle (2pt);
					\node at (-0.15,-0.25) {$0$};
					
					\coordinate (w) at (1.1,1.0);
					\fill (w) circle (2pt);
					\node at (1.25,1.1) {$w$};
					
					\draw[black, thick, ->]
					(w) circle (0.5);
					\node at (1.8,0.4) {$C_w(z)$};
					
				\end{tikzpicture}
				\caption{Integration contours for the residue computation}
			\end{figure}
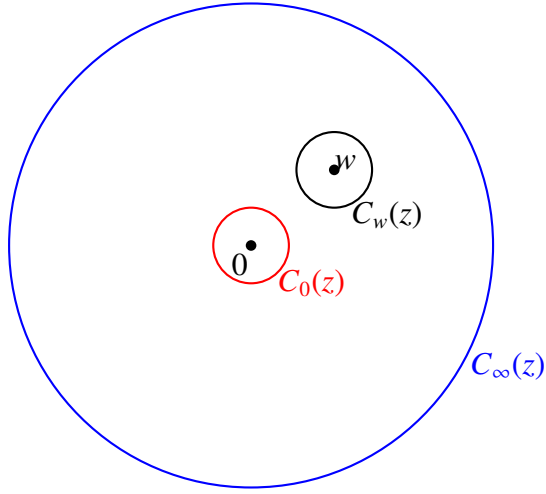
			
					
					
					
					
					
					
			
			Choose the determinations $l_p(w)$ and $l_p(w)+\pi i$ of $\log(w)$ and $\log(-w)$, respectively. Let $I(\cdot, e^{l_p(w)})$ and $\tilde I(\cdot, e^{l_p(w)+\pi i})$ be the $P(w)$ and $P(-w)$-intertwining operators corresponding to $\varphi$ and $\widetilde{\varphi}$, respectively. Then by \eqref{eq:inftyaction} and Corollary~\ref{coro2.8}, we have
			\begin{align*}
				&\braket*{\varphi}{\rho_{\infty}\left(a\o\frac{z^n}{(z-w)^m}\right). e^{-wL(1)}(v^3)'\o v^1\o v^2}\\
				&=\Res_z^{\mathrm{out}} \braket*{Y_{(M^3)'}(\vartheta(a),z^{-1})e^{-wL(1)}(v^3)'}{I(v^1,e^{l_p(w)})v^2}\iota_{|z|>|w|}\left(\frac{z^n}{(z-w)^m}\right) dz\\
				&=-\Res_z^{\mathrm{out}} \braket*{(v^3)'}{e^{-wL(-1)}Y_{M^3}(a,z)I(v^1,e^{l_p(w)})v^2 }\iota_{|z|>|w|}\left(\frac{z^n}{(z-w)^m}\right)dz\\
				&=-\int_{C_{\infty}(z)} \braket*{(v^3)'}{Y_{M^3}(a,z-w)e^{-wL(-1)}I(v^1,e^{l_p(w)})v^2 }\iota_{|z|>|w|}\left(\frac{z^n}{(z-w)^m}\right)dz\\
				&=-\int_{C_{\infty}(z)} \braket*{(v^3)'}{Y_{M^3}(a,z-w)\tilde{I}(v^2,e^{l_p(w)+\pi i})v^1 }\iota_{|z|>|w|}\left(\frac{z^n}{(z-w)^m}\right)dz\\
				&=-\int_{C_{\infty}(u)} \braket*{(v^3)'}{Y_{M^3}(a,u)\tilde{I}(v^2,e^{l_p(w)+\pi i})v^1 }\iota_{|u|>|-w|}\left(\frac{(u+w)^n}{u^m}\right)du\quad (z-w)\mapsto u\\
				&=\braket*{\widetilde{\varphi}}{\rho_{\infty}\left(a\o\frac{u^{-m}}{(u+w)^{-n}}\right). (v^3)'\o v^2\o v^1},
			\end{align*}
			where $C_\infty(u)$ is a circle of $u$ centered at $0$ such that $|u|>|-w|$. Indeed, the translated contour lies in $|u|>|w|$, since the radius of $C_\infty(z)$ is larger than $2|w|$. The matrix coefficient is rational in $u$, with possible poles only at $0$ and $-w$, so the translated contour can be deformed to $C_\infty(u)$ without changing the integral. Similarly, by \eqref{eq:1action} and Corollary~\ref{coro2.8}, we have
			\begin{align*}
				&\braket*{\varphi}{e^{-wL(1)}(v^3)'\o \rho_{w}\left(a\o\frac{z^n}{(z-w)^m}\right).v^1\o v^2}\\
				&=\Res_{z=w} \braket*{e^{-wL(1)}(v^3)'}{I(Y_{M^1}(a,z-w)v^1,e^{l_p(w)})v^2} \iota_{|w|>|z-w|}\left(\frac{z^n}{(z-w)^m}\right)dz\\
				&=\int_{C_w(z)} \braket*{(v^3)'}{\tilde{I}(v^2,e^{l_p(w)+\pi i})Y_{M^1}(a,z-w)v^1} \iota_{|w|>|z-w|}\left(\frac{z^n}{(z-w)^m}\right)dz\\
				&=\int_{C_0(u)} \braket*{(v^3)'}{\tilde{I}(v^2,e^{l_p(w)+\pi i})Y_{M^1}(a,u)v^1} \iota_{|-w|>|u|}\left(\frac{(u+w)^n}{u^m}\right)du\quad (z-w)\mapsto u\\
				&=\braket*{\widetilde{\varphi}}{(v^3)'\o v^2\o \rho_{0}\left(a\o\frac{u^{-m}}{(u+w)^{-n}}\right).v^1},
			\end{align*}
			where $C_0(u)$ is a circle of $u$ centered at $0$ such that $|u|<|-w|$. Finally, by \eqref{eq:0action} and Corollary~\ref{coro2.8}, we have
			\begin{align*}
				&\braket*{\varphi}{e^{-wL(1)}(v^3)'\o v^1\o \rho_{0}\left(a\o\frac{z^n}{(z-w)^m}\right).v^2}\\
				&=\Res_{z=0}\braket*{e^{-wL(1)}(v^3)'}{I(v^1, e^{l_p(w)}) Y_{M^2}(a,z)v^2 }\iota_{|z|<|w|}\left(\frac{z^n}{(z-w)^m}\right)dz\\
				&=\int_{C_0(z)}\braket*{(v^3)'}{\tilde{I}( Y_{M^2}(a,z)v^2, e^{l_p(w)+\pi i}) v^1 }\iota_{|z|<|w|}\left(\frac{z^n}{(z-w)^m}\right)dz\\
				&=\int_{C_{-w}(u)}\braket*{(v^3)'}{\tilde{I}( Y_{M^2}(a,u-(-w))v^2, e^{l_p(w)+\pi i}) v^1 }\iota_{|u+w|<|w|}\left(\frac{(u+w)^n}{u^m}\right)du\quad (z-w)\mapsto u\\
				&=\braket*{\widetilde{\varphi}}{(v^3)'\o \rho_{-w}\left(a\o\frac{u^{-m}}{(u+w)^{-n}}\right).v^2\o v^1},
			\end{align*}
			This proves \eqref{relvarphiandtildevarphi}.
		\end{proof}
		
		\begin{remark}\label{remark:Om0}
			In view of \eqref{relvarphiandtildevarphi}, the geometric effect of $\Om_0$-action on the space of three-point conformal blocks is given by shifting the marked points $(w,0)$ uniformly to $(0,-w)$, with attached modules unchanged. i.e.,
			\[ (\infty, w,0, (M^3)', M^1,M^2)\mapsto (\infty, 0,-w,(M^3)', M^1,M^2). \]
			See Figure~\ref{fig3} for an illustration.
			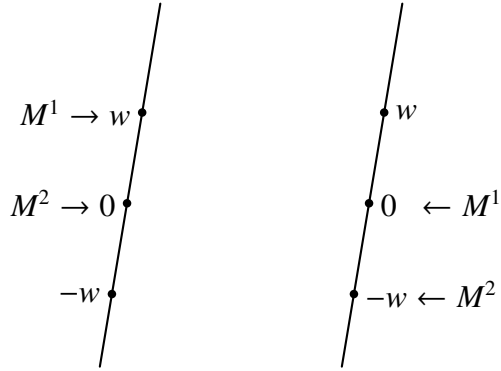
\begin{figure}[ht]
				\centering
				\begin{tikzpicture}[thick, line cap=round, line join=round, scale=0.8]
					\begin{scope}[xshift=7cm]
						
						\coordinate (Ltop) at (-1, 3);
						\coordinate (Lbot) at (-2, -3);
						
						\coordinate (Rtop) at ( 3, 3);
						\coordinate (Rbot) at ( 2, -3);
						
						\draw (Lbot) -- (Ltop);
						\draw (Rtop) -- (Rbot);
						
						\coordinate (L0) at ($(Ltop)!0.3!(Lbot)$);
						\coordinate (Lw1) at ($(Ltop)!0.55!(Lbot)$);
						\coordinate (Linf) at ($(Ltop)!0.8!(Lbot)$);
						\fill (L0) circle (2pt); \node[left] at (L0) {$M^1\ra w$};
						\fill (Lw1) circle (2pt); \node[left] at (Lw1) {$M^2\ra 0$};
						\fill (Linf) circle (2pt); \node[left] at (Linf) {$ -w$};
						
						\coordinate (Rinf) at ($(Rtop)!0.3!(Rbot)$);
						\coordinate (Rw2) at ($(Rtop)!0.55!(Rbot)$);
						\coordinate (R0) at ($(Rtop)!0.8!(Rbot)$);
						\fill (Rinf) circle (2pt); \node[right] at (Rinf) {$w$};
						\fill (Rw2) circle (2pt); \node[right] at (Rw2) {$0 \ \ \leftarrow M^1$};
						\fill (R0) circle (2pt); \node[right] at (R0) {$-w\leftarrow M^2$};
						
						
					\end{scope}
				\end{tikzpicture}
				\caption{Translation of the marked points $(w,0)$ to $(0,-w)$}\label{fig3}
			\end{figure}
			In fact, with the point at $\infty$ fixed, these three-point
			conformal blocks form a sheaf $\mathscr C$ over the configuration space
			$\mathrm{Conf}_2(\A^1)
			=\{(z_1,z_2)\in\A^1\times\A^1:z_1\neq z_2\}$,
			equipped with its standard flat connection
			\cite{FBZ04,DGT24}.
			The isomorphism $\Om_0$ is given by parallel transport along
			the straight line path $(w,0)\mapsto(0,-w)$,
			followed by reordering the two finite marked points
			and their attached modules.
		\end{remark}

			\subsection[A(V)-module isomorphism Rw]{\texorpdfstring{The $A(V)$-module isomorphism $R_w: M^2\odot_{-w,0} M^1\ra M^1\odot_{w,0} M^2$}{The A(V)-module isomorphism R\_w: M\textasciicircum{}2 odot\_\{-w,0\} M\textasciicircum{}1 -> M\textasciicircum{}1 odot\_\{w,0\} M\textasciicircum{}2}}
			
			\begin{prop}\label{prop3.7}
				Let $M^1,M^2,M^3\in \CC_1(V)$. Then there exists a linear isomorphism which is natural in $M^1,M^2$:
				\begin{align*}
					\Om_0: \Hom_{A(V)}\left(M^1\odot_{w,0}M^2,M^3(0)\right)&\simeq \Hom_{A(V)}\left(M^2\odot_{-w,0}M^1,M^3(0)\right),\quad f\mapsto \tilde{f},\\
					\braket*{(u^3)'}{\tilde{f}(v^2\odot_{-w}v^1)}&=\braket*{(u^3)'}{f(v^1\odot_{w} v^2)},
				\end{align*}
				for all $(u^3)'\in M^3(0)^\ast$, $v^1\in M^1$, and $v^2\in M^2$.
			\end{prop}
			\begin{proof}
				Using Lemma~\ref{lm:rescfbandcfb} and Corollary~\ref{coro:3.10}, we can extend the diagram \eqref{diagrampart1} to the following commutative diagram:
				\begin{equation}\label{diagrampart2}
					\begin{tikzcd}
						I\fusion{M^1}{M^2}{M^3}\arrow[r,"\Om_0","\simeq"'] \arrow[d,"\simeq"]& I\fusion{M^2}{M^1}{M^3}\arrow[d,"\simeq"]\\
						\mathcal{M}[P(w)]_{M^1M^2}^{M^3}\arrow[r,"\Om_0","\simeq"'] \arrow[d,"\simeq"] & \mathcal{M}[P(-w)]_{M^2M^1}^{M^3}\arrow[d,"\simeq"]\\
						\mathscr{C}((M^3)' \o M^1\o M^2)_{(\P^1,\infty,w,0)}\arrow[r,"\Om_0","\simeq"'] \arrow[d,"\simeq","\eqref{eq:cfb-restriction}"']& \mathscr{C}((M^3)' \o M^2\o M^1)_{(\P^1,\infty,-w,0)} \arrow[d,"\simeq","\eqref{eq:cfb-restriction}"']\\
						\mathscr{C}_{\mathrm{res}}(M^3(0)^\ast \o M^1\o M^2)_{(\P^1,\infty,w,0)}\arrow[r,dashed,"\Om_0","\simeq"'] \arrow[d,"\simeq","\eqref{eq:rescfbandhom}"']& \mathscr{C}_{\mathrm{res}}(M^3(0)^\ast \o M^2\o M^1)_{(\P^1,\infty,-w,0)}\arrow[d,"\simeq","\eqref{eq:rescfbandhom}"']\\
						\Hom_{A(V)}\left(M^1\odot_{w,0}M^2, M^3(0)\right)\arrow[r,dashed,"\Om_0","\simeq"']& \Hom_{A(V)}\left(M^2\odot_{-w,0}M^1, M^3(0)\right).
					\end{tikzcd}
				\end{equation}
				Given $\varphi\in \mathscr{C}((M^3)' \o M^1\o M^2)_{(\P^1,\infty,w,0)}$, we denote the correspondence of elements in the vector spaces of \eqref{diagrampart2} by
				\[
				\begin{tikzcd}
					\varphi\arrow[r,mapsto,"\Om_0"] \arrow[d,mapsto,"\eqref{eq:cfb-restriction}"']& \widetilde{\varphi}\arrow[d,mapsto,"\eqref{eq:cfb-restriction}"]\\
					\varphi_{\mathrm{res}}\arrow[r,mapsto,"\Om_0"]\arrow[d,mapsto,"\eqref{eq:rescfbandhom}"']& \widetilde{\varphi}_{\mathrm{res}}\arrow[d,mapsto,"\eqref{eq:rescfbandhom}"]\\
					f \arrow[r,dashed,mapsto,"\Om_0"]&\tilde{f}.
				\end{tikzcd}
				\]
				Note that for $(u^3)'\in M^3(0)^\ast$, we have $L(1).(u^3)'=0$. Then by \eqref{IsoOm0} and \eqref{f}, we have
				\begin{equation}\label{tildefand f}
					\begin{aligned}
						\braket*{(u^3)'}{\tilde{f}(v^2\odot_{-w} v^1)}&=\braket*{\widetilde{\varphi}_{\mathrm{res}}}{(u^3)'\o v^2\o v^1}=\braket*{\varphi_{\mathrm{res}}}{e^{-wL(1)}(u^3)'\o v^1\o v^2}
						\\&=\braket*{\varphi_{\mathrm{res}}}{(u^3)'\o v^1\o v^2}=\braket*{(u^3)'}{f(v^1\odot_{w} v^2)},
					\end{aligned}
				\end{equation}
				where $v^1\in M^1,v^2\in M^2$, and $(u^3)'\in M^3(0)^\ast$. Finally, given morphisms $\al:N^1\to M^1$ and $ \b:N^2\to M^2,$ we claim that
				\begin{equation}\label{eq:naturalOm0}
					\Om^{N^1,N^2}_0(f\circ (\al\odot_{w,0}\b))=\Om^{M^1,M^2}_0(f)\circ (\b\odot_{-w,0} \al),
				\end{equation}
				for any $f: M^1\odot_{w,0}M^2\to M^3(0)$ an $A(V)$-homomorphism. Indeed, by the definition of $\Om_0(f)$, we have for any $(u^3)'\in M^3(0)^\ast$,
				\begin{align*}
					&\braket*{(u^3)'}{ \left(\Om^{N^1,N^2}_0(f\circ (\al\odot_{w,0}\b))\right)(n^2\odot_{-w}n^1)}=\braket*{(u^3)'}{f(\al(n^1)\odot_w \b(n^2))}\\
					&=\braket*{(u^3)'}{\tilde{f}(\b(n^2)\odot_{-w}\al(n^1))}=\braket*{(u^3)'}{\left(\Om^{M^1,M^2}_0(f)\circ (\b\odot_{-w,0} \al)\right)(n^2\odot_{-w}n^1)}.
				\end{align*}
				This shows \eqref{eq:naturalOm0}. Hence $\Om_0$ is natural in $M^1$ and $M^2$.
			\end{proof}
			
			
			\begin{lm}\label{lm3.10}
				Let $M^1,M^2,M^3\in \CC_1(V)$, and
				\[
				\Om_0: \Hom_{A(V)}\left(M^1\odot_{w,0}M^2,M^3(0)\right)\simeq \Hom_{A(V)}\left(M^2\odot_{-w,0}M^1,M^3(0)\right),\quad f\mapsto \tilde{f}\]
				be the isomorphism given by \eqref{diagrampart2}. Let $a\o (z^n/(z-w)^m)\in \mathcal{L}_{\P^1\bs\{w,0 \}}(V)_{\leq 0}$, with $n-m\leq \wt a-1$, and $a\o (u^{-m}/(u+w)^{-n})\in \mathcal{L}_{\P^1\bs\{-w,0\}}(V)_{\leq 0}$. Then
				\begin{align}
					f\left(\rho_w\left(a\o\frac{z^n}{(z-w)^m}\right).v^1\odot_w v^2\right)&=\tilde{f}\left(v^2\odot_{-w}\rho_0\left(a\o\frac{u^{-m}}{(u+w)^{-n}}\right).v^1\right)\label{f1},\\
					f\left(v^1\odot_w \rho_0\left(a\o\frac{z^n}{(z-w)^m}\right).v^2\right)&=\tilde{f}\left(\rho_{-w}\left(a\o\frac{u^{-m}}{(u+w)^{-n}}\right).v^2\odot_{-w} v^1\right). \label{f2}
				\end{align}
			\end{lm}
			\begin{proof}
				By \eqref{relvarphiandtildevarphi} and \eqref{f}, and the fact that $e^{-wL(1)}(u^3)'=(u^3)'$ for all $(u^3)'\in M^3(0)^\ast$, we have
				\begin{align*}
					&\braket*{(u^3)'}{f\left(\rho_w\left(a\o\frac{z^n}{(z-w)^m}\right).v^1\odot_w v^2\right)}\\
					&=\braket*{\varphi_{\mathrm{res}}}{(u^3)'\o \rho_{w}\left(a\o\frac{z^n}{(z-w)^m}\right).v^1\o v^2}=\braket*{\widetilde{\varphi}_{\mathrm{res}}}{(u^3)'\o v^2\o \rho_{0}\left(a\o\frac{u^{-m}}{(u+w)^{-n}}\right).v^1}\\
					&=\braket*{(u^3)'}{\tilde{f}\left(v^2\odot_{-w}\rho_0\left(a\o\frac{u^{-m}}{(u+w)^{-n}}\right).v^1\right)}
				\end{align*}
				This proves \eqref{f1}. Similarly, \eqref{f2} follows from the third equality in \eqref{relvarphiandtildevarphi} and \eqref{f}, we omit the details.
			\end{proof}
			
			The following is our main theorem in this Section, which is essential for the definition of braiding isomorphism on $\Mod_{\mathrm{fd}}(A(V))$.
			
			\begin{thm}\label{thm:Rw}
				Let $M^1,M^2\in \CC_1(V)$, and $w\in \P^1\bs\{\infty, 0\}.$ Then
				\begin{equation}\label{Rw}
					R_w: M^2\odot_{-w,0} M^1\ra M^1\odot_{w,0} M^2,\quad R_w(v^2\odot_{-w}v^1)=v^1\odot_w v^2
				\end{equation}
				is a canonical isomorphism of left $A(V)$-modules which is natural in $M^1$ and $M^2$. It satisfies
				\begin{align}
					R_w\left(v^2\odot_{-w}\rho_0\left(a\o\frac{u^{-m}}{(u+w)^{-n}}\right).v^1\right)&=\rho_w\left(a\o\frac{z^n}{(z-w)^m}\right).v^1\odot_w v^2, \label{eq:Rw1}\\
					R_w\left(\rho_{-w}\left(a\o\frac{u^{-m}}{(u+w)^{-n}}\right).v^2\odot_{-w} v^1\right)&=v^1\odot_w \rho_0\left(a\o\frac{z^n}{(z-w)^m}\right).v^2,\label{eq:Rw2}
				\end{align}
				where $n-m\leq \wt a-1$.
			\end{thm}
			
			\begin{proof}
				By Proposition~\ref{prop:condforC1V}, $M^2\odot_{-w,0} M^1$ and $M^1\odot_{w,0} M^2$ are finite-dimensional left $A(V)$-module. Moreover, $(\Phi^\L\dashv \Om): \Mod_{\mathrm{fd}}(A(V))\rightleftarrows \CC_1(V)$ is an equivalence of categories. Let $M^3=\Phi^\L(M^1\odot_{w,0} M^2)$. Then $M^3\in \CC_1(V)$ is a finite direct sum of irreducible ordinary $V$-modules. Hence \[
				M^3(0)=\Om (\Phi^\L (M^1\odot_{w,0} M^2))=M^1\odot_{w,0} M^2.
				\]
				By Proposition~\ref{prop:condforC1V}, $(M^3)'\in \CC_1(V)$, and so $(M^3)'\cong\Phi^\L(\Om((M^3)'))=\Phi^\L(M^3(0)^\ast)$. Thus the hypotheses of Lemma~\ref{lm:rescfbandcfb} hold. By Proposition~\ref{prop3.7} there exists a canonical isomorphism
				\begin{align*}
					\Om_0: \Hom_{A(V)}\left(M^1\odot_{w,0}M^2,M^1\odot_{w,0}M^2\right)&\simeq \Hom_{A(V)}\left(M^2\odot_{-w,0}M^1,M^1\odot_{w,0}M^2\right),\\
					f&\mapsto \tilde{f}.
				\end{align*}
				Let $f=\Id$. Then $\widetilde{\Id}: M^2\odot_{-w,0}M^1\ra M^1\odot_{w,0}M^2$ is an $A(V)$-homomorphism. Moreover, for any $(u^3)'\in (M^3(0))^\ast= (M^1\odot_{w,0} M^2)^\ast$, it follows from \eqref{tildefand f} that
				\[
				\braket*{(u^3)'}{\widetilde{\Id}(v^2\odot_{-w} v^1)}=\braket*{(u^3)'}{\Id(v^1\odot_{w} v^2)}=\braket*{(u^3)'}{v^1\odot_{w} v^2}.
				\]
				Thus, $\widetilde{\Id}(v^2\odot_{-w} v^1)=v^1\odot_{w} v^2$, and $R_w=\widetilde{\Id}$ \eqref{Rw} is a homomorphism of left $A(V)$-modules. Finally, we may take $M^3=\Phi^\L(M^2\odot_{-w,0}M^1)$ in diagram \eqref{diagrampart2} and construct the inverse of $R_w$ using the inverse of $\Om_0$. Hence $R_w$ is an isomorphism of $A(V)$-modules.
				
				Finally, to show \eqref{eq:Rw1} and \eqref{eq:Rw2}, we let $M^3=\Phi^\L(M^1\odot_{w,0}M^2)$ and $f=\Id$ in Lemma~\ref{lm3.10}. Then \eqref{eq:Rw1} follows from \eqref{f1}, \eqref{eq:Rw2} follows from \eqref{f2}. Naturality follows directly from the defining formula \eqref{Rw}.
				Indeed, for $V$-module homomorphisms
				$\alpha:M^1\to N^1$ and $\beta:M^2\to N^2$, we have
				\[ (\alpha\odot_{w,0}\beta)\circ R_w^{M^1,M^2} =
				R_w^{N^1,N^2}\circ(\beta\odot_{-w,0}\alpha),
				\]
				since both sides send $v^2\odot_{-w}v^1$ to
				$\alpha(v^1)\odot_w\beta(v^2)$.
			\end{proof}
			
			
			
			\begin{remark}
				In particular, by \eqref{odotaction}, we have
				\begin{equation}\label{Rwcomm}
					R_w\left(\rho_{-w,0}\left(a\o\frac{u^{-m}}{(u+w)^{-n}}\right).(v^2\odot_{-w}v^1)\right)=\rho_{w,0} \left(a\o\frac{z^n}{(z-w)^m}\right).(v^1\odot_{w} v^2).
				\end{equation}
				If we take $n=\wt a-1$ and $m=0$ in \eqref{Rwcomm}, then by \eqref{AVactionq} we have
				\begin{align*}
					R_w([a].(v^2\odot_{-w}v^1))&=R_w\left(\rho_{-w,0}\left(a\o \frac{1}{(u+w)^{-\wt a+1}}\right).(v^2\odot_{-w}v^1)\right)\\
					&=\rho_{w,0}\left(a\o z^{\wt a-1}\right).(v^1\odot_w v^2)\\
					&=[a].R_w(v^2\odot_{-w}v^1),
				\end{align*}
				for any $[a]\in A(V)$.
				This gives an alternative way of showing the $A(V)$-linearity of $R_w$ \eqref{Rw}.
			\end{remark}
			
			\subsection{Parallel transport of three-point conformal blocks}\label{subsec:parallel}
			Given two distinct points $x,y\in \P^{1}\bs\{ \infty,0\}$, there are associated spaces of three-point conformal blocks
			\[
			\mathscr{C}((M^3)' \o M^{2}\o M^{1})_{(\P^{1},\infty,x,0)}\quad \mathrm{and} \quad \mathscr{C}((M^3)' \o M^{2}\o M^{1})_{(\P^{1},\infty,y,0)}.
			\]
			We first prove that they are isomorphic as vector spaces using the idea of parallel transport isomorphism by Huang-Lepowsky-Zhang \cite{HLZ}.
			
			Without loss of generality, we assume $y=1$. Note that fixing a determination of $\log(x)$ with the normalization $\log(1)=0$
			amounts to choosing a homotopy class of paths with fixed endpoints:
			\[
			\gamma:[0,1]\to \C^\times,\qquad \gamma(0)=1,\ \gamma(1)=x.
			\]
			Denote by $l_\gamma(x)$ the logarithm of $x$ determined by $\gamma$, with $l_\gamma(1)=0$.
			
			\begin{example}\label{ex4.1}
				Let $x=-1$. There are two standard paths from $1$ to $-1$. One along the upper-half unit circle, the other along the lower-half unit circle, see Figure~\ref{fig2}.
				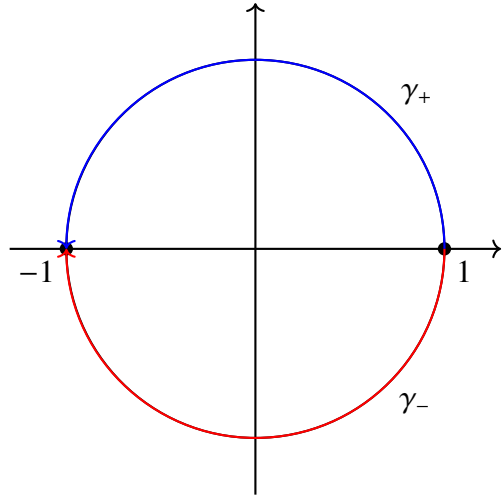
\begin{figure}[ht]
					\centering
					\begin{tikzpicture}[scale=2.5, thick]
						
						\draw[->] (-1.3,0) -- (1.3,0);
						\draw[->] (0,-1.3) -- (0,1.3);
						
						\draw (0,0) circle (1);
						
						\fill (1,0) circle (1pt);
						\node[below right] at (1,0) {$1$};
						
						\fill (-1,0) circle (1pt);
						\node[below left] at (-1,0) {$-1$};
						
						\draw[->, blue, thick]
						(1,0) arc (0:180:1);
						\node[above left] at (1,0.7) {$\gamma_+$};
						
						\draw[->, red, thick]
						(1,0) arc (0:-180:1);
						\node[below right] at (0.7,-0.7) {$\gamma_-$};
						
					\end{tikzpicture}
					\caption{Paths $\gamma_+$ and $\gamma_-$ from $1$ to $-1$}\label{fig2}
				\end{figure}
				The branch cut for the logarithm defined by $\gamma_+$ could be chosen to the negative $y$-axis; while the branch cut for the logarithm defined by $\gamma_-$ could be chosen to the positive $y$-axis.
				Then $l_{\gamma_+}(-1)=\pi i$ and $l_{\gamma_-}(-1)=-\pi i$.
				
				Similar to Remark~\ref{remark:Om0}, by fixing the marked points
				at $\infty$ and $0$, we obtain a sheaf of conformal blocks
				$\mathscr C_0$ over $\C^\times$, whose fiber at $z$ is
				\[
				\mathscr C_0|_z	=
				\mathscr C\bigl((M^3)'\o M^2\o M^1\bigr)
				_{(\P^1,\infty,z,0)}.	\]
				Here $\C^\times$ is identified with the slice
				$\{(z,0):z\neq0\}\subset\mathrm{Conf}_2(\A^1)$.
				The isomorphisms $T_{\ga_\pm}$ in \eqref{Tgamma}
				are parallel transport for the standard flat connection
				along the two reversed semicircular paths
				$\ga_\pm^{-1}$ from $-1$ to $1$ in Figure~\ref{fig2}.
			\end{example}
			
			
			
			\begin{lm}\label{lm3.1}
				Fix a path $\gamma:1\sim x$ in $\C^\times$. Let $l_\gamma(x)$ be the logarithm of $x$ uniquely determined by $\gamma$ and $l_\gamma(1)=0$.
				Let $M^{1},M^{2},M^3$ be ordinary $V$-modules. With respect to the action \eqref{eq:actionofchiralliealgebraelements} of the chiral Lie algebras on $(M^3)'\o M^{2}\o M^{1}$, we have
				\begin{equation}\label{eq3.1}
					\begin{aligned}
						&x^{\wt a-n+m-1}\cdot \left((x^{L(0)}\o x^{-L(0)}\o x^{-L(0)})\circ \rho_{\infty,x,0}\left( a\o\frac{z^n}{(z-x)^m}\right)\right).((v^3)'\o v^{2}\o v^{1})\\
						&= \left(\rho_{\infty,1,0}\left( a\o\frac{z^n}{(z-1)^m}\right)\circ (x^{L(0)}\o x^{-L(0)}\o x^{-L(0)})\right).((v^3)'\o v^{2}\o v^{1}),
					\end{aligned}
				\end{equation}
				where $x^{L(0)}=e^{l_\gamma(x)L(0)}=\sum_{j\geq 0}\frac{1}{j!} (l_\gamma(x))^jL(0)^j$ as an operator on ordinary modules, and $(v^3)'\o v^{2}\o v^{1}\in (M^3)' \o M^{2}\o M^{1}$.
			\end{lm}
			\begin{proof}
				First, we note that $x^{L(0)}$ is a well-defined operator on an ordinary module $M$. Indeed, given any $v\in M_\la$, with $\la\in \C$, we have $x^{L(0)}.v=\sum_{j\geq 0}\frac{1}{j!} (l_\gamma(x)\lambda)^j\cdot v$, wherein the coefficient is a convergent series. Since $M$ is a direct sum of $M_\la$'s. Hence $x^{L(0)}$ is well-defined. Furthermore, the operator $x^{L(0)}$ is invertible, with inverse $x^{-L(0)}$.
				
				With the notation \eqref{eq:inftyaction}, we have $[L(0),a'_{(n-m-j)}]=(-\wt a+n-m-j+1) a'_{(n-m-j)}$. Since $M^i$ are ordinary $V$-modules, we may choose $(v^3)'\in (M^3)'$, $v^1\in M^1$, and $v^2\in M^2$ to be $L(0)$-eigenvectors. It follows from \eqref{eq:actionofchiralliealgebraelements} that
				\begin{align*}
					&x^{\wt a-n+m-1}\cdot \left((x^{L(0)}\o x^{-L(0)}\o x^{-L(0)})\circ \rho_{\infty,x,0}\left( a\o\frac{z^n}{(z-x)^m}\right)\right).((v^3)'\o v^{2}\o v^{1})\\
					&=-\sum_{j\geq 0}\binom{-m}{j} (-1)^jx^{\wt a-n+m-1+j} x^{L(0)}.(a'_{(n-m-j)}(v^3)')\o x^{-L(0)}.v^{2}\o x^{-L(0)}.v^{1}\\
					&\ \ \ + \sum_{j\geq 0} \binom{n}{j} x^{\wt a-n+m-1+n-j}x^{L(0)}.(v^3)'\o x^{-L(0)}.(a_{(j-m)}v^{2})\o x^{-L(0)}.v^{1}\\
					&\ \ \ +\sum_{j\geq 0} \binom{-m}{j}(-1)^{-m-j}x^{\wt a-n+m-1-m-j}x^{L(0)}.(v^3)'\o x^{-L(0)}.v^{2}\o x^{-L(0)}.(a_{(n+j)}v^{1})\\
					&= -\sum_{j\geq 0}\binom{-m}{j} (-1)^jx^{\wt(v^3)'-\wt v^{2}-\wt v^{1}}\cdot a'_{(n-m-j)}(v^3)'\o v^{2}\o v^{1}\\
					&\ \ \ +\sum_{j\geq 0} \binom{n}{j} x^{\wt(v^3)'-\wt v^{2}-\wt v^{1}}(v^3)'\o a_{(j-m)}v^{2}\o v^{1}\\
					&\ \ \ +\sum_{j\geq 0} \binom{-m}{j}(-1)^{-m-j}x^{\wt(v^3)'-\wt v^{2}-\wt v^{1}}(v^3)'\o v^{2}\o a_{(n+j)}v^{1}\\
					&=x^{\wt(v^3)'-\wt v^{2}-\wt v^{1}}\cdot \left( \rho_{\infty,1,0}\left( a\o\frac{z^n}{(z-1)^m}\right).((v^3)'\o v^{2}\o v^{1})\right)\\
					&= \left( \rho_{\infty,1,0}\left( a\o\frac{z^n}{(z-1)^m}\right)\circ (x^{L(0)}\o x^{-L(0)}\o x^{-L(0)})\right).((v^3)'\o v^{2}\o v^{1}),
				\end{align*}
				where $x^\al$ is defined as $e^{\al l_\gamma(x)}$ for any complex number $\al\in \C$.
			\end{proof}
			
			\begin{lm}\label{lm:Tgamma}
				Fix a path $\gamma:1\sim x$ in $\C^\times$. Let $l_\gamma(x)$ be the logarithm of $x$ uniquely determined by $\gamma$ and $l_\gamma(1)=0$. Then there exists an isomorphism of conformal blocks
				\begin{equation}\label{Tgamma}
					\begin{aligned}
						T_\gamma: \mathscr{C}((M^3)'\o M^{2}\o M^{1})&_{(\P^{1},\infty, x,0)}\ra \mathscr{C}((M^3)'\o M^{2}\o M^{1})_{(\P^{1},\infty, 1,0)},\quad \varphi\mapsto \varphi_{\gamma},\\
						\braket*{\varphi_\gamma}{(v^3)'\o v^{2}\o v^{1}}&=\braket*{\varphi}{x^{-L(0)}.(v^3)'\o x^{L(0)}.v^{2}\o x^{L(0)}.v^{1}}
					\end{aligned}
				\end{equation}
				is an isomorphism of three-point conformal blocks on $\P^{1}$, where $x^{L(0)}=e^{l_\ga (x) L(0)}$.
			\end{lm}
			\begin{proof}
				First, we show that $\varphi_\gamma$ is invariant under the action of $\mathcal{L}_{\P^{1}\bs\{\infty,1,0\}}(V)$. i.e., $T_\gamma$ is well-defined. It follows from \eqref{eq3.1} and \eqref{Tgamma} that
				\begin{align*}
					&\braket*{\varphi_\gamma}{\rho_{\infty,1,0}\left(a\o \frac{z^n}{(z-1)^m}\right).(x^{L(0)}.(v^3)'\o x^{-L(0)}.v^{2}\o x^{-L(0)}.v^{1})}\\
					&=\braket*{\varphi}{\left((x^{-L(0)}\o x^{L(0)}\o x^{L(0)})\circ \rho_{\infty,1,0}\left(a\o \frac{z^n}{(z-1)^m}\right)\right).(x^{L(0)}.(v^3)'\o x^{-L(0)}.v^{2}\o x^{-L(0)}.v^{1})}\\
					&=x^{\wt a-n+m-1}\braket*{\varphi}{\rho_{\infty,x,0}\left(a\o \frac{z^n}{(z-x)^m}\right).((v^3)'\o v^{2}\o v^{1})}\\
					&=0,
				\end{align*}
				as $\varphi\in \mathscr{C}((M^3)'\o M^{2}\o M^{1})_{(\P^{1},\infty, x,0)}$ is invariant under $\mathcal{L}_{\P^{1}\bs\{\infty,x,0\}}(V)$. Moreover, since $x^{L(0)}$ is an invertible operator on $M^i$, this shows $\varphi_\gamma$ vanishes on $\mathcal{L}_{\P^{1}\bs\{\infty,1,0\}}(V). ((M^3)'\o M^{2}\o M^{1})$.
				
				It is clear that the inverse of $T_\gamma$ is given by
				\begin{equation}\label{Tgammainverse}
					\begin{aligned}
						T_{\gamma^{-1}}: \mathscr{C}((M^3)'\o M^{2}\o M^{1})&_{(\P^{1},\infty, 1,0)}\ra \mathscr{C}((M^3)'\o M^{2}\o M^{1})_{(\P^{1},\infty, x,0)},\quad \varphi\mapsto \varphi_{\gamma^{-1}},\\
						\braket*{\varphi_{\gamma^{-1}}}{(v^3)'\o v^{2}\o v^{1}}&=\braket*{\varphi}{x^{L(0)}.(v^3)'\o x^{-L(0)}.v^{2}\o x^{-L(0)}.v^{1}}
					\end{aligned}
				\end{equation}
				The well-definedness of $T_{\gamma^{-1}}$ can be proved similarly using \eqref{eq3.1}.
			\end{proof}
			
			\begin{prop}\label{prop:Tgamma}
				Let $M^1,M^2,M^3\in \CC_1(V)$. Then there exists a linear isomorphism which is natural in $M^1,M^2$:
				\begin{equation}\label{fga}
					\begin{aligned}
						T_\gamma: \Hom_{A(V)}\left(M^2\odot_{x,0}M^1, M^3(0)\right) &\simeq \Hom_{A(V)}(M^2\odot_{1,0}M^1,M^3(0)), \quad f\mapsto f_{\ga},\\
						f_\ga(v^2\odot_1 v^1)
						&=x^{-L(0)}.f(x^{L(0)}.v^{2}\odot_x x^{L(0)}.v^{1}),
					\end{aligned}
				\end{equation}
				for all $(u^3)'\in M^3(0)^\ast$, $v^1\in M^1$, and $v^2\in M^2$.
			\end{prop}
			\begin{proof}
				Similar to Proposition~\ref{prop3.7}, using Lemma~\ref{lm:rescfbandcfb} and Corollary~\ref{coro:3.10} again,
				we can extend the isomorphism $T_\gamma$ to the following commutative diagram:
				
				\begin{equation}\label{diagramga1}
					\begin{tikzcd}
						\mathscr{C}((M^3)' \o M^{2}\o M^{1})_{(\P^{1},\infty,x,0)}\arrow[r,"T_\ga","\simeq"'] \arrow[d,"\simeq","\eqref{eq:cfb-restriction}"']& \mathscr{C}((M^3)' \o M^{2}\o M^{1})_{(\P^{1},\infty,1,0)} \arrow[d,"\simeq","\eqref{eq:cfb-restriction}"']\\
						\mathscr{C}_{\mathrm{res}}(M^3(0)^\ast \o M^{2}\o M^{1})_{(\P^{1},\infty,x,0)}\arrow[r,dashed,"T_\ga","\simeq"'] \arrow[d,"\simeq","\eqref{eq:rescfbandhom}"']& \mathscr{C}_{\mathrm{res}}(M^3(0)^\ast \o M^{2}\o M^{1})_{(\P^{1},\infty,1,0)}\arrow[d,"\simeq","\eqref{eq:rescfbandhom}"']\\
						\Hom_{A(V)}\left(M^{2}\odot_{x,0}M^{1}, M^3(0)\right)\arrow[r,dashed,"T_\ga","\simeq"']& \Hom_{A(V)}\left(M^{2}\odot_{1,0}M^{1}, M^3(0)\right),
					\end{tikzcd}
				\end{equation}
				where the vertical maps are given by \eqref{eq:cfb-restriction} and \eqref{eq:rescfbandhom}.
				Given $\varphi\in \mathscr{C}((M^3)'\o M^2\o M^1)_{(\P^1,\infty, x,0)}$, the corresponding elements under the isomorphisms in \eqref{diagrampart1} are denoted by
				\[
				\begin{tikzcd}
					\varphi\arrow[r,mapsto,"T_\ga"] \arrow[d,mapsto,"\eqref{eq:cfb-restriction}"']& \varphi_\ga\arrow[d,mapsto,"\eqref{eq:cfb-restriction}"]\\
					\varphi_{\mathrm{res}}\arrow[r,mapsto,"T_\ga"]\arrow[d,mapsto,"\eqref{eq:rescfbandhom}"']& \varphi_{\mathrm{res},\ga}\arrow[d,mapsto,"\eqref{eq:rescfbandhom}"]\\
					f \arrow[r,dashed,mapsto,"T_\ga"]&f_\ga,
				\end{tikzcd}
				\]
				For $(u^3)'\in M^3(0)^\ast$ and $u^3\in M^3(0)$, we have $\braket*{L(0)(u^3)'}{u^3}=\braket*{(u^3)'}{L(0)u^3}$. It follows from \eqref{Tgamma} that given any $A(V)$-homomorphism $f: M^{2}\odot_{x,0}M^{1}\ra M^3(0)$, the corresponding $A(V)$-homomorphism $f_\ga=T_\ga(f)$ has the following expression:
				\begin{align*}
					\braket*{(u^3)'}{f_\gamma(v^{2}\odot_1 v^{1})}
					&=\braket*{x^{-L(0)}.(u^3)'}{f(x^{L(0)}.v^{2}\odot_x x^{L(0)}.v^{1})}\\
					&=\braket*{(u^3)'}{x^{-L(0)}.f(x^{L(0)}.v^{2}\odot_x x^{L(0)}.v^{1})},
				\end{align*}
				where $x^{L(0)}=e^{l_\ga(x)L(0)}$. This proves \eqref{fga}. Finally, for $\alpha:N^1\to M^1$ and $\b:N^2\to M^2$, it follows immediately from \eqref{fga} that
				\[
				T_\gamma^{N^1,N^2}
				\!\left(f\circ(\beta\odot_{x,0}\alpha)\right)=
				T_\gamma^{M^1,M^2}(f)
				\circ(\beta\odot_{1,0}\alpha).
				\]
				The naturality of $T_\ga$ is clear.
			\end{proof}
			
			
			\subsection{\texorpdfstring{Construction of $R^{M^2,M^1}_\pm$ for the $\odot$-tensor product}{Construction of R\textasciicircum{}\{M\textasciicircum{}2,M\textasciicircum{}1\}\_+ and R\textasciicircum{}\{M\textasciicircum{}2,M\textasciicircum{}1\}\_- for the odot-tensor product}}
			
			We adopted the following notation in Definition~\ref{df:odotw}:
			\[
			M\odot N=M\odot_{1,0} N,\quad \mathrm{and}\quad u\odot v=u \odot_1 v\in M\odot N.
			\]
			
			By Proposition~\ref{prop:condforC1V}, every object $S$ of $\Mod_{\mathrm{fd}}(A(V))$ is a finite direct sum of irreducible modules $S=\bigoplus_i S^i$, with $S^i\cong W^i(0)$ for some irreducible $W^i\in\CC_1(V)$. Note that $[\om]$
			is a central element of $A(V)$ \cite{Z96} and acts on
			$W^i(0)$ by $h^i\cdot\Id_{W^i(0)}$, where $h^i\in\C$ is the conformal weight of $W^i$.
			Hence $[\om]$ acts semisimply on $S$, and we can write
			\begin{equation}\label{eq:epiiomaction}
				e^{\pm\pi i[\om]}\big|_S
				=
				\sum_i e^{\pm\pi ih^i}\,\Id_{S^i}.
			\end{equation}
			Since $[\om]$ is central, these are mutually inverse
			$A(V)$-linear automorphisms of $S$, which are natural in $S$.
			
			\begin{thm}\label{thm:Rmatrix}
				Let $M^1,M^2\in \CC_1(V)$. Then there exists isomorphisms of left $A(V)$-modules which are natural in $M^1,M^2$:
				\begin{equation}\label{R+-}
					\begin{aligned}
						R^{M^2,M^1}_\pm:& M^2\odot M^1\ra M^1\odot M^2,\\
						R^{M^2,M^1}_\pm (v^2\odot v^1)&=e^{\mp \pi i L(0)}. \left(e^{\pm \pi i L(0)}v^1\odot e^{\pm \pi i L(0)}v^2\right),
					\end{aligned}
				\end{equation}
				where the outer factor $e^{\mp\pi iL(0)}$ acts on $M^1\odot M^2$ through $e^{\mp\pi i[\om]}$ \eqref{eq:epiiomaction}.
				Moreover,
				\begin{equation}\label{eq:opposite-braiding-inverse}
					\left(R_-^{M^1,M^2}\right)^{-1}=R_+^{M^2,M^1},\qquad \left(R_+^{M^1,M^2}\right)^{-1}=R_-^{M^2,M^1}.
				\end{equation}
			\end{thm}
			\begin{proof}[First Proof of Theorem~\ref{thm:Rmatrix}: using the natural identifications of Hom-spaces] \leavevmode\par
				
				Let $\ga_\pm:1\sim -1$ be the paths in
				Example~\ref{ex4.1}, so that
				$l_{\ga_\pm}(-1)=\pm\pi i$.
				By Proposition~\ref{prop3.7} and
				Proposition~\ref{prop:Tgamma}, for any
				$M^3\in \CC_1(V)$, we have a composition of isomorphisms
				\begin{align*}
					&\Hom_{A(V)}
					\left(M^1\odot_{1,0}M^2,M^3(0)\right)
					\xrightarrow{\Om_0}
					\Hom_{A(V)}
					\left(M^2\odot_{-1,0}M^1,M^3(0)\right)\\
					&\hspace{4em}\xrightarrow{T_{\ga_\pm}}
					\Hom_{A(V)}
					\left(M^2\odot_{1,0}M^1,M^3(0)\right), \qquad f\mapsto\tilde f\mapsto\tilde f_{\ga_\pm}.
				\end{align*}
				It follows from \eqref{tildefand f} and \eqref{fga}
				that
				\begin{equation}\label{imageoff}
					\tilde f_{\ga_\pm}(v^2\odot v^1)
					= e^{\mp\pi iL(0)}.
					f\left(
					e^{\pm\pi iL(0)}v^1\odot
					e^{\pm\pi iL(0)}v^2
					\right).
				\end{equation}
				Similar to the proof of Theorem~\ref{thm:Rw}, since
				$M^1\odot_{1,0}M^2$ is a finite-dimensional object in the semisimple category $\Mod_{\mathrm{fd}}(A(V))$, by letting
				$$
				M^3=\Phi^\L(M^1\odot_{1,0}M^2),\quad \mathrm{with} \quad M^3(0)=M^1\odot_{1,0}M^2,
				$$
				the element $f=\Id\in \Hom_{A(V)}
				(M^1\odot_{1,0}M^2,M^1\odot_{1,0}M^2)$ gives rise to an isomorphism of $A(V)$-modules under $\Om_0$:
				\[
				R_1=\widetilde{\Id}:M^2\odot_{-1,0}M^1\ra M^1\odot_{1,0}M^2,
				\qquad
				R_1(v^2\odot_{-1}v^1)=v^1\odot v^2.
				\]
				Then by Proposition~\ref{prop:Tgamma}, together with \eqref{imageoff} for $f=\Id$, we have homomorphisms of $A(V)$-modules by applying $T_{\ga_\pm}$
				\begin{align*}
					& R_\pm^{M^2,M^1} :=
					T_{\ga_\pm}(R_1)=
					\widetilde{\Id}_{\ga_\pm}: M^2\odot_{1,0}M^1\to M^1\odot_{1,0}M^2,\\
					&R_\pm^{M^2,M^1}(v^2\odot v^1)=e^{\mp\pi iL(0)}.\left(e^{\pm\pi iL(0)}v^1\odot e^{\pm\pi iL(0)}v^2\right),
				\end{align*}
				which is the formula \eqref{R+-}.
				
				Let $\alpha:M^1\ra N^1$ and $\beta:M^2\ra N^2$
				be $V$-module homomorphisms. Since $\alpha$ and
				$\beta$ commute with $L(0)$, and
				$\alpha\odot\beta$ is an $A(V)$-homomorphism,
				\eqref{R+-} gives
				\[ (\alpha\odot\beta)\circ R_\pm^{M^2,M^1} =
				R_\pm^{N^2,N^1}\circ(\beta\odot\alpha). \]
				Hence $R_\pm^{M^2,M^1}$ are natural in $M^1$ and $M^2$.
				
				Finally, we may interchange $M^2,M^1$ in \eqref{R+-} and obtain $A(V)$-homomorphisms $R_\pm^{M^1,M^2}: M^1\odot M^2\to M^2\odot M^1 $. Since $ R_-^{M^1,M^2}$ commutes with the action of $e^{-\pi i[\om]}$, we have
				\begin{align*}
					R_-^{M^1,M^2}\left(R_+^{M^2,M^1}(v^2\odot v^1)\right)&=R_-^{M^1,M^2}\left(e^{-\pi i [\om]}.\left(e^{\pi i L(0)}v^1\odot e^{\pi i L(0)}v^2\right)\right)
					\\ &= e^{-\pi i[\om]}. R_-^{M^1,M^2}
					\left(e^{\pi iL(0)}v^1\odot e^{\pi iL(0)}v^2
					\right)\\
					&= e^{-\pi i[\om]}e^{\pi i[\om]}. \left(
					e^{-\pi iL(0)}e^{\pi iL(0)}v^2 \odot
					e^{-\pi iL(0)}e^{\pi iL(0)}v^1 \right)\\
					&=v^2\odot v^1.
				\end{align*}
				Similarly,
				$
				R_+^{M^2,M^1}R_-^{M^1,M^2}(v^1\odot v^2) =
				v^1\odot v^2.
				$ Therefore, $R_+^{M^2,M^1}$ and $R_-^{M^1,M^2}$
				are mutually inverse isomorphisms of $A(V)$-modules.
				Interchanging $M^1$ and $M^2$ shows that
				$R_-^{M^2,M^1}$ and $R_+^{M^1,M^2}$ are also
				mutually inverse. This proves
				\eqref{eq:opposite-braiding-inverse} and completes
				the proof.
			\end{proof}
			\begin{proof}[Second Proof of Theorem~\ref{thm:Rmatrix}: using relations in restricted conformal blocks] \leavevmode\par
				First, we show well-definedness. Using the isomorphism
				$R_1$ in Theorem~\ref{thm:Rw}, define
				\[ \widetilde R_\pm:
				M^2\o_\C M^1\ra M^1\odot M^2, \quad
				\widetilde R_\pm(v^2\o v^1) =
				e^{\mp\pi i[\om]}.
				R_1\left( e^{\pm\pi iL(0)}v^2
				\odot_{-1} e^{\pm\pi iL(0)}v^1
				\right). \]
				By \eqref{Rw}, this is precisely
				\[
				\widetilde R_\pm(v^2\o v^1) = e^{\mp\pi i[\om]}.
				\left( e^{\pm\pi iL(0)}v^1
				\odot e^{\pm\pi iL(0)}v^2
				\right). \]
				Let $a\in V$ be homogeneous and $m,n\in\Z$ satisfy
				$n-m<\wt a-1$.
				Using the identity
				\[ e^{\pm\pi iL(0)}a_{(j)}e^{\mp\pi iL(0)} =
				(-1)^{\wt a-j-1}a_{(j)} \]
				and the local-action formulas \eqref{eq:1action}, \eqref{eq:0action}, we can show that
				\[ \begin{aligned}
					e^{\pm\pi iL(0)}
					\rho_1\left(a\o\frac{t^{-m}}{(t-1)^{-n}}\right)
					&=(-1)^{\wt a+n-m-1}
					\rho_{-1}\left(a\o\frac{u^{-m}}{(u+1)^{-n}}\right)
					e^{\pm\pi iL(0)},\\
					e^{\pm\pi iL(0)}
					\rho_0\left(a\o\frac{t^{-m}}{(t-1)^{-n}}\right)
					&=(-1)^{\wt a+n-m-1}
					\rho_0\left(a\o\frac{u^{-m}}{(u+1)^{-n}}\right)
					e^{\pm\pi iL(0)}.
				\end{aligned}
				\]
				Apply \eqref{eq:Rw2} with $w=1$, using the formulas above, we have
				\begin{align*}
					&\widetilde R_\pm\left(
					\rho_1\left(a\o\frac{t^{-m}}{(t-1)^{-n}}\right).v^2 \o v^1\right)= e^{\mp \pi i [\om]}. R_1\left( e^{\pm \pi i L(0)} \rho_1\left(a\o\frac{t^{-m}}{(t-1)^{-n}}\right).v^2 \odot_{-1} e^{\pm \pi i L(0)}v^1\right)
					\\
					&= (-1)^{\wt a+n-m-1}
					e^{\mp\pi i[\om]}.
					R_1\left( \rho_{-1}\left(a\o\frac{u^{-m}}{(u+1)^{-n}}\right).
					e^{\pm\pi iL(0)}v^2 \odot_{-1}e^{\pm\pi iL(0)}v^1 \right)\\
					&= (-1)^{\wt a+n-m-1}
					e^{\mp\pi i[\om]}. \left( e^{\pm\pi iL(0)}v^1 \odot \rho_0\left(a\o\frac{z^n}{(z-1)^m}\right).
					e^{\pm\pi iL(0)}v^2 \right).
				\end{align*}
				Similarly, apply \eqref{eq:Rw1} with $w=1$, we have
				\begin{align*}
					&\widetilde R_\pm\left( v^2\o
					\rho_0\left(a\o\frac{t^{-m}}{(t-1)^{-n}}\right).v^1\right)=\\
					&=(-1)^{\wt a+n-m-1} e^{\mp\pi i[\om]}. R_1\left( e^{\pm\pi iL(0)}v^2 \odot_{-1}
					\rho_0\left(a\o\frac{u^{-m}}{(u+1)^{-n}}\right).e^{\pm\pi iL(0)}v^1
					\right)\\
					&= (-1)^{\wt a+n-m-1}
					e^{\mp\pi i[\om]}. \left( \rho_1\left(a\o\frac{z^n}{(z-1)^m}\right).
					e^{\pm\pi iL(0)}v^1 \odot e^{\pm\pi iL(0)}v^2 \right).
				\end{align*}
				Adding these two identities and applying
				\eqref{eq:actionfogeneralspann}, we obtain
				\[ \widetilde R_\pm\left(
				\rho_1\left(a\o\frac{t^{-m}}{(t-1)^{-n}}\right).v^2 \o v^1 + v^2\o
				\rho_0\left(a\o\frac{t^{-m}}{(t-1)^{-n}}\right).v^1
				\right)=0 \]
				in $M^1\odot M^2$, since $n-m<\wt a-1$.
				
				Note that the elements
				$a\o t^{-m}/(t-1)^{-n}$ span restricted chiral ideal
				$\mathcal L_{\P^1\bs\{1,0\}}(V)_{<0}$ as $a$ varies in $V$ and $m,n$ varies in $ \Z$, with $n-m<\wt a-1$, see \eqref{eq:defofinftyrestrictchiral}.
				Hence $\widetilde R_\pm$ annihilates the subspace
				$ \mathcal L_{\P^1\bs\{1,0\}}(V)_{<0}.
				(M^2\o_\C M^1)$
				and induces a well-defined linear map
				\[ R_\pm^{M^2,M^1}:
				M^2\odot M^1\ra M^1\odot M^2 \]
				given by \eqref{R+-}. Using \eqref{Rwcomm} with $w=1$, together with a similar argument as above, we can show that $R^{M^2,M^1}_\pm$ are $A(V)$-homomorphisms.
				The naturality and \eqref{eq:opposite-braiding-inverse} follow from a similar arguments as the first proof, we omit the details.
			\end{proof}
			
			\begin{df}
				In what follows, we write
				\begin{equation}\label{eq:braiding-formula}
					\begin{aligned}
						&\R_{M^1,M^2}:=R_-^{M^1,M^2}: M^1\odot M^2\to M^2\odot M^1,\\
						&\R_{M^1,M^2}(v^1\odot v^2)=
						e^{\pi i[\om]}.
						\left(
						e^{-\pi iL(0)}v^2\odot
						e^{-\pi iL(0)}v^1
						\right).
					\end{aligned}
				\end{equation}
				Here $[\om]$ acts on $M^2\odot M^1$, and the
				outer factor $e^{\pi i[\om]}$ may also be written
				as $e^{\pi iL(0)}$.
			\end{df}
			
			In the next Subsection, we show that this
			family of isomorphisms defines a braiding on
			$\Mod_{\mathrm{fd}}(A(V))$ with respect to the tensor structure
			transported from $\CC_1(V)$ through the equivalence
			$(\Phi^\L,\Om)$, using the canonical identification
			$\eta_1$ in \eqref{eq:eta-explicit}.

			\subsection{\texorpdfstring{Comparison with the braiding isomorphism of Huang--Lepowsky--Zhang}{Comparison with the braiding isomorphism of Huang-Lepowsky-Zhang}}\label{sec:braiding}
			
			By \cite[Theorem~1.2]{H25} and \cite[Section~12]{HLZ},
			$\CC_1(V)$ has a natural braided tensor category structure with a
			twist. As in Section~\ref{sec:eta}, we take
			$\bt=\bt_{P(1)}$ and fix the branch $l(1)=0$.
			We take the braiding associated to the path $\ga_+^{-1}$ from $-1$ to $1$
			in the upper half-plane, see Example~\ref{ex4.1}. Thus the continued logarithm at $-1$
			is $l_{\ga_+}(-1)=\pi i$.
			
			Let $\Y_{M^1,M^2}$ be the universal intertwining operator satisfying
			\[
			\Y_{M^1,M^2}(v^1,e^0)v^2=v^1\bt v^2.
			\]
			With the above choice of branch, the HLZ braiding is characterized by
			\begin{equation}\label{eq:HLZ-braiding}
				\overline{c_{M^1,M^2}}(v^1\bt v^2)
				=e^{L(-1)}\Y_{M^2,M^1}(v^2,e^{\pi i})v^1.
			\end{equation}
			Indeed, the commutativity isomorphism first sends
			$v^1\bt v^2$ to $e^{L(-1)}(v^2\bt_{P(-1)}v^1)$, and parallel
			transport from $-1$ to $1$ in the upper half-plane sends
			$v^2\bt_{P(-1)}v^1$ to
			$\Y_{M^2,M^1}(v^2,e^{\pi i})v^1$; see
			\cite[Section~12.2]{HLZ}.
			
			Following \cite[Theorem~6.8]{Liu26}, we define
			\begin{equation}\label{eq:A-tensor-product}
				S\bt T:=\Phi^\L(S)\odot \Phi^\L(T),\qquad
				f\bt g:=\Phi^\L(f)\odot \Phi^\L(g),
				\quad S,T\in\Mod_{\mathrm{fd}}(A(V)).
			\end{equation}
			By Proposition~\ref{prop:condforC1V}, \eqref{eq:A-tensor-product}
			is a bifunctor on $\Mod_{\mathrm{fd}}(A(V))$. We use $\bt$ for the tensor products in both categories, with
			the intended product determined by its arguments. We take the unit object to be
			$\1_{A(V)}:=\Om(V)=\C\vac$. Denote the counit of the adjoint equivalence $(\Phi^\L\dashv\Om)$ by
			$\varepsilon:\Phi^\L\Om\Rightarrow\Id_{\CC_1(V)}$, and define
			\begin{equation}\label{eq:J-tensorator}
				\begin{aligned}
					J_{S,T}&:=\varepsilon_{\Phi^\L(S)\bt \Phi^\L(T)}\circ \Phi^\L(\eta_1^{\Phi^\L(S),\Phi^\L(T)}):
					\Phi^\L(S\bt T)\xrightarrow{\ \simeq\ }\Phi^\L(S)\bt \Phi^\L(T),\\
					J_0&:=\varepsilon_V:
					\Phi^\L(\1_{A(V)})\xrightarrow{\ \simeq\ }V.
				\end{aligned}
			\end{equation}
			By Theorem~\ref{thm:eta-explicit} and the naturality of
			$\varepsilon$, these maps are natural isomorphisms.
			Let $\alpha$, $\lambda$, and $\rho$ be the associativity and
			unit constraints of $\CC_1(V)$, and put
			$J^{\mathrm L}_{S,T,U}=(J_{S,T}\bt\Id_{\Phi^\L(U)})\circ J_{S\bt T,U}$ and
			$J^{\mathrm R}_{S,T,U}=(\Id_{\Phi^\L(S)}\bt J_{T,U})\circ J_{S,T\bt U}$.
			Since $\Phi^\L$ is fully faithful, there exist unique natural
			isomorphisms $\alpha^{A(V)}$, $\lambda^{A(V)}$, and
			$\rho^{A(V)}$ satisfying
			\begin{equation}\label{eq:transported-constraints}
				\begin{aligned}
					\Phi^\L(\alpha^{A(V)}_{S,T,U})
					&=(J^{\mathrm R}_{S,T,U})^{-1}
					\circ\alpha_{\Phi^\L(S),\Phi^\L(T),\Phi^\L(U)}\circ J^{\mathrm L}_{S,T,U},\\
					\Phi^\L(\lambda^{A(V)}_S)
					&=\lambda_{\Phi^\L(S)}\circ(J_0\bt\Id_{\Phi^\L(S)})
					\circ J_{\1_{A(V)},S},\\
					\Phi^\L(\rho^{A(V)}_S)
					&=\rho_{\Phi^\L(S)}\circ(\Id_{\Phi^\L(S)}\bt J_0)
					\circ J_{S,\1_{A(V)}}.
				\end{aligned}
			\end{equation}
			Finally, we define
			$\R_{S,T}:=\R_{\Phi^\L(S),\Phi^\L(T)}:S\bt T\xrightarrow{\ \simeq\ }T\bt S$, see \eqref{eq:braiding-formula}.
			
			\begin{thm}\label{thm:braided-tensor-structure}
				Under the assumptions of this section, we have the following results.
				\begin{enumerate}
					\item For $M^1,M^2\in\CC_1(V)$, we have
					\begin{equation}\label{eq:key-identity}
						\eta_1^{M^2,M^1}\circ\R_{M^1,M^2}
						=\Om(c_{M^1,M^2})\circ\eta_1^{M^1,M^2}.
					\end{equation}
					In other words, $\R_{M^1,M^2}$ is the restriction of the HLZ braiding $c_{M^1,M^2}$ to bottom degrees, and it is the unique homomorphism of left $A(V)$-modules satisfying \eqref{eq:key-identity}.
					\item The tensor product \eqref{eq:A-tensor-product}, the
					unit $\1_{A(V)}$, the constraints
					\eqref{eq:transported-constraints}, and the isomorphisms $\R_{S,T}$ give
					$\Mod_{\mathrm{fd}}(A(V))$ a braided tensor category structure. In particular, $\R$ satisfies the hexagon identities, and $(\Phi^\L,J^{-1},J_0^{-1})$ is a braided tensor
					equivalence:
					\begin{equation}\label{eq:braided-functor-compatibility}
						J_{T,S}\circ \Phi^\L(\R_{S,T})
						=c_{\Phi^\L(S),\Phi^\L(T)}\circ J_{S,T}.
					\end{equation}
					\item The twist $\theta_M=e^{2\pi iL(0)}$ on $\CC_1(V)$
					corresponds to $\theta_S=e^{2\pi i[\om]}\big|_S$.
					If $\CC_1(V)$ is rigid, ribbon, or modular, then $\Mod_{\mathrm{fd}}(A(V))$ has the same structure with
					the transported duals and twist. In particular, if $V$ is strongly
					rational, then $\Phi^\L$ is an equivalence of modular tensor
					categories.
				\end{enumerate}
			\end{thm}
			\begin{proof}
				(1) Write $\Y=\Y_{M^2,M^1}$. We first determine
				the bottom-degree component of $\Y(v^2,e^{\pi i})v^1$.
				By linearity, we may assume that $v^1$ and $v^2$ are $L(0)$-eigenvectors.
				Write
				\[
				\Y(v^2,z)v^1
				=\sum_{n\in\C}v^2_{(n)}v^1\,z^{-n-1}.
				\]
				By the $L(0)$-bracket identity for intertwining operators, we have
				\[
				L(0)\bigl(v^2_{(n)}v^1\bigr)
				=(\wt v^1+\wt v^2-n-1)v^2_{(n)}v^1.
				\]
				Let $p_h$ be the projection of $\Om(M^2\bt M^1)$ onto the
				$h$-eigenspace of $[\om]$. It follows that only the term with
				$n=\wt v^1+\wt v^2-h-1$ can contribute to
				$p_h\pi_0\Y(v^2,z)v^1$. Thus we have
				\[
				p_h\pi_0\Y(v^2,e^{\pi i})v^1
				=e^{\pi i(h-\wt v^1-\wt v^2)}
				p_h\pi_0\Y(v^2,e^0)v^1.
				\]
				Since $[\om]$ acts semisimply on the finite-dimensional space
				$\Om(M^2\bt M^1)$, we sum over its eigenvalues. Since $\pi_0\Y(v^2,e^0)v^1=\pi_0(v^2\bt v^1)=\eta_1(v^2\odot v^1)$ by Theorem~\ref{thm:eta-explicit}, we obtain
				\begin{equation}\label{eq:bottom-half-turn}
					\pi_0\Y(v^2,e^{\pi i})v^1
					=e^{\pi i[\om]}\,
					\eta_1\left(e^{-\pi iL(0)}v^2\odot
					e^{-\pi iL(0)}v^1\right).
				\end{equation}
				By linearity, the same identity holds for arbitrary $v^1,v^2$.
				Then by \eqref{eq:pi0-properties}, \eqref{eq:HLZ-braiding}, \eqref{eq:bottom-half-turn}, and the $A(V)$-linearity of $\eta_1$, we have
				\begin{align*}
					\Om(c_{M^1,M^2})\,\eta_1(v^1\odot v^2)
					&=\pi_0\,\overline{c_{M^1,M^2}}(v^1\bt v^2)
					=\pi_0e^{L(-1)}\Y(v^2,e^{\pi i})v^1
					=\pi_0\Y(v^2,e^{\pi i})v^1\\
					&=\eta_1\left(e^{\pi i[\om]}
					\left(e^{-\pi iL(0)}v^2\odot
					e^{-\pi iL(0)}v^1\right)\right)
					=\eta_1\,\R_{M^1,M^2}(v^1\odot v^2).
				\end{align*}
				This proves \eqref{eq:key-identity}. The uniqueness follows since $\eta_1^{M^1,M^2}$ and $\eta_1^{M^2,M^1}$ are isomorphisms.
				
				(2) Applying (1) to $\Phi^\L(S),\Phi^\L(T)$ and using the naturality of the counit,
				we have
				\begin{align*}
					J_{T,S}\circ \Phi^\L(\R_{S,T})
					&=\varepsilon_{\Phi^\L(T)\bt \Phi^\L(S)}
					\circ \Phi^\L(\eta_1^{\Phi^\L(T),\Phi^\L(S)}\circ\R_{\Phi^\L(S),\Phi^\L(T)})\\
					&=\varepsilon_{\Phi^\L(T)\bt \Phi^\L(S)}
					\circ \Phi^\L\Om(c_{\Phi^\L(S),\Phi^\L(T)})\circ \Phi^\L(\eta_1^{\Phi^\L(S),\Phi^\L(T)})\\
					&=c_{\Phi^\L(S),\Phi^\L(T)}\circ\varepsilon_{\Phi^\L(S)\bt \Phi^\L(T)}
					\circ \Phi^\L(\eta_1^{\Phi^\L(S),\Phi^\L(T)})
					=c_{\Phi^\L(S),\Phi^\L(T)}\circ J_{S,T}.
				\end{align*}
				This proves \eqref{eq:braided-functor-compatibility}.
				By \eqref{eq:transported-constraints}, \eqref{eq:braided-functor-compatibility}, and the naturality of $J$, the maps $J$ identify the images under $\Phi^\L$ of the pentagon, triangle, and hexagon diagrams in $\Mod_{\mathrm{fd}}(A(V))$ with the corresponding diagrams in $\CC_1(V)$, which commute. Since $\Phi^\L$ is faithful, the original diagrams also commute.
				
				(3) For homogeneous $a\in V$ and $n\in\Z$, the identity
				$[L(0),a_{(n)}]=(\wt a-n-1)a_{(n)}$
				implies that $e^{2\pi iL(0)}$ commutes with all modes of $V$.
				Thus it is a natural $V$-module automorphism, whose restriction
				to $\Om(\Phi^\L(S))=S$ is $e^{2\pi i[\om]}|_S$. Since $\Phi^\L(S)$ is generated by $S$, we have $\Phi^\L(e^{2\pi i[\om]}|_S)=e^{2\pi iL(0)}|_{\Phi^\L(S)}$.
		Moreover, by \eqref{eq:HLZ-braiding} and the $L(0)$-conjugation property, 
\begin{align*}
\overline{c_{M^2,M^1}\circ c_{M^1,M^2}}\,\Y_{M^1,M^2}(v^1,x)v^2&=\Y_{M^1,M^2}(v^1,e^{2\pi i}x)v^2\\
&=e^{2\pi iL(0)}\Y_{M^1,M^2}(e^{-2\pi iL(0)}v^1,x)e^{-2\pi iL(0)}v^2,
\end{align*} 
which is the balancing identity for $\theta=e^{2\pi iL(0)}$, see also \cite{H08}.
				The duals, and the ribbon and modular structures, are transported along the braided tensor equivalence $\Phi^\L$ in the same way.
				Finally, if $V$ is strongly rational, then
				$\CC_1(V)=\Ord(V)$ with these structures is modular by
				\cite{H08}.
			\end{proof}
			
			If we use the lower half-plane, then
			$c_{M^1,M^2}$ is replaced by $c_{M^2,M^1}^{-1}$ and $\R_{M^1,M^2}$
			is replaced by $R_+^{M^1,M^2}=(\R_{M^2,M^1})^{-1}$.
			

			\section{Explicit braiding isomorphism for affine VOAs at positive integral levels}\label{sec:affine}
In this section, we describe the fusion tensor product and the braiding isomorphism on $\Mod_\fd(A(L_k(\g)))$ for $k\in\Z_{>0}$ explicitly in terms of finite-dimensional $\g$-modules and the Casimir tensor of $\g$, see Theorems~\ref{thm:affine-odot-quotient} and \ref{thm:VOA-Drinfeld-braiding}.
			\subsection{\texorpdfstring{Description of the $\odot$-tensor product for the affine VOA $L_k(\g)$}{Description of the odot-tensor product for the affine VOA L\_k(g)}}
			Throughout this section, let $\g$ be a finite-dimensional simple Lie algebra over $\C$, and let $(\cdot\mid\cdot)$ be the invariant bilinear form on $\g$ normalized by $(\theta\mid\theta)=2$, where $\theta$ is the highest root. Fix root vectors $e_\theta\in \g_\theta$ and $f_\theta\in \g_{-\theta}$. Let $h^\vee$ be the dual Coxeter number of $\g$, and $k\in\Z_{>0}$. Put
			\begin{equation}\label{eq:affine-kappa}
				\kappa=k+h^\vee,
				\qquad
				P_+^k=\{\la\in P_+:(\la\mid\theta)\leq k\}.
			\end{equation}
			Let $V^k(\g)$ be the universal affine VOA at level $k$, and let
			\[
			\hat{\g}=\g\o\C[t,t^{-1}]\oplus\C K,
			\quad \mathrm{with}\quad
			a(n)=a\o t^n,
			\]
			for $a\in \g$ and $n\in \Z$.
			For $\la\in P^k_+$, denote by
			\[
			V^k(\g,\la)
			=U(\hat{\g})\o_{U(\g[t]\oplus\C K)}L(\la)
			\]
			the Weyl module associated to $L(\la)$, where $\g\o t\C[t]$ acts trivially on $L(\la)$ and $K$ acts as $k$. $ V^k(\g,\la)$ is a $V^k(\g)$-module with a unique maximal submodule $W^k(\g,\lambda)$.
			
			Since $k\in \Z_{>0}$ and $\la\in P^k_+$, $W^k(\g,\lambda)$ is generated by one singular vector $e_\theta(-1)^{k-(\la|\theta)+1}v_\la$ \cite{Kac90}.
			Denote the irreducible quotient $\hat{\g}$ or $V^k(\g)$-module by
			\begin{equation}\label{eq:affineirr}
				\hat{L}_k(\la):=\frac{V^k(\g,\lambda)}{W^k(\g,\lambda)}=\frac{V^k(\g,\lambda)}{U(\hat{\g}).(e_\theta(-1)^{k-(\la|\theta)+1}v_\la)}.
			\end{equation}
			Let $L_k(\g)=V^k(\g)/W^k(\g)$ be the affine VOA at level $k$ \cite{FZ92}, which is a strongly rational VOA \cite{DLM97}, with irreducible modules
			\[
			\mathscr{W}=\{\hat{L}_k(\la): \la\in P^k_+\},\quad \mathrm{where }\quad \Om(\hat{L}_k(\la))=\hat{L}_k(\la)(0)=L(\la).
			\]
			Note that the Zhu algebras of $V^k(\g)$ and $L_k(\g)$ have the following identifications
			\begin{equation}\label{eq:affine-Zhu}
				\begin{aligned}
					& A(V^k(\g))\xrightarrow{\simeq} U(\g),\quad [a^1(-n_1-1)\ds a^r(-n_r-1)\vac]\mapsto (-1)^{n_1+\ds +n_r}a^r\ds a^1,\\
					& A(L_k(\g))\cong U(\g)/I_k, \qquad
					I_k=\<e_\theta^{k+1}\>=U(\g)e_\theta^{k+1}U(\g).
				\end{aligned}
			\end{equation}
			
			In order to describe the braiding isomorphism on $\Mod_{\mathrm{fd}}(A(L_k(\g)))$, we first give a more concrete description of the $\odot$-tensor product $\hat{L}_k(\la)\odot \hat{L}_k(\mu)$ using \eqref{eq:affineirr} and its defining relations \eqref{eq:actionfogeneralspann}.
			
			The following Theorem is a variation of \cite[Proposition 6.5]{AL26}.
			\begin{thm}\label{thm:affine-odot-quotient}
				For $\la,\mu\in P_+^k$, let $J_{\la,\mu}$ be the subspace of $L(\la)\o L(\mu)$ defined by
				\begin{equation}\label{eq:J-lambda-mu}
					J_{\la,\mu}
					=v_\la\o e_\theta^{k-(\la|\theta)+1}L(\mu)
					+e_\theta^{k-(\mu|\theta)+1}L(\la)\o v_\mu,
				\end{equation}
				and let $U(\g).J_{\la,\mu}$ be the $\g$-submodule generated by $J_{\la,\mu}$ with respect to the diagonal action.
				Then we have an isomorphism of $A(L_k(\g))$-modules
				\begin{equation}\label{eq:affine-odot-quotient}
					\tilde\varphi: \frac{L(\la)\o_\C L(\mu)}{U(\g).J_{\la,\mu}} \xlongrightarrow{\simeq} \hat{L}_k(\la)\odot \hat{L}_k(\mu),\qquad \overline{u\o v}\mapsto u\odot v.
				\end{equation}
				Here the diagonal $\g$-action on $L(\la)\o_\C L(\mu)$
				descends to an action of
				$A(L_k(\g))\cong U(\g)/I_k$ on the quotient.
			\end{thm}
			\begin{proof}
				Since the affine VOA $L_k(\g)$ is strongly rational, it was proved in \cite[Corollary 6.3]{Liu26} that the following map is an isomorphism of left $A(L_k(\g))$-modules
				\begin{equation}\label{eq:FZandodot1}
					A(\hat{L}_k(\la))\otimes_{A(L_k(\g))} L(\mu) \xlongrightarrow{\simeq} \hat{L}_k(\la)\odot \hat{L}_k(\mu),\quad [\al]\o v\mapsto \al\odot v,
				\end{equation}
				where $\al\in\hat{L}_k(\la)$, $v\in L(\mu)$, and $A(\hat{L}_k(\la))$ is Frenkel-Zhu's bimodule. On the other hand, it was proved in \cite[Theorem 3.2.2]{FZ92} that the following map is an isomorphism
				\begin{equation}\label{eq:FZandodot2}
					\frac{L(\la)\o_\C L(\mu)}{U(\g).\left(v_\la\o e_\theta^{k-(\la|\theta)+1}L(\mu)\right)} \xlongrightarrow{\simeq} A(\hat{L}_k(\la))\otimes_{A(L_k(\g))} L(\mu),\quad \overline{u\o v}\mapsto [u]\o v.
				\end{equation}
				By composing \eqref{eq:FZandodot1} and \eqref{eq:FZandodot2}, we see that there exists a surjective map
				\[
				\varphi: L(\la)\o_\C L(\mu)\twoheadrightarrow \hat{L}_k(\la)\odot \hat{L}_k(\mu),\quad u\o v\mapsto u\odot v,
				\]
				with $\ker(\varphi)= U(\g).(v_\la\o e_\theta^{k-(\la|\theta)+1}L(\mu))$. We claim that $\ker(\varphi)=U(\g).J_{\la,\mu}$.
				
				Indeed, by the definition of $J_{\la,\mu}$ \eqref{eq:J-lambda-mu}, it is clear that $\ker(\varphi)\ssq U(\g).J_{\la,\mu}$. On the other hand, it is clear that
				$v_\la\o e_\theta^{k-(\la|\theta)+1}v\in \ker(\varphi)$ for all $v\in L(\mu)$.
				Now apply \eqref{eq:actionfogeneralspann} to $a=e_\theta$, $t=0,s=2$, and $\wt e_\theta=1$, we have the following relation in $\hat{L}_k(\la)\odot \hat{L}_k(\mu)$
				\[
				u\odot e_\theta(-1) v=-e_\theta(0)u\odot v=-e_\theta u\odot v. \]
				Since $e_\theta(-1)^{k-(\mu|\theta)+1}v_\mu=0$ in $\hat{L}_k(\mu)$ \eqref{eq:affineirr}, it follows that for any $u\in L(\la)$, we have
				\[
				e_\theta^{k-(\mu|\theta)+1}u\odot v_\mu=(-1)^{k-(\mu|\theta)+1} u\odot e_\theta(-1)^{k-(\mu|\theta)+1}v_\mu=(-1)^{k-(\mu|\theta)+1} u\odot 0=0.
				\]
				In other words, $\varphi(e_\theta^{k-(\mu|\theta)+1}u\o v_\mu)=0$, and so $e_\theta^{k-(\mu|\theta)+1}u\o v_\mu\in \ker(\varphi)$. Thus,
				\[
				J_{\la,\mu}=v_\la\o e_\theta^{k-(\la|\theta)+1}L(\mu)+ e_\theta^{k-(\mu|\theta)+1}L(\la)\o v_\mu\ssq \ker(\varphi). \]
				Since $\ker(\varphi)$ is a $U(\g)$-submodule of $L(\la)\o_\C L(\mu)$, we have $U(\g).J_{\la,\mu}\ssq \ker(\varphi)$. This proves the isomorphism \eqref{eq:affine-odot-quotient}.
			\end{proof}
			For the simplicity of notation, we write
			\begin{equation}\label{eq:Qlamu}
				T_{\la,\mu}:=
				\frac{L(\la)\o_\C L(\mu)}{U(\g).J_{\la,\mu}},
			\end{equation}
			and denote the quotient map by
			\[
			\pi_{\la,\mu}: L(\la)\o_\C L(\mu)\twoheadrightarrow T_{\la,\mu},\quad u\o v\mapsto \overline{u\o v}.
			\]
			The following Corollary follows immediately from \eqref{eq:A-tensor-product}, Theorem~\ref{thm:braided-tensor-structure}, and Theorem~\ref{thm:affine-odot-quotient}.
			\begin{coro}
				Let $k\in \Z_{>0}$. The category $\Mod_{\mathrm{fd}}(U(\g)/I_k)$ \eqref{eq:affine-Zhu} is a modular tensor category, with fusion tensor product given by
				\begin{equation}
					L(\la) \bt L(\mu)=T_{\la,\mu}=\frac{L(\la)\o_\C L(\mu)}{U(\g).J_{\la,\mu}},\quad \la,\mu \in P^k_+.
				\end{equation}
				Moreover, the action of $U(\g)/I_k$ on $T_{\la,\mu}$ is induced by the diagonal $\g$-action on $L(\la)\o_\C L(\mu)$.
			\end{coro}
			
			\subsection{\texorpdfstring{The braiding isomorphism on $\Mod_{\fd}(A(L_k(\g)))$}{The braiding isomorphism on Mod\_fd(A(L\_k(g)))}}
			
			Under the isomorphism \eqref{eq:affine-odot-quotient}, the braiding isomorphism $\R=\mathcal R_{\hat{L}_k(\la), \hat{L}_k(\mu)}$ translates to an isomorphism $\R_{\la,\mu}: T_{\la,\mu}\xrightarrow{\simeq} T_{\mu,\la}$ that fits into the following diagram:
			\begin{equation}\label{eq:affinediagram}
				\begin{tikzcd}
					\hat{L}_k(\la)\odot \hat{L}_k(\mu) \arrow[r,"\R"] \arrow[d,"\simeq","\eqref{eq:affine-odot-quotient}"']& \hat{L}_k(\mu)\odot \hat{L}_k(\la) \arrow[d,"\eqref{eq:affine-odot-quotient}","\simeq"']\\
					T_{\la,\mu}\arrow[r,dashed,"\R_{\la,\mu}"] & T_{\mu,\la}
				\end{tikzcd}
			\end{equation}
			
			In order to give an explicit description of $\R_{\la,\mu}$, in view of \eqref{eq:braiding-formula}, we first need to determine the action of $[\om]\in A(L_k(\g))$ on $T_{\la,\mu}$.
			The Virasoro element $\om $ of the affine VOA $L_k(\g)$ is given by the Segal-Sugawara construction. More precisely, choose dual bases $\{x_a\}_{a=1}^{\dim\g}$ and $\{x^a\}_{a=1}^{\dim\g}$ of $\g$ with respect to $(\cdot\mid\cdot)$, recall that
			$$\omega=\frac{1}{2\kappa}\sum_a x^a(-1)x_a(-1)\vac\in L_k(\g)_2.$$
			Consider the following two elements
			\begin{equation}\label{eq:CandOmega}
				C=\sum_a x_ax^a\in U(\g),
				\quad \Omega=\sum_a x_a\o x^a\in U(\g)\o U(\g).
			\end{equation}
			Let $\Delta: U(\g)\to U(\g)\o_\C U(\g)$ be the standard coproduct, then
			\begin{equation}\label{eq:coproduct-Casimir}
				\Delta(C)=C\o 1+1\o C+2\Omega.
			\end{equation}
			Note that $C$ is the Casimir element, which acts on the irreducible $\g$-module $L(\la)$ by the scalar $(\la|\la+2\rho)$. Then
			\[
			\frac{1}{2\kappa}C|_{L(\la)}=\frac{(\la|\la+2\rho)}{2\kappa}\cdot \Id_{L(\la)}=h_\la\cdot \Id_{L(\la)}, \quad \la\in P^k_+.
			\]
			Under the isomorphism of Zhu algebra \eqref{eq:affine-Zhu}, we have
			\begin{equation}\label{eq:[om]inquotient}
				[\om]=\frac{1}{2\kappa}\sum_a x_ax^a+I_k= \frac{1}{2\kappa} \bar{C}\in U(\g)/I_k.
			\end{equation}
			
			\begin{lm}\label{prop:affine-omega-action}
				Under the quotient identification \eqref{eq:affine-odot-quotient}, the action of $[\omega]\in U(\g)/I_k$ on $T_{\la,\mu}$ \eqref{eq:Qlamu} is given by
				\begin{equation}\label{eq:affine-omega-action}
					[\omega].\pi_{\la,\mu}(u\o v)
					= \pi_{\la,\mu}\left(
					(h_\la+h_\mu)u\o v
					+\frac{1}{\kappa}\Omega_{\la,\mu}(u\o v)
					\right),
				\end{equation}
				where $u\in L(\la)$, $v\in L(\mu)$, and $\Omega_{\la,\mu}$ is defined by
				\begin{equation}\label{eq:Omlamu}
					U(\g)\o_\C U(\g)\to \End(L(\la)\o_\C L(\mu)),\quad \Om\mapsto \Om_{\la,\mu}.
				\end{equation}
				In particular, $U(\g).J_{\la,\mu}\subset L(\la)\o_{\C} L(\mu)$ is invariant under $\Omega_{\la,\mu}$.
			\end{lm}
			\begin{proof}
				By Theorem~\ref{thm:affine-odot-quotient}, the $U(\g)/I_k$-action on $T_{\la,\mu}$ is induced by the diagonal $\g$-action. Then for any $u\o v\in L(\la)\o_\C L(\mu)$, it follows from \eqref{eq:coproduct-Casimir} and \eqref{eq:[om]inquotient} that
				\begin{align*}
					[\omega]. \pi_{\la,\mu}(u\o v)&= [\omega]. \overline{u\o v}=\frac{1}{2\kappa}\overline{\Delta(C).(u\o v)}\\
					&=\frac{1}{2\kappa} \pi_{\la,\mu}\left(Cu\o v+u\o Cv+ 2\Om_{\la,\mu}(u\o v)\right) \\
					&= \pi_{\la,\mu}\left(
					(h_\la+h_\mu)u\o v
					+\frac{1}{\kappa}\Omega_{\la,\mu}(u\o v)
					\right),
				\end{align*}
				which proves \eqref{eq:affine-omega-action}. Since $U(\g).J_{\la,\mu}$ is a diagonal $\g$-submodule, it is invariant under $\Delta(C)$. On $L(\la)\o L(\mu)$, the operator $C\o 1+1\o C$ is the scalar $2\kappa (h_\la+h_\mu)\cdot \Id$. Hence \eqref{eq:coproduct-Casimir} implies that $U(\g).J_{\la,\mu}$ is invariant under $\Omega_{\la,\mu}$.
			\end{proof}
			
			\begin{proof}[Second proof of \eqref{eq:affine-omega-action}]
				There is an alternative way to prove \eqref{eq:affine-omega-action} using the relations in $\odot$-tensor product only, without invoking the identification of Zhu algebras \eqref{eq:affine-Zhu}.
				
				Indeed, by \eqref{A(V)action}, the action of $[\om]\in A(L_k(\g))$ on $\hat{L}_k(\la)\odot \hat{L}_k(\mu)$ is given by
				\begin{align*}
					[\om].(u\odot v)&=\sum_{j\geq 0} \binom{1}{j} \om_{(j)} u\odot v+u\odot o(\om)v\\
					&=L(0)u\odot v+L(-1)u\odot v+u\odot L(0)v,
				\end{align*}
				for $u\in L(\la)=\hat{L}_k(\la)(0),v\in L(\mu)=\hat{L}_k(\mu)(0)$. Since $L(0)$ acts semisimply on the irreducible ordinary module $\hat{L}_k(\la)$, we have $L(0)u=h_\la u$ and $L(0)v=h_\mu v$. Moreover, it follows from the Sugawara formula \cite{Kac90} and $x^a(i)u=0$ for $i>0$ that
				\[
				L(-1)u=\frac{1}{\kappa} \sum_a \sum_{i\geq 0} x_a(-1-i)x^a(i)u=\frac{1}{\kappa} \sum_a x_a(-1)x^a(0)u.
				\]
				Apply \eqref{eq:actionfogeneralspann} to $x^a\in L_k(\g)_1$, $s=1$, and $t=1$, we have
				\begin{align*}
					L(-1)u\odot v&=-\frac{1}{\kappa} \sum_a \sum_{j\geq 0}\binom{-1}{j} (-1)^{-1-j} x^a(0)u\odot x_a(j)v\\
					&=\frac{1}{\kappa} \sum_a x^a(0)u\odot x_a(0)v=\frac{1}{\kappa} \sum_a x^au\odot x_av.
				\end{align*}
				Since $\tilde\varphi$ is an $A(V)$-module isomorphism, it follows from \eqref{eq:CandOmega} that
				\begin{align*}
					[\om]. \pi_{\la,\mu}(u\o v)&=\tilde \varphi^{-1} ([\om].(u\odot v))=\tilde\varphi^{-1}\left((h_\la+h_\mu)u\odot v+\frac{1}{\kappa} \sum_a x^au\odot x_a v\right)\\
					&=\pi_{\la,\mu}\left((h_\la+h_\mu)u\o v+\frac{1}{\kappa} \Om_{\la,\mu}(u\o v) \right),
				\end{align*}
				where we used $\sum_a x^au\o x_a v=\Om_{\la,\mu}(u\o v)$. This
				proves \eqref{eq:affine-omega-action}.
			\end{proof}
			
			Since $(h_\la+h_\mu)\cdot \Id$ and $\frac{1}{\kappa}\Om_{\la,\mu}$ are commutative when viewed as operators on $L(\la)\o_\C L(\mu)$, it follows from \eqref{eq:affine-omega-action} that
			\begin{equation}\label{eq:epiiomcalc}
				e^{\pi i [\om]}\pi_{\la,\mu}(u\o v)=\pi_{\la,\mu}\left(e^{\pi i (h_\la+h_\mu)}\exp\left(\frac{\pi i }{\kappa} \Om_{\la,\mu}\right)(u\o v) \right).
			\end{equation}
			Now we define a tensor-swapping operator. Let
			\[
			\tau:L(\la)\o_\C L(\mu)\ra L(\mu)\o_\C L(\la),
			\quad\tau(u\o v)=v\o u.
			\]
			It is easy to see from \eqref{eq:CandOmega} that
			$
			\tau\circ \Omega_{\la,\mu}
			=\Omega_{\mu,\la}\circ \tau.
			$
			Then by \eqref{eq:braiding-formula}, the corresponding elements in vector spaces of the diagram \eqref{eq:affinediagram} are given by
			\[
			\begin{tikzcd}
				u\odot v\arrow[r,mapsto,"\mathcal R"] \arrow[d,mapsto,"\eqref{eq:affine-odot-quotient}"'] &
				e^{\pi i[\om]}.
				\left(e^{-\pi iL(0)}v\odot e^{-\pi iL(0)}u\right)
				\arrow[d,mapsto,"\eqref{eq:affine-odot-quotient}"] \\ \pi_{\la,\mu}(u\o v)
				\arrow[r,mapsto,dashed,"\mathcal R_{\la,\mu}"]
				&
				\tilde\varphi^{-1} \left( e^{\pi i[\om]}.
				\left(e^{-\pi iL(0)}v\odot e^{-\pi iL(0)}u\right)\right).
			\end{tikzcd}
			\]
			It follows from \eqref{eq:affine-odot-quotient} and \eqref{eq:epiiomcalc} that
			\begin{align*}
				&\tilde \varphi^{-1} \left( e^{\pi i[\om]}.
				\left(e^{-\pi iL(0)}v\odot e^{-\pi iL(0)}u\right)\right)=e^{\pi i [\om]} \pi_{\mu,\la} \left(e^{-\pi i h_\mu}v\o e^{-\pi i h_\la}u\right)\\
				&=\pi_{\mu,\la} \left(\exp\left(\frac{\pi i }{\kappa} \Om_{\mu,\la}\right) (v\o u)\right)=\pi_{\mu,\la} \left(\exp\left(\frac{\pi i }{\kappa} \Om_{\mu,\la}\right) \tau(u\o v)\right)\\
				&=\pi_{\mu,\la}\left(\left(\tau\circ \exp\left(\frac{\pi i }{\kappa} \Om_{\la,\mu}\right)\right)(u\o v) \right).
			\end{align*}
			Therefore, we proved the following Theorem.
			\begin{thm}\label{thm:VOA-Drinfeld-braiding}
				Under the identifications $\hat{L}_k(\la)\odot \hat{L}_k(\mu)\simeq T_{\la,\mu}=(L(\la)\o_\C L(\mu))/U(\g).J_{\la,\mu}$ in Theorem~\ref{thm:affine-odot-quotient}, the braiding isomorphism on the modular tensor category $\Mod_\fd(A(L_k(\g)))\simeq \Mod_{\fd}(U(\g)/I_k)$ is given by
				\begin{equation}\label{eq:VOA-Drinfeld-braiding}
					\R_{\la,\mu}: T_{\la,\mu}\xrightarrow{\simeq} T_{\mu,\la},\quad \R_{\la,\mu}\left(\pi_{\la,\mu}(u\o v)\right)
					=\pi_{\mu,\la}\left(\left(\tau\circ \exp\left(\frac{\pi i }{\kappa} \Om_{\la,\mu}\right)\right)(u\o v) \right).
				\end{equation}
				Equivalently, denote $q=e^{\pi i/\kappa}\in \C^\times$, then we can write
				\[
				\R_{\la,\mu}(\overline{u\o v})=\overline{ q^{\Om_{\mu,\la}} (v\o u)},
				\]
				where $\la,\mu\in P^k_+$, $u\in L(\la)$, $v\in L(\mu)$, and $\Om_{\la,\mu}$ is defined by \eqref{eq:Omlamu}.
			\end{thm}

\section{\texorpdfstring{Braiding isomorphism on $KL^k(\g)$ at generic levels}{Braiding isomorphism on KL\textasciicircum{}k(g) at generic levels}}\label{sec:generic}

Let $\g$, $(\cdot\mid\cdot)$, $h^\vee$, and $\hat{\g}$ be as in Section~\ref{sec:affine}, and let $k\in \C$ with $\kappa=k+h^\vee\neq 0$. For $\la\in P_+$, we denote by $V^k(\g,\la)$ the Weyl module associated to $L(\la)$, as in Section~\ref{sec:affine}. 

Throughout this section, we assume that either $k+h^\vee\notin \Q$, or $\g=\sl_2$ and $k=-2\pm\frac{1}{p}$ with $p\in\Z_{>0}$. In both cases, there is no singular vector in the vacuum module, that is, $L_k(\g)=V^k(\g)$ \cite{GK07} and $L_{-2\pm \frac{1}{p}}(\sl_2)=V^{-2\pm \frac{1}{p}}(\sl_2)$ \cite{AM26}. We write $V=L_k(\g)=V^k(\g)$. Then
$$A(V)\cong U(\g),\quad [a^1(-n_1-1)\ds a^r(-n_r-1)\vac]\mapsto (-1)^{n_1+\ds +n_r}a^r\ds a^1,$$
where $a^i\in \g$ for all $i$, and $n_1\geq \ds \geq n_r\geq 0$. Hence $\Mod_\fd(A(V))$ is the category of finite-dimensional $\g$-modules, which is semisimple. Moreover, the category $\CC_1(V)=KL^k(\g)$ is semisimple and $\Phi^\L(\Mod_\fd(U(\g)))\ssq KL^k(\g)$, see Subsection~\ref{sec:2.2.3} and Example~\ref{ex:levels}. Hence Theorems~\ref{thm:eta-explicit}, \ref{thm:Rmatrix}, and \ref{thm:braided-tensor-structure} apply to $V$.

\subsection{\texorpdfstring{Description of the $\odot$-tensor product at generic levels}{Description of the odot-tensor product at generic levels}}
			\begin{lm}\label{lm:affine-span}
				Let $S,T\in \Mod_\fd(A(L_k(\g)))$. Then $ \Phi^\L(S)\odot \Phi^\L(T)=\spn\{s\odot t : s\in S, t\in T\}. $ In particular, the canonical map
				\begin{equation}\label{eq:cansurj}
					S\otimes_\C T\to \Phi^\L(S)\odot \Phi^\L(T)=S\bt T,\quad s\o t\mapsto s\odot t\end{equation}
				is surjective. Furthermore, the action of $A(L_k(\g))$ on $\Phi^\L(S)\odot \Phi^\L(T)$ is given as follows: for any $a\in \g$, with $[a]=a(-1)\vac+O(V)\in A(L_k(\g))=U(\g)$,
				\begin{equation}\label{eq:Aactiononodot}
					[a].(s\odot t)=a(0)s\odot t+ s\odot a(0)t,\quad s\in S,\ t\in T.
				\end{equation}
			\end{lm}
			\begin{proof}
				It follows from \cite[Lemma 4.5]{Liu26} that
				$\Phi^\L(S)\odot \Phi^\L(T)=\spn\{v\odot t: v\in \Phi^\L(S), t\in T\}$. Moreover, since $\mathscr{U}(L_k(\g))\cong \hat{U}_k(\hat{\g})/\hat{I}_k$, where $\hat{I}_k$ is the closed two-sided ideal generated by the image of the singular vectors of $V^k(\g)$ in $\hat{U}_k(\hat{\g})$ \cite{FZ92,FBZ04}, then we have
				\[
				\Phi^\L(S)=\spn\{a^1(-n_1)\ds a^r(-n_r)s: a^i\in \g, n_1\geq \ds \geq n_r\geq 1,s\in S \}.
				\]
				Note that $a^i\in \g\cong V_1$.
				Then by \eqref{eq:finalodotrel} with $\wt a^1=1$, $s=1$, and $t=n_1$, we have
				\begin{align*}
					a^1(-n_1)\ds a^r(-n_r)s\odot t&= -\sum_{j\geq 0}\binom{-n_1}{j} (-1)^{-n_1-j} (a^2(-n_2)\ds a^r(-n_r)s)\odot a^1(j)t\\
					&=(-1)^{-n_1-1} (a^2(-n_2)\ds a^r(-n_r)s)\odot a^1(0)t,
				\end{align*}
				since $t\in T=\Om(\Phi^\L(T))$ and $a^1(j)t=0$ for $j\geq 1$. Proceed inductively, we have
				\[
				a^1(-n_1)\ds a^r(-n_r)s\odot t=(-1)^{-n_1-\ds -n_r-r} s\odot (a^r(0)\ds a^1(0)t),
				\]
				where $a^r(0)\ds a^1(0)t\in T$.
				Hence $\Phi^\L(S)\odot \Phi^\L(T)=S\odot T$. Finally, the action formula \eqref{eq:Aactiononodot} of $A(L_k(\g))$ on $s\odot t$ is given by \eqref{A(V)action}.
			\end{proof}
	The following Proposition generalizes \cite[Lemma 6.3]{AL26}.
			\begin{prop}\label{prop:generic-odot}
				Assume $V=V^k(\g)$, where either $k+h^\vee\notin\Q$,
				or $\g=\sl_2$ and $k=-2\pm\frac{1}{p}$ with
				$p\in\mathbb Z_{>0}$. Then we have
				\[
				\Phi^\L(L(\la))\cong V^k(\g,\la),
				\qquad \forall\,L(\la)\in\Mod_\fd(A(L_k(\g))).
				\]
				In particular, for any simple modules $L(\la),L(\mu)\in\Mod_\fd(A(L_k(\g)))$ the canonical map \eqref{eq:cansurj} is a natural isomorphism of $A(L_k(\g))\cong U(\g)$-modules
				\begin{equation}\label{eq:genericlevelodot}
					L(\la)\o_\C L(\mu)\simeq L(\la)\bt L(\mu),\quad s\o t\mapsto s\odot t,
				\end{equation}
				where the left-hand side carries the usual diagonal $\g$-action.
			\end{prop}
			\begin{proof}
				Since $\scrU(V^k(\g))\cong \hat{U}_k(\hat{\g})$ and $\deg(a(n))=-n$, for all $a\in \g=V_1$ and $n\in \Z$, then by \eqref{Verma},
				\[
				\Phi^\L(L(\la))=\left(\frac{\hat{U}_k(\hat{\g})}{ {\hat{U}_k(\hat{\g})}{\hat{U}_k(\hat{\g}_{<0})}} \right) \otimes_{\hat{U}(\hat{\g}_0)} L(\la)\cong \mathrm{Ind}^{\hat{\g}}_{\hat{\g}_{\geq 0}} L(\la)= V^k(\g,\lambda).
				\]
				Similar to the proof of \cite[Lemma 6.3]{AL26}, we define a linear map
				\[
				\tilde \varphi: V^k(\g,\la)\otimes_\C V^k(\g,\mu)\to L(\la)\o_\C L(\mu)
				\]
				on the tensor products of PBW basis elements of $V^k(\g,\la)$ and $V^k(\g,\mu)$ by induction on the total degree, as follows.
				
				For $u\in L(\la)$ and $v\in L(\mu)$, define $\tilde\varphi(u\o v):=u\o v\in L(\la)\o_\C L(\mu)$.
				Choose an ordered basis $\{a^1,\ds, a^n\}$ of $\g$. A PBW basis element of $V^k(\g,\la)$ of positive degree can be uniquely written as $a(-r)u$ for some $a=a^i$, $r\geq 1$, and $u$ a PBW basis element of lower degree. Define
				\begin{equation}\label{eq:rel1}
					\tilde{\varphi}(a(-r)u\o v):=-\sum_{j\geq 0}\binom{-r}{j} (-1)^{-r-j} \tilde{\varphi}(u\o a(j)v).
				\end{equation}
				Similarly, given a PBW basis element $a(-r)v$ of $V^k(\g,\mu)$, with $r\geq 1$, define
				\begin{equation}\label{eq:rel2}
					\tilde{\varphi}(u\o a(-r)v):=-\sum_{j\geq 0}\binom{-r}{j}\tilde \varphi (a(j)u\o v).
				\end{equation}
				Note that
				$$\deg(u\o v)<\deg (a(-r)u\o v),\quad \deg(a(j)u\o v)=-j+\deg (u\o v)<\deg(u\o a(-r)v),$$ for all $j\geq 0$, we can apply the induction hypothesis after rewriting
				$a(j)u$ and $a(j)v$ as linear combinations of PBW basis
				elements for all $j\geq 0$. Hence $\tilde\varphi$ is well-defined.
				
				Moreover, using the commutation relations in $\hat{\g}$ and induction on the total degree, one can show that \eqref{eq:rel1} and
				\eqref{eq:rel2} are compatible with PBW reordering and hold
				for arbitrary pairs of PBW basis elements. By bilinearity, \eqref{eq:rel1} and
				\eqref{eq:rel2} hold for all $u\in V^k(\g,\la)$ and $v\in V^k(\g,\mu)$.
				
				Now, we claim that given $b=a^1(-n_1)\ds a^r(-n_r)\vac\in V^k(\g)$, there exists a rational function
				$$F_b: V^k(\g,\la)\o_\C V^k(\g,\mu)\to (L(\la)\o_\C L(\mu))[z,z^{-1},(z-1)^{-1}],\quad u\o v\mapsto F_b(u\o v)(z),$$
				with poles only at $z=0$ and $z=1$, such that
				\begin{equation}\label{eq:Fb}
					\begin{aligned}
						\iota_{z=1}F_b(u\o v)(z)
						&=\tilde\varphi\bigl(Y(b,z-1)u\otimes v\bigr),\\
						\iota_{z=0}F_b(u\o v)(z)
						&=\tilde\varphi\bigl(u\otimes Y(b,z)v\bigr),\\
						\iota_{z=\infty} F_b(u\o v)(z)
						&\in z^{-\wt b}
						\bigl(L(\la)\otimes_\C L(\mu)\bigr)[\![z^{-1}]\!].
					\end{aligned}
				\end{equation}
				We use induction on the length $r$ of $b$ to construct $F_b$ \eqref{eq:Fb}. When $r=0$ and $b=\vac$, we let $F_\vac(u\o v)(z)=\tilde{\varphi}(u\o v)$. For $b=a(-1)\vac$, let
				\[
				F_{a(-1)\vac}(u\o v)(z):=
				\sum_{j\ge0}\frac{\tilde\varphi(a(j)u\otimes v)}{(z-1)^{j+1}}+
				\sum_{j\ge0}\frac{\tilde\varphi(u\otimes a(j)v)}{z^{j+1}}.
				\]
				
				In view of \eqref{eq:finalodotrel}, $\tilde{\varphi}$ is defined by the $\odot$ tensor product relation \eqref{eq:finalodotrel} with particular choices $(s,t)=(1,r)$ and $(s,t)=(r+1,0)$. Therefore, it is easy to show by induction that $\tilde\varphi$ preserves the relations \eqref{eq:finalodotrel}. In particular, $\tilde\varphi$ induces a well-defined linear map
				\[
				\varphi: V^k(\g,\la)\odot V^k(\g,\mu)\to L(\la)\o_\C L(\mu)
				\]
				which is clearly the inverse of the surjective map \eqref{eq:cansurj}.
				This shows \eqref{eq:genericlevelodot} is an isomorphism. Furthermore, it follows from \eqref{eq:Aactiononodot} that \eqref{eq:genericlevelodot} is an isomorphism of $\g$-modules, where the $\g$-act on $L(\la)\o_\C L(\mu)$ diagonally.
			\end{proof}
			
			\subsection{\texorpdfstring{The braiding on $\Mod_\fd(U(\g))$ and Drinfeld's $R$-matrix}{The braiding isomorphism on Mod\_fd(U(g))}}
For $\la\in P_+$, let $h_\la=(\la|\la+2\rho)/2\kappa$ be the conformal weight of $V^k(\g,\la)$. The second proof of \eqref{eq:affine-omega-action} only uses the relations \eqref{eq:actionfogeneralspann} and the Segal-Sugawara construction. Hence the same argument shows that the action of $[\om]\in A(V)\cong U(\g)$ on $L(\la)\o_\C L(\mu)\cong L(\la)\bt L(\mu)$ \eqref{eq:genericlevelodot} is given by
			\[
			[\om].(u\o v)=(h_\la+h_\mu)u\o v+\frac{1}{\kappa}\Om_{\la,\mu}(u\o v),\quad u\in L(\la),\ v\in L(\mu).
			\]
			Since $\Mod_\fd(U(\g))$ is semisimple and the canonical map \eqref{eq:cansurj} is natural, the isomorphism \eqref{eq:genericlevelodot} holds for all objects $S,T$ of $\Mod_\fd(U(\g))$, i.e., $S\bt T=S\o_\C T$ with the diagonal action of $U(\g)$. 
	
Then by \eqref{eq:braiding-formula}, the same calculation as in the proof of Theorem~\ref{thm:VOA-Drinfeld-braiding}, and the naturality of $\R$, we have the following Corollary.
			\begin{coro}\label{coro:generic-braiding}
				Let $V=V^k(\g)$ be as above. The category $\Mod_{\mathrm{fd}}(U(\g))$ is a braided tensor category, with fusion tensor product given by
				\[	S \bt T=S\o_\C T,	\]
				on which $U(\g)$ acts diagonally, with the associativity and unit constraints \eqref{eq:transported-constraints}, and the braiding isomorphism is given by
				\begin{equation}\label{eq:generic-braiding}
					\R_{S,T}: S\o_\C T\xrightarrow{\simeq} T\o_\C S,\quad \R_{S,T}(u\o v)=\left(\tau\circ \exp\left(\frac{\pi i }{\kappa} \Om_{S,T}\right)\right)(u\o v)=q^{\Om_{T,S}}(v\o u),
				\end{equation}
				where $q=e^{\pi i/\kappa}$, $u\in S$, $v\in T$, and $\Om_{S,T}$ is the action of $\Om$ \eqref{eq:CandOmega} on $S\o_\C T$. Moreover, $\Phi^\L:\Mod_{\mathrm{fd}}(U(\g))\to KL^k(\g)$ is a braided tensor equivalence.
			\end{coro}
			
In this case, there is no quotient in \eqref{eq:generic-braiding} since 	$S \bt T=S\o_\C T$, and the braiding isomorphism is given by an $R$-matrix. Indeed, let
	\begin{equation}\label{eq:Rmatrix}
\mathscr R=\exp\left(\frac{\pi i}{\kappa}\Om\right)=\sum_{n\geq 0}\frac{1}{n!}\left(\frac{\pi i}{\kappa}\right)^n\Om^n\in U(\g)\hat{\otimes} U(\g),	
\end{equation}		
			which converges on $S\o_\C T$ for all $S,T\in\Mod_\fd(U(\g))$, since $S\o_\C T$ is finite-dimensional. Then \eqref{eq:generic-braiding} can be written as $$\R_{S,T}=\tau\circ\mathscr R|_{S\o_\C T},\quad S,T\in \Mod_\fd(U(\g)).$$ 
			
\subsubsection{Relations with Drinfeld's $R$-matrix}			
 $\mathscr R$ \eqref{eq:Rmatrix} is a specialization of Drinfeld's $R$-matrix \cite{Dri90}. We first recall some notation and constructions in \cite[Section 19]{KL3}. Let $\varpi$ be a formal variable, and let $\A=\C[\![\varpi]\!]$. Let $(a_{ij})_{i,j\in I}$ be the Cartan matrix of $\g$, and choose integers $\delta_i>0$ such that $\delta_ia_{ij}=\delta_ja_{ji}$ for all $i,j\in I$. Under our normalization $(\theta|\theta)=2$, we have
		\[
		\delta_i=\frac{d}{2}(\alpha_i|\alpha_i),\quad \mathrm{where}\quad d=\max_{i\in I}\delta_i.
		\]
		Let $\g_\A=\A\o_\C\g$, and let $\mathfrak U=\A\o_\C U(\g)$ be its universal enveloping algebra over $\A$, with the coproduct $\Delta:\mathfrak U\ra\mathfrak U\o_\A\mathfrak U$ given by $\Delta(x)=x\o 1+1\o x$ for $x\in\g_\A$. The Casimir element $C$ and the element $\Om$ in \eqref{eq:CandOmega} can be viewed as elements $C_\A\in\mathfrak U$ and $\Om_\A\in\mathfrak U\o_\A\mathfrak U$. The following element is introduced in \cite[Section 19.2]{KL3}:
		\begin{equation}\label{eq:t}
			t=\frac{d}{2}\left(\Delta(C_\A)-1\o C_\A-C_\A\o 1\right)=d\,\Om_\A.
		\end{equation}
		It satisfies $\Delta(x)t=t\Delta(x)$ in $\mathfrak U\o_\A\mathfrak U$ for all $x\in\g_\A$.
		
		Let $\mathscr D$ be the category of $\mathfrak U$-modules which are free of finite rank as $\A$-modules. It is referred to as the {\em Drinfeld category} in \cite{KL3}. In particular, $\A\o_\C S$ is an object of $\mathscr D$ for any $S\in\Mod_\fd(U(\g))$. For $V_1,V_2\in\mathscr D$, $V_1\o_\A V_2$ is a $\mathfrak U$-module via $\Delta$, and the series
		\[
		e^{\pi i\varpi t}=\sum_{n\geq 0}\varpi^n\frac{(\pi i t)^n}{n!}
		\]
		converges $\varpi$-adically on $V_1\o_\A V_2$. Following \cite[Section 19.12]{KL3}, the commutativity isomorphism is defined by
		\begin{equation}\label{eq:com}
			C_{V_1,V_2}: V_1\o_\A V_2\xrightarrow{\simeq} V_2\o_\A V_1,\quad x_1\o x_2\mapsto e^{\pi i\varpi t}(x_2\o x_1).
		\end{equation}
		Since $t$ commutes with $\Delta(\g_\A)$, $C_{V_1,V_2}$ is a $\mathfrak U$-module isomorphism.
		
		\begin{lm}\label{lm:Drinfeld-category}\cite{Dri90,KL3}
			The category $\mathscr D$ carries an $\A$-linear braided monoidal structure. Its tensor product is $V_1\o_\A V_2$, with the $\mathfrak U$-action induced by $\Delta$, and its commutativity isomorphism is given by \eqref{eq:com}. The associator is defined by comparing asymptotically normalized solutions of the KZ equation at $0$ and $1$.
			
		Moreover, $\mathscr D$ is equivalent, as an $\A$-linear braided monoidal category, to the category of type-$1$ modules over the formal Drinfeld--Jimbo quantum group $U_q(\g)$ that are free of finite rank over $\A$, where $q=e^{\pi i\varpi}$.
		\end{lm}
		
		Now we compare $C_{V_1,V_2}$ \eqref{eq:com} with our braiding  $\R_{S,T}$ \eqref{eq:generic-braiding}. For finite-dimensional $U(\g)$-modules $S,T$, we let $V_1=\A\o_\C S$ and $V_2=\A\o_\C T$. Then $V_1\o_\A V_2=\A\o_\C(S\o_\C T)$, on which $\Om_\A$ acts by $\Id_\A\o\Om_{S,T}$. For any $z\in\C$, define
		\begin{equation}\label{eq:Cs}
			\mathcal C_{S,T}(z)=\tau\circ\exp\left(z\cdot\pi i\,(d\,\Om_{S,T})\right): S\o_\C T\xrightarrow{\simeq} T\o_\C S.
		\end{equation}
		Since $S\o_\C T$ is finite-dimensional, this gives a family of well-defined $U(\g)$-module isomorphisms, and the matrix coefficients of $\mathcal C_{S,T}(z)$ with respect to any bases of $S\o_\C T$ and $T\o_\C S$ are holomorphic functions of $z$. Let $\mathrm{Hol}(\C)$ be the ring of holomorphic functions on $\C$, and define the formal Taylor expansion
		\[
	\iota_0:\mathrm{Hol}(\C)\ra\A=\C[\![\varpi]\!],\quad f\mapsto\sum_{n\geq 0}\frac{f^{(n)}(0)}{n!}\varpi^n.
		\]
		It is an injective ring homomorphism. Applying $\iota_0$ to the matrix coefficients of $\mathcal C_{S,T}(z)$, we have
		\begin{equation}\label{eq:TayCV1V2}
\iota_0(\mathcal C_{S,T})=\tau_\A\circ\exp\left(\pi i\varpi\,(d\,\Om_\A)\right)=C_{V_1,V_2},
		\end{equation}
		in view of \eqref{eq:t}, \eqref{eq:com}, and $\tau\circ\Om_{S,T}=\Om_{T,S}\circ\tau$, where $\tau_\A=\Id_\A\o\tau$. On the other hand, let $z=1/(d\kappa)$, then by \eqref{eq:generic-braiding} we have
		\[
		\mathcal C_{S,T}\left(\frac{1}{d\kappa}\right)=\tau\circ\exp\left(\frac{\pi i}{\kappa}\Om_{S,T}\right)=\R_{S,T}.
		\]
		Therefore, we have the following comparison between our braiding isomorphism and Drinfeld's commutativity isomorphism.
		\begin{prop}\label{prop:generic-Drinfeld}
			Let $S,T\in\Mod_\fd(U(\g))$, and let $V_1=\A\o_\C S$ and $V_2=\A\o_\C T$. The family of $U(\g)$-module isomorphisms $\mathcal C_{S,T}(z):S\o_\C T\ra T\o_\C S$ in \eqref{eq:Cs} is holomorphic in $z\in\C$, its formal Taylor expansion at $z=0$ is Drinfeld's commutativity isomorphism $C_{V_1,V_2}$ \eqref{eq:com}, and its value at $z=1/(d\kappa)$ is the braiding isomorphism \eqref{eq:generic-braiding}:
			\[
			\mathrm{Tay}_0(\mathcal C_{S,T})=C_{V_1,V_2},\qquad \mathcal C_{S,T}\left(\frac{1}{d\kappa}\right)=\R_{S,T}.
			\]
	In other words, with respect to any bases of $S\o_\C T$ and $T\o_\C S$, the matrix coefficients of $C_{V_1,V_2}$ are power series in $\varpi$ with infinite radius of convergence, and their values at $\varpi=1/(d\kappa)$ give $\R_{S,T}$. In particular, $\mathscr R$ \eqref{eq:Rmatrix} is Drinfeld's $R$-matrix $e^{\pi i\varpi t}$ at $\varpi=1/(d\kappa)$.
		\end{prop}
		
	\begin{remark}\label{rem:generic-associativity}
		The associators $\al$ in the category $\Mod_\fd(U(\g))$ in Corollary~\ref{coro:generic-braiding} are transported from $KL^k(\g)$ through $\Phi^\L$, and $\R$ satisfies the hexagon identities with respect to them by Theorem~\ref{thm:braided-tensor-structure}. Under \eqref{eq:genericlevelodot}, they are natural $U(\g)$-module automorphisms of $S\o_\C T\o_\C U$ for $S,T,U\in\Mod_{\mathrm{fd}}(U(\g))$, but they are not the trivial ones in general. Indeed, for $\g=\sl_2$ and $S=T=U=L(\om_1)$, a direct computation shows that $\R$ does not satisfy the hexagon identities with respect to the trivial associativity constraint unless $1/\kappa\in 2\Z$. This is consistent with Lemma~\ref{lm:Drinfeld-category}, where the associativity isomorphism of $\mathscr D$ is defined by the KZ equation.
	\end{remark}
			

			

		\end{document}